\documentclass[arxiv]{myarticle}

\initbibli

\newcommand{\ann}{\mathrm{an}}

\newcommand{\sprolop}{\overline{\sigma}}
\NewSymbol{sprol}{\sprolop}{difference prolongation map}
\newglossaryentry{sprolneq}{name={\ensuremath{\sprolN[\neq j](c)}},description={c\ldots \sigma^{j-1}(c)\sigma^{j}(c)\ldots \sigma^{n}(c)},sort={\nabla\neq}}
\newcommand{\sprolneq}[2]{\glsonlyfirst{sprolneq}{\sprolN[\neq#2](#1)}}
\newcommand{\sprolto}[3]{\sprolneq{#1}{#3},#2_{#3}}
\newglossaryentry{sprolleq}{name={\ensuremath{\sprolN[\leq j](c)}},description={$c\ldots \sigma^{j}(c)$},sort={\nabla\leq}}
\newcommand{\sprolleq}[2]{\glsonlyfirst{sprolleq}{\sprolN[\leq#2](#1)}}

\NewSymbol{LAQs}{\LL}[\AnnN,\QN,\sigma]{language of analytic fields with an automorphism}
\NewSymbol{TAs}{\rT}[\AnnN,\sigma]{theory of analytic fields with an automorphism}
\NewSymbol{TAsH}{\rT}[\AnnN,\sigma-\Henop]{theory of analytic $\sigma$-Henselian fields with an automorphism}
\newglossaryentry{genans}{name={\ensuremath{\langle\_\rangle_{\sigma}}},description={generated $\LAQs$-structure},sort={<>sigma}}
\NewDocumentCommand{\genans}{om}{\glsonlyfirst{genans}{\IfNoValueTF{#1}{}{#1}\langle#2\rangle_{\sigma}}}
\NewSymbol{LAsac}{\LL}[\AnnN,\QN,\sigma]<\acN>{angular component language of analytic fields with an automorphism}
\NewSymbol{TAsHac}{\rT}[\AnnN,\sigma-\Henop]<\acN>{theory of analytic $\sigma$-Henselian fields with an automorphism and angular components}
\NewSymbol{LAsacfr}{\LL}[\AnnN,\QN,\sigma]<\acN,\fr>{angular component difference analytic language in finitely ramified mixed characteristic}
\newglossaryentry{TAsHacfr}{name={\ensuremath{\TAsHN[p]<\acN,e-\fr>}},description={theory of analytic $\sigma$-Henselian fields with an automorphism and angular components in finitely ramified mixed characteristic},sort={TAsHacefr}}
\NewDocumentCommand{\TAsHacfr}{mm}{\glsonlyfirst{TAsHacfr}{\TAsHN[#2]<\acN,#1-\fr>}}

\newcommand{\Coarop}{\mathfrak{C}}
\NewSymbol{Coar}{\Coarop}<\infty>{Coarsening functor}[C\infty]
\newcommand{\UCoarop}{\mathfrak{UC}}
\NewSymbol{UCoar}{\UCoarop}<\infty>{Uncoarsening functor}[UC\infty]
\newglossaryentry{infty}{name={\ensuremath{\__{\infty}}},description={coarsening in mixed characteristic},sort={\infty}}
\newcommand{\valinf}{\glsonlyfirst{infty}{\val_{\infty}}}
\newcommand{\resinf}{\glshyperlink[{\resN[\infty]}]{infty}}
\newcommand{\Valinf}{\glshyperlink[\Val_{\infty}]{infty}}
\newcommand{\Midinf}{\glshyperlink[\Mid_{\infty}]{infty}}
\newcommand{\ltinf}{\glshyperlink[{\ltN[\infty]}]{infty}}
\newcommand{\ltfinf}{\glshyperlink[{\ltf[\infty]}]{infty}}
\newglossaryentry{Delta}{name={\ensuremath{\_^{\Delta}}},description={coarsening},sort={Delta}}
\newcommand{\valC}[1]{\glsonlyfirst{Delta}{\val^{#1}}}
\newcommand{\ValC}[1]{\glshyperlink[\Val^{#1}]{Delta}}
\newcommand{\MidC}[1]{\glshyperlink[\Mid^{#1}]{Delta}}
\NewDocumentCommand{\resC}{om}{\glshyperlink[{\resN[#1]<#2>}]{Delta}}
\NewDocumentCommand{\resfC}{om}{\glshyperlink[{\resfN[#1]<#2>}]{Delta}}
\newcommand{\resvalC}[1]{\glshyperlink[\widetilde{\val}^{#1}]{Delta}}
\newcommand{\resValC}[1]{\glshyperlink[\widetilde{\Val}^{#1}]{Delta}}
\NewDocumentCommand{\resresfC}{om}{\glshyperlink[{\widetilde{\resfN}_{#1}^{#2}}]{Delta}}
\NewDocumentCommand{\ltfC}{om}{\glshyperlink[{\ltfN[#1]<#2>}]{Delta}}
\NewDocumentCommand{\ltC}{om}{\glshyperlink[{\ltN[#1]<#2>}]{Delta}}
\NewDocumentCommand{\valltC}{om}{\glshyperlink[{\vallt[#1]^{#2}}]{Delta}}

\newcommand{\diffop}{\mathrm{d}}
\newglossaryentry{diff}{name={\ensuremath{\diffop f_{\uple{x}}}},description={differential of $f$ at $\uple{x}$},sort={dfx}}
\newglossaryentry{diffpart}{name={\ensuremath{\partial f/\partial x_{i}(\uple{x})}},description={partial differential of $f$ at $\uple{x}$ along x_{i}},sort={dfdxix}}
\NewDocumentCommand{\diff}{omm}{%
\IfNoValueTF{#1}{\glsonlyfirst{diff}{\diffop{#2}_{#3}}}%
{\glsonlyfirst{diffpart}{\partial #2/\partial x_{#1}(#3)}}}

\newcommand{\Full}{\mathfrak{F}}
\NewSymbol{valres}{\val}[\res]{residual valuation}
\boolfalse{valresGlo}
\NewSymbol{reslt}{\resf}[\lt]{residual leading term}
\boolfalse{resltGlo}
\newcommand{\injresop}{\mathrm{i}}
\NewSymbol{injres}{\injresop}{}
\boolfalse{injresGlo}
\RenewDocumentCommand{\sec}{o}{\secop\IfNoValueTF{#1}{}{_{#1}}}
\newcommand{\secinjop}{\mathrm{t}}
\NewSymbol{secinj}{\secinjop}{}
\boolfalse{secinjGlo}
\NewSymbol{secinjres}{\secinjop}[\res]{}
\boolfalse{secinjresGlo}
\NewSymbol{secres}{\sec}[\res]{}
\boolfalse{secresGlo}
\newcommand{\ltsece}[1]{#1^{\secop,e}}
\newcommand{\Lltsece}{\ltsece{\LL}}
\newcommand{\Tace}{\rT^{\acN,e}}
\newcommand{\Lace}{\LL^{\acN,e}}
\newcommand{\Lltinf}{\LL^{\lt[\infty]}}
\NewSymbol{Linf}{\LL}<\infty>{}
\boolfalse{LinfGlo}
\NewSymbol{Tinf}{\rT}<\infty>{}
\boolfalse{TinfGlo}
\newcommand{\Termop}{\mathcal{T}}
\NewDocumentCommand{\Term}{om}{\Termop\IfNoValueTF{#1}{}{_{\!\!#1}}(#2)}

\boolfalse{valFirst}
\boolfalse{WittfFirst}
\boolfalse{WittrFirst}
\boolfalse{substrFirst}
\boolfalse{StrFirst}
\boolfalse{ACFFirst}
\boolfalse{ACVFFirst}
\boolfalse{KFirst}
\boolfalse{ValgpFirst}
\boolfalse{ltFirst}
\boolfalse{LltplusFirst}
\boolfalse{acFirst}
\boolfalse{fixFirst}
\boolfalse{WittLAQsFirst}
\boolfalse{TAsHFirst}
\boolfalse{genanFirst}

\renewcommand{\resN}{\ResN}
\renewcommand{\res}{\Res}
\renewcommand{\vallt}{\val}

\author{Silvain Rideau\thanks{Partially supported by ANR MODIG (ANR-09-BLAN-0047) and ValCoMo (ANR-13-BS01-0006)}
}

\title{Some properties of analytic difference valued fields}
\date{\today}

\begin{document}
\maketitle

\begin{abstract}
We prove field quantifier elimination for valued fields endowed with both an analytic structure and an automorphism that are $\sigma$-Henselian. From this result we can deduce various Ax-Kochen-Ersov type results with respect to completeness and the $\NIP$ property. The main example we are interested in is the field of Witt vectors on the algebraic closure of $\Ff_{p}$ endowed with its natural analytic structure and the lifting of the Frobenius. It turns out we can give a (reasonable) axiomatization of its first order theory and that this theory is $\NIP$.
\end{abstract}

\medskip


\phantomsection
\addcontentsline{toc}{section}{Introduction}
\section*{Introduction}
\markright{Introduction}

Since the work of Ax, Kochen and Eršov on valued fields (e.g. \cite{AxKo}) and their proof that the theory of a Henselian valued field is essentially controlled (in equicharacteristic zero) by the theory of the residue field and the value group, model theory of Henselian valued fields has been very active and productive. Among later developments one may note Macintyre's result in \cite{MacQp} of elimination of quantifiers for $p$-adic fields and the proof by Pas of valued fields quantifier elimination for equicharacteristic zero Henselian fields with angular components in \cite{PasEQ}, which implies the Ax-Kochen-Eršov principle. Another notable result is the one by Basarab and Kuhlmann (see \cite{Bas,BK,Kuh}) of valued field quantifier elimination for Henselian valued fields with amc-congruences, a language that does not make the class of definable sets grow (whereas angular components might). Another result in the Ax-Kochen-Eršov spirit is the proof by Delon in \cite{DelonNIP} --- extended by Bélair in \cite{BelairNIP} --- that Henselian valued fields do not have the independence property if and only if their residue field does not have it (their value group never has the independence property by \cite{GuSchNIP}.)

But model theorists have not limited themselves to giving an increasingly refined description of the model theory of Henselian valued fields. There have also been attempts at extending those results to valued fields with more structure. The two most notable enrichments that have been studied are, on the one hand, analytic structures as initiated by \cite{DvdD} and studied thereafter by a great number of people (among many others \cite{vdDAnAKE,vdDHM,LRSep,LRUnif,CLR,CluLipAnn}) and, on the other hand, $\mathcal{D}$-structures (a generalization of both difference and differential structures), first for derivations and certain isometries in \cite{ScaDvalF} but also for greater classes of isometries in \cite{ScaRelFrob,BMS,AzgvdD} and then for automorphisms that might not be isometries \cite{AzgOmeInc,Pal,HruFrob,GuzPoi,DurOna}. The model theory of valued differential fields is also quite central to the model theoretic study of transseries (see for example \cite{ADHtrans}) but the techniques and results in this last field seem quite orthogonal to those in other references given above and to our work here.

The goal of the present paper is to study valued fields with \emph{both} an analytic structure and an automorphism. The main result of this paper is Theorem\,\ref{thm:EQ TAsH}, which states that $\sigma$-Henselian (cf. Definition\,\ref{def:sigma-Hensel}) valued fields with analytic structure and any automorphism $\sigma$ eliminate field quantifiers resplendently in the leading term language (cf. Definition\,\ref{def:Llt}). We then deduce various Ax-Kochen-Eršov type results for analytic difference valued fields (both with respect to the theory and the independence property). We also try to give a systematic and comprehensive approach to quantifier elimination in (enriched) valued fields through some more abstract considerations (mainly found in the appendix).

In \cite{ScaADF}, Scanlon already attempted to study analytic difference valued fields in the case of an isometry, but the definition of $\sigma$-Henselianity given there is too weak to actually work, although some incorrect computations hide this fact. The axiomatization and all the proofs had to be redone entirely but, as stated earlier, this paper does not only contain a corrected version of the results in \cite{ScaADF}, it also generalizes these results from the isometric case to the case of any valued field automorphism.

Some ideas from \cite{ScaADF} could be salvaged though, among them the fact that Weierstrass preparation (see Definition\,\ref{def:Weierstrass prep}) allows us to be close enough to the polynomial case to adapt the proofs from the purely valued difference setting. Nevertheless this adaptation is not as straightforward as one would hope, essentially because Weierstrass preparation only holds in one variable, but one variable in the difference world actually gives rise to many variables in the non difference world. The main ingredient to overcome this obstacle is a careful study of differentiability of terms in many variables (see Definition\,\ref{def:lin approx def}) that allows us to give a new definition of $\sigma$-Henselianity in \ref{def:sigma-Hensel}. These techniques can probably be used to prove results in greater generality, for example: valued fields with both analytic structure and $\mathcal{D}$-structure.

Our interest in the model theory of valued fields with both analytic structure and difference structure is not simply a wish to see Ax-Kochen-Eršov type results extended to more and more complicated structures and in particular to the combination of two structures where things are known to work well. It is also motivated by possible applications to diophantine and number theoretic problems. In particular, it is the right model-theoretic setting in which to understand Buium's $p$-differential geometry. More precisely any $p$-differential function over $\Wittf(\alg{\Ff_{p}})$ is definable in $\Wittf(\alg{\Ff_{p}})$ equipped with the lifting of the Frobenius and symbols for all $p$-adic analytic functions $\sum a_{I}x^{I}$ where $\val(a_{I})\to\infty$ as $\card{I}\to\infty$. See \cite[Section\,4]{ScaADF} for an example of how a good model theoretic understanding of this structure can help to show uniformity of certain diophantine results.

The organization of this text is as follows. Section\,\ref{sec:lang} is a description of valued field languages, with either angular components or $\lt$-structure. In Section\,\ref{sec:coar}, we show that it is possible to transfer elimination of quantifier results from equicharacteristic zero to mixed characteristic (using the theoretical framework of Appendix\,\ref{sec:cat}). Sections\,\ref{sec:ann} and \ref{sec:isom} describe the class of analytic difference valued fields we will be studying. Section\,\ref{sec:alg} is concerned with purely analytical matters, it describes the link between analytic $1$-types and the underlying algebraic $1$-type. In Section\,\ref{sec:EQ} we prove the main result of this paper, Theorem\,\ref{thm:EQ TAsH}, a field quantifier elimination result for $\sigma$-Henselian analytic difference valued fields. We also prove an Ax-Kochen-Eršov principle for these fields. Finally Section\,\ref{sec:NIP} shows how this quantifier elimination result also allows us to give conditions on the residue field and the value group for such fields to have (or not have) the independence property. The appendix contains an account of the more abstract model theory at work in the rest of the paper to help smooth out the arguments. Appendix\,\ref{sec:cat}, in particular, sets up a general setting for transfer of elimination of quantifier results.

I would like to thank Élisabeth Bouscaren and Tom Scanlon for our numerous discussions. Without them none of the mathematics presented here would be understandable, correct or even exist. I also want to thank Raf Cluckers for having so readily answered all my questions about analytic structures as I was discovering them. Finally, I would like to thank Koushik Pal for taking the time to discuss the non-isometric case with me. Our discussions led to the generalization of the proofs to the non-isometric case.

\section{Languages of valued fields}\label{sec:lang}

We will be considering valued fields of \emph{characteristic zero}. They will mainly be considered in two kinds of languages. On the one hand, the language with leading terms, also known in the work of Basarab and Kuhlmann (cf. \cite{Bas,BK,Kuh}) as amc-congruences and in later work as $\lt$-sorts (e.g. \cite{HruKaz}) and on the other hand the language with angular components, also known as the Denef-Pas language.

\begin{definition}[Llt]($\Llt$, the leading term language)
The language $\Llt$ has the following sorts: a sort $\K!$ and a family of sorts $(\lt![n])_{n\in\Nn_{>0}}$. On the sort $\K$, the language consists of the ring language. The language also contains functions $\ltf[n] : \K\to\lt[n]$ for all $n\in\Nn_{>0}$ and $\ltf[m,n] : \lt[n]\to\lt[m]$ for all $m|n$.
\end{definition}

Any valued field can be considered as an $\Llt$-structure by interpreting $\K$ as the field and $\lt[n]$ as  $(\inv*{K}/1+n\Mid) \cup \{0\}$ where $\Mid$ is the maximal ideal of the valuation ring $\Val$. We will write $\inv*{lt}[n]$ for $(\inv*{K}/1+n\Mid) = \lt[n]\sminus \{0\}$. Then $\ltf[n]$ is interpreted as the canonical surjection $\inv*{K}\to\inv*{lt}[n]$ and it sends $0$ to $0$; $\ltf[n,m]$ is interpreted likewise. Note that the sorts $\lt[n]$ have a rich structure given by the following commutative diagram (where $\Def{\res[n]}{\Val/n\Mid}$, $\Valgp!$ denotes the value group and all the lines are exact):

\[\xymatrix@R=10pt@C=20pt{
1\ar[r]&\inv{\Val}\ar[r]\ar[dd]^{\resf[n]}\ar@/_15pt/[ddd]|!{[ddl];[dd]}\hole_{\resf[m]}&\inv*{K}\ar[dd]^{\ltf[n]}\ar[rd]^{\val!}\ar@/_15pt/[ddd]|!{[ddl];[dd]}\hole_{\ltf[m]}\\
&&&\Valgp\ar[r]&0\\
1\ar[r]&\inv*{res}[n]\ar[r]\ar[d]^{\resf[m,n]}&\inv*{lt}[n]\ar[ru]_{\vallt[n]}\ar[d]^{\ltf[m,n]}\\
1\ar[r]&\inv*{res}[m]\ar[r]&\inv*{lt}[m]\ar@/_15pt/_{\vallt[m]}[uur]
}\]

We will denote the $\Llt$-theory of characteristic zero valued fields by $\Tvf$. If we need to specify the residual characteristic, we will write $\Tvf[0,0]$ or  $\Tvf[0,p]$. Let $\Def{\lt}{\bigcup_{n}\lt[n]}$. These sorts are closed in $\Llt$ (see Definition\,\ref{def:cl sorts}). In order to eliminate $\K$-quantifiers, we will have to add some structure on the $\lt$ sorts.

\begin{definition}
The language $\Lltplus!$ is the enrichment of $\Llt$ with, on each $\lt[n]$, the language of (multiplicative) groups $\{1_{n},\cdot_{n}\}$, a symbol $0_{n}$ and a binary predicate $\Div[n]$, and functions $+_{m,n}:\lt[n]<2>\to\lt[m]$ for all $m|n$.
\end{definition}

The multiplicative structure on $\lt[n]$ is interpreted as its multiplicative (semi-)group structure, i.e. the group structure of $\inv*{lt}[n]$ and $0_{n} \cdot_{n} x = x \cdot_{n} 0_{n} = 0_{n}$. The relation $x\Div[n] y$ is interpreted as $\vallt[n](x) \leq \vallt[n](y)$. For all $x,y\in\K$ such that $\val(x+y) \leq \min\{\val(x),\val(y)\} + \val(n)-\val(m)$, $\ltf[n](x) +_{m,n} \ltf[n](y)$ is interpreted as $\ltf[m](x+y)$ and $0_{n}$ otherwise. This is well defined.

We will denote by $\THen$ the theory of characteristic zero Henselian valued fields in $\Lltplus$.

\begin{remark}
\begin{thm@enum}
\item If $K$ has equicharacteristic zero, then for all $m|n$, $\ltf[m,n]$ is an isomorphism. Hence if we are working in equicharacteristic zero, we will only need to consider $\lt[1]$. In that case we also have that $\res[1]  = \inv*{res}[1]\cup\{0\}\subseteq \inv*{lt}[1]\cup\{0\} = \lt[1]$. The additive structure is also simpler: we only need to consider the $+_{1,1}$ function on $\lt[1]$. It extends the additive structure of $\res[1]$ and makes every fiber of $\vallt[1]$ into an $\res[1]$-vector space of dimension $1$ (if we consider $0_{1}$ to be the zero of every fiber).
\item If $K$ has mixed characteristic $p$, then whenever $m|n$ and $\val(n) = \val(m)$ (i.e. when $p$ does not divide $n/m$) $\ltf[m,n]$ is an isomorphism. In particular for all $n\in\Nn_{>0}$, $\ltf[n,p^{\val(n)}]$ is an isomorphism (where we identify $\val(p)$ and $1$).
\item One could wonder then why consider all the $\lt[n]$ when the only relevant ones are the $\lt[p^{n}]$ in mixed characteristic $p$ and $\lt[1]$ in equicharacteristic zero. The main reason is that we want enough uniformity to be able to talk of $\Tvf$ without specifying the residual characteristic or adding a constant for the characteristic exponent (in particular if one wishes to consider ultraproducts of valued fields with growing residual characteristic, although we will not do so here).
\end{thm@enum}
\end{remark}

The use of this language is mainly motivated by the following result that originates in \cite{Bas,BK}, although the phrasing in terms of resplendence first appears in \cite{ScaPhD}. By resplendent quantifier elimination relative to $\lt$, we mean that quantifiers on the sorts other than those in $\lt$ can be eliminated (namely the field quantifiers here) and that this result is true for any enrichment on the $\lt$-sorts (see Appendix\,\ref{sec:resplEQ} for precise definitions).

\begin{theorem}[EQ THen]
The theory $\THen$ eliminates $\K$-quantifiers resplendently relatively to $\lt$.
\end{theorem}

Later, we will add analytic and difference structures, hence we will consider an enrichment of $\Llt$ by new terms on $\K$ and predicates and terms on $\lt$ (although none on both $\K$ and $\lt$; this is what we call in Appendix\,\ref{sec:resplEQ} an $\lt$-enrichment of a $\K$-term enrichment of $\Llt$). Let $\LL$ be such a language and let $\Sigma_{\lt}$ denote the new sorts coming from the $\lt$-enrichment.

\begin{remark}[decomp form Kterm]
Any quantifier free $\LL$-formula $\phi(\uple{x},\uple{y})$ where $\uple{x}$ are $\K$-variables and $\uple{y}$ are $\lt$-variables, is equivalent modulo $\Tvf$ to a formula of the form $\psi(\ltf[\uple{n}](\uple{u}(\uple{x})),\uple{y})$ where $\psi$ is a quantifier free $\Sortrestr{\LL}{\lt\cup\Sigma_{\lt}}$-formula and $\uple{u}$ are $\Sortrestr{\LL}{\K}$-terms. Indeed the only predicate involving $\K$ is the equality and $t(\uple{x}) = s(\uple{x})$ is equivalent to $\ltf[1](t(\uple{x}) - s(\uple{x})) = 0$. The statement follows immediately.
\end{remark}

Here is an easy corollary that will be very helpful later on to uniformize certain results.

\begin{corollary}[uniformization]
Let $T$ be an $\LL$-theory that eliminates $\K$-quantifiers, $M\models T$, $C\substr! M$  (i.e. $C$ is a substructure of $M$) and $\uple{x}$, $\uple{y}\in\K(M)$ be such that for all $\Sortrestr{\LL}{\K}(C)$-terms $\uple{u}$, and all $n\in\Nn_{>0}$, $\ltf[n](\uple{u}(\uple{x})) = \ltf[n](\uple{u}(\uple{y}))$. Then $\uple{x}$ and $\uple{y}$ have the same $\LL(C)$-type. 
\end{corollary}

\begin{proof}
Let $f : M\to M$ be the identity on $\lt\cup\Sigma_{\lt}(M)$ and send $u(\uple{x})$ to $u(\uple{y})$ for all $\Sortrestr{\LL}{\K}(C)$-term $u$. By Remark\,\ref{rem:decomp form Kterm}, $f$ is a partial $\Morl*{\LL}[\lt]$-isomorphism. But $\K$-quantifiers elimination implies that $f$ is in fact elementary.
\end{proof}

The other kind of valued field language, the one with angular components, essentially boils down to giving oneself a section of the short sequences defining the $\lt[n]$. That statement is made explicit in \ref{prop:equiv Lsec Lac}. 

\begin{definition}($\Lac$, the angular component language)
The language $\Lac$ has the following sorts: $\K$, $\Valgpinf$ and $(\res[n])_{n\in\Nn_{>0}}$. The sorts $\K$ and $\res[n]$ come with the ring language and the sort $\Valgpinf$ comes with the language of ordered (additive) groups and a constant $\infty$. The language also contains a function $\val : \K \to \Valgpinf$, for all $n$, functions $\ac[n] : \K\to\res[n]$, $\resf[n] : \K\to\res[n]$, $\valres[n]:\res[n]\to\Valgpinf$, $\secres[n] : \Valgpinf\to\res[n]$ and for all $m|n$, functions $\resf[m,n] : \res[n]\to \res[m]$ and $\secinjres[m,n] : \res[n]\to \res[m]$.
\end{definition}

As one might guess, the $\res[n]$ are interpreted as the residue rings $\Val/n\Mid$. As with $\lt$, we will write $\Def{\res}{\bigcup_{n}\res[n]}$. The $\resf[n]$ and $\resf[m,n]$ denote the canonical surjections $\Val \to \res[n]$ and $\res[n]\to\res[m]$, extended by zero outside their domains. The function $\ac![n]$ denotes an angular component, i.e a multiplicative homomorphism $\inv*{K}\to \inv*{res}[n]$ which extends the canonical surjection on $\inv{\Val}$ and sends $0$ to $0_{n}$. Moreover, the system of the $\ac[n]$ should be consistent, i.e. $\resf[m,n]\comp \ac[n] = \ac[m]$. The function $\valres[n]$ is interpreted as the function induced by $\val$ on $\res[n]\sminus\{0\}$ and sending $0_{n}$ to $\infty$. The function $\secres[n]$ is defined by $\secres[n](\val(x)) = \resf[n](x)\ac[n](x)^{-1}$ and $\secres[n](\infty) = 0_{n}$. Finally, the function $\secinjres[m,n]$ is defined by $\secinjres[m,n](\resf[n](x)) = \ac[m](x)$ when $\val(x)\leq \val(n)-\val(m)$ and $0_{m}$ otherwise (this is well-defined).

It should be noted that any valued field that is sufficiently saturated can be endowed with angular components (cf. \cite[Corollary\,1.6]{PasExistAC}).

The following pages, up to Definition\,\ref{def:swiss cheese} contain a (very technical) account of how to deduce quantifier elimination results in the angular component language from results in the leading term language. Although they are essential to prove Ax-Kochen-Eršov type results, angular components only appear again in Sections\,\ref{subsec:EQ} and \ref{sec:NIP}, and a reader mostly interested in the broader picture can safely skip those pages.

Let $\Lltsec$ be the enrichment of $\Lltplus$ obtained by adding a sort $\Valgpinf$, equipped with the language of ordered (additive) groups, a family of sorts $\res[n]$, equipped with the ring language, symbols $\vallt[n]:\lt[n] \to\Valgpinf$ for the functions induced by the valuation, symbols $\injres[n] : \res[n]\to \lt[n]$ for the injection of $\inv*{res}[n]\to \lt[n]$ extended by $0$ outside $\inv*{res}[n]$, symbols $\reslt[n] : \lt[n]\to\res[n]$ for the map sending $a(1 + n\Mid)$ to $a + n\Mid$, $\sec[n] : \Valgpinf\to\lt[n]$ for a coherent system of sections of $\vallt[n]$ compatible with the $\ltf[m,n]$ and symbols $\secinj[n] : \lt[n]\to\res[n]$ interpreted as $\secinj[n](x) = \injres[n]^{-1}(x\sec[n](\vallt[n](x))^{-1})$. Let $\Tvfsec$ be the $\Lltsec$-theory of characteristic zero valued fields and $\Tvfac$ the $\Lac$-theory of characteristic zero valued fields.

Let $\Lltsece$ be an $\lt$-enrichment (with potentially new sorts $\Sigma_{\lt}$) of a $\K$-enrichment (with potentially new sorts $\Sigma_{\K}$) of $\Lltsec$ and $T^{e}$ be an $\Lltsece$-theory extending $\Lltsec$. We define $\Lace$ to be the language containing:
\begin{enumerate}
\item $\Lac\cup\Sortrestr{\Lltsece}{\K\cup\Sigma_{\K}}$;
\item The new sorts $\Sigma_{\lt}$;
\item For each new function symbol $f : \prod S_{i} \to \lt[n]$, two functions symbols $f_{\res} : \prod T_{i}\to\res[n]$ and $f_{\Valgp} : \prod T_{i}\to\Valgpinf$ where $T_{i} = \res[m]\times\Valgpinf$ whenever $S_{i} = \lt[m]$ and $T_{i} = S_{i}$ otherwise;  
\item For each new function symbol $f : \prod S_{i} \to S$, where $S\neq \lt[n]$, the same symbol $f$ but with domain $\prod T_{i}$ as above;
\item For each new predicate $R \subseteq \prod S_{i}$, the same symbol $R$ but as a predicate in $\prod T_{i}$ for $T_{i}$ as above.
\end{enumerate}
We also define $\Tace$ to be the theory containing:
\begin{enumerate}
\item $\Tvfac$;
\item For all new function symbol $f$, whenever $f$ or $f_{\res}$ and $f_{\Valgp}$ (depending on the case) is applied to an argument (corresponding to an $\lt[n]$-variable of $f$) outside of $\inv*{res}[n]\times\Valgp\cup\{0,\infty\}$, then $f$ has the same value as if $f$ were applied to $(0,\infty)$ instead;
\item For all new symbol $f$ with image $\lt[n]$, $\im(f_{\res},f_{\Valgp})\subseteq\inv*{res}[n]\times\Valgp\cup(0,\infty)$;
\item For all new predicate $R$, $R$ applied to an argument outside of $\inv*{res}[n]\times\Valgp\cup\{0,\infty\}$ is equivalent to $R$ applied to $(0,\infty)$ instead;
\item The theory $T^{e}$ translated in $\Lace$ as explained in the following proposition.
\end{enumerate}

In the following proposition, $\Str(T)$ denote the category of substructures of models of $T$, i.e. models of $T_{\forall}$. See Appendix\,\ref{sec:cat}, for precise definitions.

\begin{proposition}[equiv Lsec Lac]
There exist functors $F : \Str(\Tace)\to \Str(T^{e})$ and $G : \Str(T^{e})\to\Str(\Tace)$ that respect models, cardinality and elementary submodels and induce an equivalence of categories between $\Str(\Tace)$ and $\Str(T^{e})$. Moreover $G$ sends $\res\cup\Valgpinf$ to $\lt\cup\res\cup\Valgpinf$.
\end{proposition}

\begin{proof}
Let $C$ be an $\Lace$-substructure (inside some $M\models \Tace$), we define $F(C)$ to have the same underlying sets for all sorts common to $\Lace$ and $\Lltsece$ and $\lt[n](F(C)) = (\inv*{res}[n](C)\times(\Valgpinf(C)\sminus\{\infty\}))\cup\{(0_{n},\infty)\}$. All the structure on the sorts common to $\Lltsece$ and $\Lace$ is inherited from $C$. We define $\ltf[n](x) = (\ac[n](x),\val(x))$ and $\ltf[m,n](x,\gamma) = (\resf[m,n](x),\gamma)$. The (semi-)group structure on $\lt[n]$ is the product (semi-)group structure, $0_{n}$ is interpreted as $(0_{n},\infty)$. We set $(x,\gamma)\Div[n](y,\delta)$ to hold if and only if $\gamma \leq \delta$ and we define $(x,\gamma) +_{m,n} (y,\delta)$ as $(\resf[m,n](x),\gamma)$ if $\gamma < \delta$, $(\resf[m,n](y),\delta)$ if $\delta < \gamma$ and $(\secinjres[m,n](x+y),\gamma + \valres[n](x+y))$ if $\delta = \gamma$. The functions $\vallt[n]$ are interpreted as the right projections and the functions $\secinj[n]$ as the left projections. Finally, define $\injres[n](x) = (x,0)$ on $\inv*{res}[n]$ and $\injres[n](x) = (0,\infty)$ otherwise, $\reslt[n](x,\gamma) = x\secres[n](\gamma)$, $\sec[n](\gamma) = (1,\gamma)$ if $\gamma\neq \infty$ and $\sec[n](\infty) = (0,\infty)$. For each function $f : \prod S_{i} \to \lt[n]$ for some $n$, define $\uple{u} : \prod S_{i}\to \prod T_{i}$ to be such that $u_{i}(\uple{x}) = x_{i}$ if $S_{i}\neq\lt[m]$ and $u_{i}(\uple{x}) = (\secinj[m](x_{i}),\vallt[m](x_{i}))$ if $S_{i} = \lt[m]$. Then $f^{F(C)}(\uple{x}) = (f^{C}_{\res}(\uple{u}(\uple{x})),f^{C}_{\Valgp}(\uple{u}(\uple{x})))$. If $f : \prod S_{i}\to S$ where $S\neq \lt[n]$ for any $n$, then define $f^{F(C)}(\uple{x}) = f^{C}(\uple{u}(\uple{x}))$ and finally $F(C)\models R(\uple{x})$ if and only if $C\models R(\uple{u}(\uple{x}))$.

If $f : C_{1}\to C_{2}$ is an $\Lace$-isomorphism, we define $F(f)$ to be $f$ on all sorts common to $\Lace$ and $\Lltsece$ and $F(f)(x,\gamma) = (f(x),f(\gamma))$. It is easy to check that $F(f)$ is an $\Lltsece$-isomorphism.

Let $D$ be an $\Lltsece$-structure (inside some $N\models T^{e}$), define $G(D)$ to be the restriction of $D$ to all $\Lace$-sorts enriched with $\val = \vallt[n]\comp\ltf[1]$, $\resf[n] = \reslt[n]\comp \ltf[n]$, $\ac[n] = \secinj[n]\comp\ltf[n]$. Moreover, for any function $f :\prod S_{i} \to \lt[n]$ for some $n$, let $\uple{v} : \prod T_{i}\to \prod S_{i}$ to be such that $v_{i}(\uple{x}) = x_{i}$ if $S_{i}\neq\lt[m]$ for any $m$ and $v_{i}(\uple{x}) = \injres[m](y_{i})\sec[m](\gamma_{i})$ where $x_{i} = (y_{i},\gamma_{i})$, if $S_{i} = \lt[m]$. Then define $f^{G(D)}_{\res}(\uple{x}) = \secinj[n](f^{D}(\uple{v}(\uple{x})))$ and $f^{G(D)}_{\Valgp}(\uple{x}) = \vallt[n](f^{D}(\uple{v}(\uple{x})))$. If $f :\prod S_{i} \to S$ where $S\neq\lt[n]$ for any $n$, then $f^{G(D)}(\uple{x}) = f^{D}(\uple{v}(\uple{x}))$ and finally $G(D)\models R(\uple{x})$ if and only if $D\models R(\uple{v}(\uple{x}))$. If $f : D_{1}\to D_{2}$ is an $\Lltsece$-isomorphism, it is easy to show that the restriction of $f$ to the $\Lace$-sorts is an $\Lace$-isomorphism.

Now, one can check that for any $\Lltsece$-formula $\phi(\uple{x})$ there exists an $\Lace$-formula $\phi^{\ac,e}(\uple{y})$ such that for any $C\in \Str(\Tace)$ and $\uple{c}\in C$, $C\models\phi(\uple{c})$ if and only if $F(C)\models\phi^{\ac,e}(\uple{u}(\uple{c}))$ where $u$ is as above (for the sorts corresponding to $\uple{x}$). Similarly, to any $\Lace$-formula $\psi(\uple{x})$ we can associate an $\Lltsece$-formula $\ltsece{\psi}(\uple{x})$ such that for any $D\in\Str(T)$ and $d\in D$, $D\models\psi(\uple{d})$ if and only if $G(D)\models\ltsece{\psi}(\uple{d})$. One can also check that for all $\Lltsece$-formula $\phi$, $T\models \ltsece{(\phi^{\ac,e})}(\uple{u}(\uple{x}))\iff \phi(\uple{x})$ and for all $\Lace$-formula $\psi$, $\Tace\models(\ltsece{\psi})^{\ac,e}\iff \psi$. The rest of the proposition follows.
\end{proof}

\begin{remark}
\begin{thm@enum}
\item The functions $\secinjres[m,n]$ are actually not needed, if we Morleyize on $\res\cup\Valgpinf$, as they are definable using only quantification in the $\res[n]$.
\item As with the leading terms structure, in equicharacteristic zero, the angular component structure is a lot simpler. We only need $\val$ and $\ac[1]$ (and none of the $\vallt[n]$, $\secres[n]$ or $\secinjres[m,n]$).
\item\label{rem:mixed char ac} In mixed characteristic with finite ramification (i.e. $\Valgp$ has a smallest positive element $1$ and $\val(p) = k \cdot 1$ for some $k\in\Nn_{>0}$) the structure is also simpler. The functions $\valres[n]$, $\secres[n]$ and $\secinjres[m,n]$ can be redefined (without $\K$-quantifiers) knowing only $\secres[n](1)$. Let $\Lacfr$ be the language $(\Lac\sminus\{\valres[n],\secres[n],\secinjres[m,n]\mid m,n\in\Nn_{>0}\})\cup\{c_{n}\}$ where $c_{n}$ will be interpreted as $\secres[n](1)$ --- i.e. as $\resf[n](x)\ac[n](x)^{-1}$ for $x$ with minimal positive valuation. This is the language in which finitely ramified mixed characteristic fields with angular components are usually considered --- and eliminate field quantifiers.
\end{thm@enum}
\end{remark}

To finish this section let us define balls and Swiss cheeses.

\begin{definition}[swiss cheese](Balls and Swiss cheeses)
Let $(K,v)$ be a valued field, $\gamma\in\val(K)$ and $a\in K$. Write $\Def{\oball{\gamma}{a}}{\{x\in \K(M)\mid \val(x-a) > \gamma\}}$ for the open ball of center $a$ and radius $\gamma$, and $\Def{\cball{\gamma}{a}}{\{x\in \K(M)\mid \val(x-a) \geq \gamma\}}$ for the closed ball of center $a$ and radius $\gamma$.

A Swiss cheese is a set of the form $b\sminus(\bigcup_{i=1,\ldots, n} b_{i})$ where $b$ and the $b_i$ are open or closed balls.
\end{definition}

We allow closed balls to have radius $\infty$ --- i.e. singletons are balls --- and we allow open balls to have radius $-\infty$ --- i.e. $\K$ itself is an open ball.

\begin{definition}($\Ldiv$)
The language $\Ldiv$ has a unique sort $\K$ equipped with the ring language and a binary predicate $\Div$.
\end{definition}

In a valued field $(K,\val)$, the predicate $x\Div y$ will denote $\val(x)\leq\val(y)$. If $C\subseteq K$, we will denote by $\SWC(C)$, the set of all quantifier free $\Ldiv(C)$-definable sets in one variable. Note that all those sets are finite unions of swiss cheeses.

Note that later on, our valued fields may be endowed with more than one valuation. In that case, we will write $\oball[\Val]{\gamma}{a}$ or $\SWC<\Val>(C)$ to specify that we are considering the valuation associated to $\Val$. For a tuple $\uple{a}\in\K$, we will extend the notation for balls by writing $\Def{\oball{\gamma}{\uple{a}}}{\{\uple{b}\mid \val(\uple{b}-\uple{a}) > \gamma\}}$ and $\Def{\cball{\gamma}{\uple{a}}}{\{\uple{b}\mid \val(\uple{b}-\uple{a}) \geq \gamma\}}$ where $\Def{\glslink{valuple}{\val(\uple{a})}}{\min_{i}\{\val(a_{i})\}}$.

\section{Coarsening}\label{sec:coar}

The goal of this section is to provide the necessary tools for the reduction to the equicharacteristic zero case. This is a classical method , which underlies most existing proofs of $K$-quantifier elimination for enriched mixed characteristic Henselian fields. We present it here on its own, as a general transfer principle which we will then be able to invoke directly, in order, hopefully, to make the proofs clearer.

\begin{definition}(Coarsening valuations)
Let $(K,\val)$ be a valued field, $\Delta\subseteq\Valgp(K)$ a convex subgroup and $\pi : \Valgp(K)\to \Valgp(K)/\Delta$ the canonical projection. Let $\Def{\valC{\Delta}}{\pi\comp \val}$, extended to $0$ by $\valC{\Delta}(0) = \infty$.
\end{definition}

\begin{remark}
The valuation $\valC{\Delta}$ is a valuation coarser than $\val$. Its valuation ring is $\Def{\ValC{\Delta}}{\{x\in K\mid \exists\delta\in\Delta,\,\delta < \val(x)\}}\supseteq \Val(K)$ and its maximal ideal is $\Def{\MidC{\Delta}}{\{x\in K\mid \val(x) > \Delta\}}\subseteq\Mid(K)$. Its residue field $\resC[1]{\Delta}$ is in fact a valued field for the valuation $\resvalC{\Delta}$ defined by $\resvalC{\Delta}(x+\MidC{\Delta}) := \val(x)$ for all $x\in\ValC{\Delta}\sminus\MidC{\Delta}$ and $\resvalC{\Delta}(\MidC{\Delta}) = \infty$. Then $\resvalC{\Delta}(\resC[1]{\Delta}) = \Delta^{\infty} = \Delta\cup\{\infty\}$. The valuation ring of $\resC[1]{\Delta}$ is $\Def{\resValC{\Delta}}{\Val/\MidC{\Delta}}$, its maximal ideal is $\Mid/\MidC{\Delta}$ and its residue field is $\res[1]$. Moreover, if $\ltfC[n]{\Delta} : K\to \inv{K}/(1+n\MidC{\Delta})\cup\{0\} =: \ltC[n]{\Delta}$ is the canonical projection, $\ltf[n]$ factorizes through $\ltf[n]<\Delta>$; i.e. there is a function $\pi_{n} : \ltC[n]{\Delta} \to \lt[n]$ such that $\ltf[n] = \pi_{n}\comp\ltf[n]<\Delta>$.
\end{remark}

\[\xymatrix{
\inv{\Val}\mono[r]\ar[d]\ar@/_20pt/[dd]_{\resf[1]}&\inv{(\ValC{\Delta})}\mono[r]\ar[d]^{\resfC[1]{\Delta}}&\inv{\K}\ar[d]^{\ltfC[1]{\Delta}}\\
\inv{(\resValC{\Delta})}\mono[r]\ar[d]^{\resresfC[1]{\Delta}}&\inv{(\resC[1]{\Delta})}\mono[r]\ar[d]^{\resvalC{\Delta}}&\inv{(\lt[1]<\Delta>)}\\
\inv*{res}[1]&\Delta
}\hspace{40pt}
\xymatrix@C=7pt{
&\K\ar[d]|{\ltf[n]}\ar[ddl]_{\ltfC[n]{\Delta}}\ar[drr]^{\val}\ar[ddr]|!{[d];[drr]}\hole_(0.7){\valC{\Delta}}\\
&\lt[n]\ar[rr]|{\vallt[n]}&&\Valgpinf\ar[dl]^{\pi}\\
\ltC[n]{\Delta}\ar[rr]_{\valltC[n]{\Delta}}\ar[ur]|{\pi_{n}}&&(\Valgp/\Delta)^{\infty}
}\]

Before we go on let us explain the link between open balls for the coarsened valuations and open balls for the original valuation.

\begin{proposition}[ball sub coarsen]
Let $(K,\val)$ be a valued field and $\Delta$ a convex subgroup of its valuation group. Let $S$ be an $\Val$-Swiss cheese, $b$ an $\ValC{\Delta}$-ball, $c$, $d\in K$ such that $b = \oball[\ValC{\Delta}]{\valC{\Delta}(d)}{c}$. If $b\subseteq S$, there exists $d'\in K$ such that $\valC{\Delta}(d') = \valC{\Delta}(d)$ and $b\subseteq\oball[\Val]{\val(d')}{c}\subseteq S$.
\end{proposition}

\begin{proof}
Let $(g_{\alpha})$ be a cofinal (ordinal indexed) sequence in $\Delta$. We have $b = \bigcap_{\alpha}\oball[\Val]{\val(dg_{\alpha})}{c}$. Indeed, $\valC{\Delta}(dg_{\alpha}) = \valC{\Delta}(d)$ and hence $b = \oball[\ValC{\Delta}]{\valC{\Delta}(dg_{\alpha})}{c} \subseteq \oball[\Val]{\val(dg_{\alpha})}{c}$. Conversely, if $x\in \bigcap_{\alpha}\oball[\Val]{\val(dg_\alpha)}{c}$, then $\val((x-c)/d) > \val(g_{\alpha})$ for all $\alpha$, hence $(x-c)/d \in\MidC{\Delta}$.

Let $b'$ be any $\Val$-ball, then $b =  \bigcap_{\alpha}\oball[\Val]{\val(dg_{\alpha})}{c} \subseteq b'$ if and only if there exists $\alpha_{0}$ such that $\oball[\Val]{\val(dg_{\alpha_{0}})}{c}\subseteq b'$ and $b\cap b' = \bigcap_{\alpha}\oball[\Val]{\val(dg_{\alpha})}{c} \cap b'= \emptyset$ if and only if there exists $\alpha_{0}$ such that $\oball[\Val]{\val(dg_{\alpha_{0}})}{c}\cap b' = \emptyset$. These statements still hold for Boolean combinations of balls hence there is some $\alpha_{0}$ such that $\oball[\Val]{\val(dg_{\alpha_{0}})}{c}\subseteq S$. 
\end{proof}

When $(K,\val)$ is a mixed characteristic valued field, the coarsened valuation we are interested in is the one associated to $\Delta_{p}$ the convex group generated by $\val(p)$ as $(K,\valC{\Delta_{p}})$ has equicharacteristic zero. We will write $\Def{\valinf}{\valC{\Delta_{p}}}$, $\Def{\resinf}{\resC[1]{\Delta_{p}}}$, $\Def{\Valinf}{\ValC{\Delta_{p}}} = \Val[p^{-1}]$ and $\Def{\Midinf}{\MidC{\Delta_{p}}} = \bigcap_{n\in\Nn} p^{n}\Mid$. As the coarsened field has equicharacteristic zero, all $\ltC[n]{\Delta_{p}}$ are the same and we will write $\Def{\ltinf}{\inv{K}/(1+\Midinf)}\cup\{0\} = \ltC[1]{\Delta_{p}}$.

\begin{remark}[prodef ltinf]
We can --- and we will --- identify $\ltinf$ (canonically) with a subgroup of $\limproj\lt[n]$ and the canonical projection $K\to \ltinf$ then coincides with $\limproj \ltf[n] : K\to\limproj\lt[n]$, in particular, $\ltinf = (\limproj \ltf[n])(K)$. Similarly, $\resValC{\Delta_{p}}$ can be identified with a subring of $\limproj \res[n]$ and $\resinf = \Frac{\resValC{\Delta_{p}}} \subseteq \Frac{\limproj \res[n]} = (\limproj \res[n])[\ltfinf(p)^{-1}]$. The inclusions are equalities if $K$ is $\aleph_{1}$-saturated. In particular, $\limproj \ltf[n]$ is surjective.
\end{remark}

\[\xymatrix{&\K\ar[dl]_{\ltf[m]}\ar[d]|{\ltf[n]}\ar[dr]|{\limproj\ltf[n]}\ar[r]^{\ltfinf}&\ltinf\mono[d]\\
\lt[m]\ar[r]_{\ltf[m,n]}&\lt[n]&\limproj\lt[n]\ar[l]^(0.6){\pi_{n}}\ar@/^20pt/[ll]^{\pi_{m}}}\]

Hence $(K,\valinf)$ is prodefinable --- i.e. a prolimit of definable sets --- in $(K,\val)$  with its $\Llt$-structure.

Let $\LL$ be an $\lt$-enrichment of a $\K$-enrichment of $\Llt$ with new sorts $\Sigma_{\K}$ and $\Sigma_{\lt}$ respectively. Somewhat abusing notation, when writing $\K$ we will mean $\K\cup\Sigma_{\K}$ and when writing $\lt$ we will mean $\bigcup_{n}\lt[n]\cup\Sigma_{\lt}$ (and rely on the context for it to make sense). Let $T\supseteq \Tvf[0,p]$ be an $\LL$-theory. Let $\Lltinf$ be a copy of $\Llt$ (as $\Lltinf$ will only be used in equicharacteristic zero, we will only need its $\lt[1]$, which we will denote $\ltinf$ to avoid confusion with the original $\lt[1]$). Let $\Linf$ be $\Lltinf\cup\Sortrestr{\LL}{\K\cup\Sigma_{K}}\cup\Sortrestr{\LL}{\lt\cup\Sigma_{\lt}}\cup\{\pi_{n}\mid n\in\Nn_{>0}\}$ where $\pi_{n}$ is a function symbol $\ltinf\to\lt[n]$. Let $\Tinf$ be the theory containing:
\begin{itemize}
\item $\Tvf[0,0]<\infty>$, i.e. the theory of equicharacteristic zero valued fields in $\Lltinf$;
\item The translation of $T$ into $\Linf$ by replacing $\ltf[n]$ by $\pi_{n}\comp\ltfinf$.
\end{itemize}

Recall that $\Str(T)$ is the category of substructures of models of $T$ and that whenever  $F : \Str(T_{1})\to\Str(T_{2})$ is a functor and $\kappa$ a cardinal, we denote by $\Str[F,\kappa](T_{2})$ the full subcategory of $\Str(T_{2})$ of structures that embed into some $F(M)$ for $M\models T_{1}$ $\kappa$-saturated. See Appendix\,\ref{sec:cat} for precise definitions.

The main goal of the following proposition is to show that quantifier elimination results in equicharacteristic zero can be transferred to mixed characteristic using result from Appendix\,\ref{sec:cat}.

\begin{proposition}[red zero](Reduction to equicharacteristic zero)
We can define functors $\Coar : \Str(T)\to\Str(\Tinf)$ and $\UCoar:\Str(\Tinf)\to \Str(T)$ which respect cardinality up to $\aleph_{0}$ and induce an equivalence of categories between $\Str(T)$ and $\Str[\Coar,\aleph_{1}](\Tinf)$. Moreover, $\Coar$ respects $\aleph_{1}$-saturated models and $\UCoar$ respects models and elementary submodels and sends $\lt$ to $\lt\cup\ltinf$ (which are closed).
\end{proposition}

\begin{proof}
Let $C\substr M\models T$ be $\LL$-structures. Then $\Coar(C)$ has underlying sets $\K(\Coar(C)) = K(C)$, $\ltinf(\Coar(C)) = \limproj\lt[n](C)$ and $\lt(\Coar(C)) = \lt(C)$, keeping the same structure on $\K$ and $\lt$, defining $\ltfinf$ to be $\limproj{\ltf[n]}$ and $\pi_{n}$ to be the canonical projection $\ltinf\to\lt[n]$. Now, if $f:C_{1}\to C_{2}$ is an $\LL$-embedding, let us write $\Def{f_{\infty}}{\limproj\Sortrestr{f}{\lt[n]}}$. By definition, we have $\pi_{n}\comp f_{\infty} = \Sortrestr{f}{\lt[n]}\comp \pi_{n}$ and by immediate diagrammatic considerations, $\ltfinf\comp \Sortrestr{f}{\K} = f_{\infty} \comp\ltfinf$ and $f_{\infty}$ is injective. Then, let $\Coar(f)$ be $\Sortrestr{f}{\K}\cup f_{\infty}\cup\Sortrestr{f}{\lt}$. As $f$ is an $\LL$-embedding, $\Sortrestr{f}{\K}$ respects the structure on $\K$, $\Sortrestr{f}{\lt}$ respects the structure on $\lt$ and, as we have already seen, $\Coar(f)$ respects $\ltfinf$ and $\pi_{n}$. Hence $\Coar(f)$ is an $\Linf$-embedding.

If $M\models T$ is $\aleph_{1}$-saturated, it follows from Remark\,\ref{rem:prodef ltinf} that $\Coar(M)\models \Tinf$. Beware though that $\Coar(M)$ is never $\aleph_{0}$-saturated because if it were, we would find $x\neq y\in\ltinf(M_{1})$ such that for all $n\in\Nn_{>0}$, $\pi_{n}(x) = \pi_{n}(y)$, contradicting the fact that $\ltinf(M_{1}) = \limproj\lt[n](M_{1})$. Let $C$ be a substructure of $M$. We will denote $i$ the injection. Then $\Coar(i)$ is an embedding of $\Coar(C)$ into $\Coar(M)$ and $\Coar$ is indeed a functor to $\Str(T)$.

The functor $\UCoar$ is defined as the restriction functor that restricts $\Linf$-structures to the sorts $\K$ and $\lt$. It is clear that if $C$ is an $\LL$-structure in some model of $T$, then $\UCoar\comp\Coar(C)$ is trivially isomorphic to $C$. Now if $D$ is in $\Str(\Tinf)$ there will be three leading term structures (and hence valuations) on $D$: the one associated with the $\Lltinf$-structure of $C$ (which is definable), whose valuation ring is $\Val$, the one given by $\ltf[n] = \pi_{n}\comp\ltfinf$ (which is definable), whose valuation ring is $\Valinf$, and the one given by $\limproj \ltf[n]$ (which is only prodefinable), whose valuation ring is $\Val[p^{-1}]$. In general, we have $\Val\subset\Val[p^{-1}]\subset\Valinf$, but if $D = \Coar(C)$ --- or $D$ embeds in some $\Coar(C)$ --- $\Val[p^{-1}] = \Valinf$ and $\limproj \ltf[n](D) = \ltfinf(D)$. Hence, if $C$ embeds in some $\Coar(M)$ then $\Coar\comp\UCoar(C)$ is (naturally) isomorphic to $C$.

Functoriality of all the previous constructions is a (tedious but) easy verification
\end{proof}

\section{Analytic structure}\label{sec:ann}

In \cite{CluLipAnn}, Cluckers and Lipshitz study valued fields with analytic structure. Let us recall some of their results. From now on, $A$ will be a Noetherian ring separated and complete for its $I$-adic topology for some ideal $I$. Let $\genan[A]{\uple{X}}$ be the ring of power series with coefficients in $A$ whose coefficients $I$-adically converge to $0$. Let us also define $\Def{\Ann_{m,n}}{{\genan[A]{\uple{X}}[[\uple{Y}]]}}$ where $\card{\uple{X}} = m$ and $\card{\uple{Y}} = n$ and $\Def{\Ann}{\bigcup_{m,n}\Ann_{m,n}}$. Note that $\Ann$ is a separated Weierstrass system over $(A,I)$ as defined in \cite[Example\,4.4.(1)]{CluLipAnn}.  The main example to keep in mind here will be $\genan[\Wittr{\alg{\Ff_{p}}}]{\uple{X}}[[\uple{Y}]]$ which is a separated Weierstrass system over $(\Wittr{\alg{\Ff_{p}}},p\Wittr{\alg{\Ff_{p}}})$.

\begin{definition}[Q]($\Q$)
We will extensively use a quotient symbol $\Q : \K<2> \to \K$ that is interpreted as $\Q(x,y) = x/y$, when $y\neq 0$ and $\Q(x,0) = 0$.
\end{definition}

\begin{definition}($\ValR$)
Let $\ValR$ be a valuation ring of $K$ included in $\Val$, let $\MidR$ be its maximal ideal and $\valR$ its valuation. We have $\Mid\subseteq\MidR\subseteq\ValR\subseteq\Val$. Also, note that $1 + n\Mid\subseteq 1 +n\MidR \subseteq \inv*{ValR}$ and hence the valuation $\valR$ corresponding to $\ValR$ factors through $\ltf[n]$, i.e. there is some function $f_{n}$ such that $\valR = f_{n}\comp\ltf[n]$. We will also be using a new predicate $x\DivR[1] y$ on $\lt[1]$ interpreted by $f_{1}(x)\leq f_{1}(y)$.
\end{definition}

Note that $\Val$ is the coarsening of $\ValR$ associated to the convex subgroup $\inv{\Val}/\inv*{ValR}$ of $\inv*{K}/\inv*{ValR}$. Note also that $\ValR$ is then definable by the (quantifier free) formula, $\ltf[1](1) \DivR[1]\ltf[1](x)$. In fact the whole leading term structure associated to $\ValR$ is interpretable in $\Llt\cup\{\DivR[1]\}$.

\begin{definition}[ann structure](Fields with separated analytic $\Ann$-structure)
Let $\LA$ be the language $\Lltplus$ enriched with a symbol for each element in $\Ann$ (we will identify the elements in $\Ann$ and the corresponding symbols). For each $E\in\inv*{Ann}[m,n]$ let also $E_{k} : \lt[k]<m+n>\to\lt[k]$ be a new symbol and $\Def{\LAQ}{\LA\cup\{\DivR[1],\Q\}\cup\{E_{k}\mid E \in\inv*{Ann}[m,n],$ $m$, $n$, $k\in\Nn\}}$. The theory $\TA$ of fields with separated analytic $\Ann$-structure consists of the following:
\begin{thm@enum}
\item $\Tvf$;
\item $\Q$ is interpreted as in Definition\,\ref{def:Q};
\item $\DivR[1]$ comes from a valuation subring $\ValR\subseteq\Val$ with fraction field $\K$;
\item Each symbol $f\in\Ann_{m,n}$ is interpreted as a function $\ValR<m>\times\MidR<n> \to \ValR$ (the symbols will be interpreted as $0$ outside $\ValR<m>\times\MidR<n>$);
\item The interpretations $i_{m,n} : \Ann_{m,n}\to \ValR<\ValR<m>\times\MidR<n>>$ are morphisms of the inductive system of rings $\bigcup_{m,n} \Ann_{m,n}$ to $\bigcup_{m,n} \ValR<\ValR<m>\times\MidR<n>>$, where the inclusions are the obvious ones.
\item $i_{0,0}(I)\subseteq \MidR$;
\item $i_{m,n}(X_{i})$ is the $i$-th coordinate function and $i_{m,n}(Y_{j})$ is the $(m+j)$-th coordinate function;
\item For every $E\in\inv*{Ann}[m,n]$, $E_{k}$ is interpreted as the function induced by $E$ on $\lt[k]$ when it is well-defined (we will see shortly, in Corollary\,\ref{cor:ltf unit}, that it is, in fact, always well-defined).
\end{thm@enum}

To specify the characteristic we will write $\TA[0,0]$ or $\TA[0,p]$.
\end{definition}

\begin{remark}
\begin{thm@enum}
\item These axioms imply a certain number of properties that it seems reasonable to require. First (iv) implies that every constant in $A = \Ann_{0,0}$ is interpreted in $\ValR$. By (v) and (vii) polynomials in $\Ann$ are interpreted as polynomials. And (v) implies that any ring equality between functions in $\Ann_{m,n}$ for some $m$ and $n$ are also true in models of $\TA$. Using Weierstrass division (see Proposition\,\ref{prop:WD}) one can also show that compositional identities in $\Ann$ are also true in models of $\TA$.
\item We allow the analytic structure to be over a smaller valuation ring in order to be able to coarsen the valuation while staying in our setting of analytic structures.
\end{thm@enum}
\end{remark}

From now on, we will write $\Def{\genan!{C}}{\gen{\LAQ}{C}}$ and $\Def{\genan[C]{\uple{c}}}{\gen{\LAQ}[C]{\uple{c}}}$ for the $\LAQ$-structures generated by $C$ and $C\uple{c}$ (cf. Definition\,\ref{def:gen}).

We could be working in a larger context here. What we really need in the proof is not that $\Ann$ is a separated Weierstrass system, as in \cite{CluLipAnn}, but the consequences of this fact, namely: Henselianity, (uniform) Weierstrass preparation, differentiability of the new function symbols and extension of the analytic structure to algebraic extensions. One could give an axiomatic treatment along those lines, but to simplify the exposition, we restrict to a more concrete case.

Also note that if $\Ann$ is not countable we may now be working in an uncountable language

Let us now describe all the nice properties of models of $\TA$. 

\begin{proposition}[TA Hens]
Let $M\models\TA$, then $M$ is Henselian.
\end{proposition}

\begin{proof}
If $\Val=\ValR$, this is proved exactly as in \cite[Lemma\,3.3]{LRRig}. The case $\ValR\neq\Val$ follows as coarsening preserves Henselianity.
\end{proof}

\begin{remark}[EQ alg form]
As $\TA$ implies $\THen$, by resplendent elimination of quantifiers in $\THen$ (cf. Theorem\,\ref{thm:EQ THen}), as $\LAQ\sminus(\Ann\cup\{\Q\})$ is an $\lt$-enrichment of $\Lltplus$, any $\LAQ\sminus(\Ann\cup\{\Q\})$-formula is equivalent modulo $\TA$ to a $\K$-quantifier free formula.
\end{remark}

Let us now show that functions from $\Ann$ have nice differential properties.

\begin{definition}
Let $K$ be a valued field and $f : K^{n}\to K$. We say that $f$ is differentiable at $\uple{a}\in K^{n}$ if there exists $\uple{d}\in K^{n}$ and $\xi$ and $\gamma\in\val(\inv*{K})$ such that for all $\uple{\epsilon}\in\oball{\xi}{\uple{a}}$,
\[\val(f(\uple{a}+\uple{\epsilon}) - f(\uple{a}) - \uple{d}\cdot\uple{\epsilon}) \geq 2\val(\uple{\epsilon}) + \gamma.\]
\end{definition}

There is a unique such $\uple{d} = (d_{i})$ and we will denote it $\diff{f}{\uple{a}}$. The $d_{i}$ are usually called the derivatives of $f$ at $\uple{a}$. We will denote them $\diff[i]{f}{\uple{a}}$.

\begin{proposition}[diff Ann]
Let $M\models\TA$ and $f\in\Ann_{m,n}$ for some $m$ and $n$. Then for all $i < m+n$ there is $g_{i}\in\Ann_{m,n}$ such that for all $\uple{a}\in K^{m+n}$, $f$ is differentiable at $\uple{a}$ and $\diff[i]{f}{\uple{a}} = g_{i}(\uple{a})$.
\end{proposition}

\begin{proof}
If $\uple{a}\nin\ValR<m>\times\MidR<n>$ then $f$ is equal to $0$ on $\oball{0}{\uple{a}}$ and the statement is trivial. If not, as $f\in \genan[A]{X}[[\uple{Y}]]$, it has a (formal) Taylor development:
\[f(X_{0}+X,Y_{0}+Y) = f(X_{0},Y_{0}) + \sum_{i}g_{i}(X_{0},Y_{0})X_{i} + \sum_{j}g_{m+j}(X_{0},Y_{0})Y_{j} + h(X_{0},Y_{0},X,Y)\]
where $h$ is a sum of terms each divisible by some quadratic monomial in $\uple{X}$ and $\uple{Y}$. As the interpretation morphisms are ring morphisms, this immediately implies that the interpretation of $f$ is differentiable in $\K(M)$ at $\uple{a}$ and that the derivatives are given by the $g_{i}$.
\end{proof}

\begin{corollary}[ltf unit]
Let $M\models\TA$, $E(\uple{x})\in\Ann_{m,n}$ and $S\subseteq \K(M)^{m+n}$. If, for all $\uple{x}\in S$, $\val(E(\uple{x})) = 0$ then, for all $\uple{x}\in S$, $\ltf[n](E(\uple{x}))$ only depends on $\resf[n](\uple{x})$.
\end{corollary}

In particular if $E\in\inv*{Ann}[m,n]$, then for all $\uple{x}\in\ValR<m>\times\MidR<n>$, $\valR(E(\uple{x})) = 0$ and hence $\val(E(\uple{x})) = 0$ and thus $\ltf[n](E(\uple{x}))$ is a function of $\resf[n](\uple{x})$ which is a function of $\ltf[n](\uple{x})$. Outside of $\ValR<m>\times\MidR<n>$, $\ltf[n](E(\uple{x}))$ is constant equal to $0$ and hence it is also a function of $\ltf[n](\uple{x})$. Hence, as announced earlier, $E$ does induce a well-defined function on $\lt[k]$ for any $k$.

\begin{proof}[Corollary\,\ref{cor:ltf unit}]
Any element with the same $\resf[n]$ residue as $\uple{x}$ is of the form $\uple{x}+n\uple{m}$ for some $\uple{m}\in\Mid$. By Proposition\,\ref{prop:diff Ann}, $E(\uple{x}+n\uple{m}) = E(\uple{x}) + \uple{G}(\uple{x})\cdot(n\uple{m}) + H(\uple{x},n\uple{m})$ where $\uple{G}(\uple{x})\in\ValR\subseteq\Val$ and $\val(H(\uple{x},n\uple{m})) \geq 2\val(n\uple{m}) > \val(n)$, hence $\resf[n](E(\uple{x}+n\uple{m})) = \resf[n](E(\uple{x}))$. As for all $\uple{z}\in S$, $\val(E(\uple{z})) = 0$, $\ltf[n](E(\uple{z})) = \resf[n](E(\uple{z}))$ and we have the expected result.
\end{proof}

Let us now (re)prove a well-known result from papers by Cluckers, Lipshitz and Robinson. There are two main reasons for which to reprove this result. The first reason is that although the proof given here is very close to the classical Denef-van den Dries proof as explained in \cite[Theorem\,4.2]{LRUnif}, the proof there only shows quantifier elimination for algebraically closed fields with analytic structures over $(\Zz,0)$. The second reason is to make sure that $\Val\neq\ValR$ does not interfere.

\begin{theorem}[EQ TA]
$\TA$ eliminates $\K$-quantifiers resplendently.
\end{theorem}

The proof of this theorem will need many definitions and properties that will only be used here and that will be introduced now.

For all $m$, $n\in\Nn$, we define $J_{m,n}$ to be the ideal $\{\sum_{\mu,\nu}a_{\mu,\nu}\uple{X}^{\mu}\uple{Y}^{\nu}\in\Ann_{m,n}\mid a_{\mu,\nu} \in I\}$ of $\Ann_{m,n}$. Most of the time we will only write $J$ and rely on context for the indices. We will also write $\upleneq{X}{n}$ for the tuple $\uple{X}$ without its $n$-th component.

Note that in the definition below and in most of this proof, elements of $\Ann$ will be considered as formal series (and not as their interpretation in some $\LAQ$-structure) and hence infinite sums as the one below do make sense.

\begin{definition}(Regularity)
Let $f\in\Ann_{m_{0},n_{0}}$, $m < m_{0}$, $n < n_{0}$. We say that:
\begin{thm@enum}
\item $f = \sum_{i}a_{i}(\upleneq{X}{m},\uple{Y})X_{m}^{i}$ is regular in $X_{m}$ of degree $d$ if $f$ is congruent to a monic polynomial in $X_{m}$ of degree $d$ modulo $J + (\uple{Y})$;
\item $f = \sum_{i}a_{i}(\uple{X},\upleneq{Y}{n})Y_{n}^{i}$ is regular in $Y_{n}$ of degree $d$ if $f$ is congruent to $Y_{n}^{d}$ modulo $J + (\upleneq{Y}{n}) + (Y_{n}^{d+1})$.
\end{thm@enum}
\end{definition}

If we do not want to specify the degree, we will just say that $f$ is regular in $X_{m}$ (resp. $Y_{n}$).

\begin{proposition}[WD](Weierstrass division and preparation)
Let $f,g\in\Ann_{m_{0},n_{0}}$ and suppose $f$ is regular in $X_{m}$ (resp. in $Y_{n}$) of degree $d$, then there exists unique $q\in\Ann_{m,n}$ and $r\in\genan[A]{{\upleneq{X}{m}}}[[\uple{Y}]][X_{m}]$ (resp. $r\in\genan[A]{\uple{X}}[[\upleneq{Y}{n}]][Y_{n}]$) of degree strictly lower than $d$ such that $g = qf + r$.

Moreover, there exists unique $P\in\genan[A]{{\upleneq{X}{m}}}[[\uple{Y}]][X_{m}]$ (resp. $P\in\genan[A]{\uple{X}}[[\upleneq{Y}{n}]][Y_{n}]$) regular in $X_{m}$ (resp. in $Y_{n}$) of degree at most $d$ and $u\in\inv*{Ann}[m,n]$ such that $f = uP$.
\end{proposition}

\begin{proof}
See \cite[Corollary\,3.3]{LRUnif}.
\end{proof}

We will be ordering multi-indices $\mu$ of the same length by lexicographic order and we write $\card{\mu} = \sum_{i}\mu_{i}$.

\begin{definition}(Preregularity)
Let $f = \sum_{\mu,\nu}f_{\mu,\nu}(\uple{X}_{2},\uple{Y}_{2})\uple{X}_{1}^{\mu}\uple{Y}_{1}^{\nu}\in\Ann_{m_{1}+m_{2},n_{1}+n_{2}}$. We say that $f$ is preregular in $(\uple{X}_{1},\uple{Y}_{1})$ of degree $(\mu_{0},\nu_{0},d)$ when:
\begin{thm@enum}
\item $f_{\mu_{0},\nu_{0}} = 1$;
\item For all $\mu$ and $\nu$ such that $\card{\mu} + \card{\nu} \geq d$, $f_{\mu,\nu}\in J + (\uple{Y}_{2})$;
\item For all $\nu<\nu_{0}$ and for all $\mu$, $f_{\mu,\nu}\in J + (\uple{Y}_{2})$;
\item For all $\mu > \mu_{0}$, $f_{\mu,\nu_{0}}\in J + (\uple{Y}_{2})$.
\end{thm@enum}
\end{definition}

\begin{remark}[prereg imp prereg 0]
Note that if $f = \sum_{\nu}f_{\nu}(\uple{X})\uple{Y}^{\nu}$ is preregular in  $(\uple{X},\uple{Y})$ of degree $(\mu_{0},\nu_{0},d)$ then $f_{\nu_{0}}$ is preregular in $\uple{X}$ of degree $(\mu_{0},0,d)$.
\end{remark}

Let $\Def{T_{d}(\uple{X})}{(X_{0}+X_{m_-1}^{d^{m-1}},\ldots,X_{i} + X_{m-1}^{d^{m-1-i}},\ldots,X_{m-2}+X_{m-1}^{d},X_{m-1})}$ where $m = \card{\uple{X}}$. We call $T_{d}$ a Weierstrass change of variables. Note that Weierstrass changes of variables are bijective.

\begin{proposition}[prereg 0 imp reg]
Let $f= \sum_{\mu,\nu}f_{\mu,\nu}(\uple{X}_{2},\uple{Y}_{2})\uple{X}_{1}^{\mu}\uple{Y}_{1}^{\nu}\in\Ann_{m_{1}+m_{2},n_{1}+n_{2}}$. Then:
\begin{thm@enum}
\item If $f$ is  preregular in $(\uple{X}_{1},\uple{Y}_{1})$ of degree $(\mu_{0},0,d)$ then $f(T_{d}(\uple{X}_{1}),\uple{X}_{2},\uple{Y})$ is regular in $X_{1,m_{1}-1}$.
\item If $f$ is preregular in $(\uple{X}_{1},\uple{Y}_{1})$ of degree $(0,\nu_{0},d)$ then $f(\uple{X},T_{d}(\uple{Y}_{1}),\uple{Y}_{2})$ is regular in $Y_{1,n_{1}-1}$.
\end{thm@enum}
\end{proposition}

\begin{proof}
Let $m=m_{1}-1$ and $n = n_{1}-1$. First assume $f$ is preregular in $(\uple{X}_{1},\uple{Y}_{1})$ of degree $(\mu_{0},0,d)$, then \[f\equiv \sum_{\mu < \mu_{0},\card{\mu}<d}f_{\mu,0}\uple{X}_{1}^{\mu}\mod J + (\uple{Y_{2}}) + (\uple{Y_{1}}).\] Furthermore, $T_{d}(\uple{X}_{1})^{\mu} = (\prod_{i=0}^{m-1}(X_{1,i}+ X_{1,m}^{d^{m-i}})^{\mu_{i}})X_{1,m}^{\mu_{m}}$ is a sum of monomials whose highest degree monomial only contains the variable $X_{1,m}$ and has degree $\sum_{i=0}^{m}d^{m-i}\mu_{i}$. It now suffices to show that this degree is maximal when $\mu = \mu_{0}$, but that is exactly what is shown in the following claim.

\begin{claim}[ord index]
Let $\mu$ and $\nu$ be two multi-indices such that $\mu < \nu$ and $\card{\mu}<d$ then \[\sum_{i=0}^{m}d^{m-i}\mu_{i} < \sum_{i=0}^{m}d^{m-i}\nu_{i}.\]
\end{claim}

\begin{proof}
Let $i_0$ be minimal such that $\mu_{i} < \nu_{i}$. Then for all $j < i_{0}$, $\mu_{j} = \nu_{j}$. Moreover, 
\begin{eqnarray*}
\sum_{i=i_{0}+1}^{m}d^{m-i}\mu_{i}&\leq &\sum_{i=i_{0}+1}^{m}d^{m-i}(d-1)\\
&=& d^{m-i_{0}}-1\\
&<& d^{m-i_{0}},
\end{eqnarray*}
hence 
\begin{eqnarray*}
\sum_{i=0}^{m}d^{m-i}\mu_{i} &<& \sum_{i=0}^{i_{0}-1}d^{m-i}\mu_{i} + d^{m-i_{0}}\mu_{i_{0}} + d^{m-i_{0}}\\
& \leq & \sum_{i=0}^{i_{0}-1}d^{m-i}\mu_{i} + d^{m-i_{0}}\nu_{i_{0}}\\
&\leq & \sum_{i=0}^{m}d^{m-i}\nu_{i}
\end{eqnarray*}
and we have proved our claim.
\end{proof}

Let us now suppose that $f$ is  preregular in $(\uple{X}_{1},\uple{Y}_{1})$ of degree $(0,\nu_{0},d)$. Then \[f\equiv \uple{Y}_{1}^{\nu_{0}} + \sum_{\nu > \nu_{0},\mu}f_{\mu,\nu}\uple{X}_{1}^{\mu}\uple{Y}_{1}^{\nu}\mod J + (\uple{Y_{2}}).\] Now, \[T_{d}(\uple{Y}_{1})^{\nu} = (\prod_{i=0}^{n-1}(Y_{1,i}+ Y_{1,n}^{d^{n-i}})^{\nu_{i}})Y_{1,n}^{\nu_{n}} \equiv Y_{1,n}^{\sum_{i=0}^{n}d^{n-i}\nu_{i}} \mod  J + (\uple{Y_{2}}) + (\upleneq{Y_{1}}{n})\] and we conclude again by Claim\,\ref{claim:ord index}.
\end{proof}

\begin{proposition}[bound deg prereg](Bound on the degree of preregularity)
Let \[f=\sum_{\mu,\nu}f_{\mu,\nu}(\uple{X}_{2},\uple{Y}_{2})\uple{X}_{1}^{\mu}\uple{Y}_{1}^{\nu}\in\Ann[m_{1}+m_{2},n_{1}+n_{2}].\] There exists $d$ such that for any $(\mu,\nu)$ with $\card{\mu}+\card{\nu} < d$, there exists $g_{\mu,\nu}\in\Ann_{m_{1}+m_{3},n_{1}+n_{3}}$ preregular in $(\uple{X}_{1},\uple{Y}_{1})$ of degree $(\mu,\nu,d)$ and $\Sortrestr{\LAQ}{\K}$-terms $\uple{u}_{\mu,\nu}$ and $\uple{s}_{\mu,\nu}$ such that for all $M\models \TA$ and every $\uple{a}\in\ValR(M)$ and $\uple{b}\in\MidR(M)$, if $f(\uple{X}_{1},\uple{a},\uple{Y}_{1},\uple{b})$ is not the zero function, then there exists $(\mu_{0},\nu_{0})$ with $\card{\mu_{0}}+\card{\nu_{0}} \leq d$ and \[f(\uple{X}_{1},\uple{a},\uple{Y}_{1},\uple{b}) = f_{\mu_{0},\nu_{0}}(\uple{a},\uple{b})g_{\mu_{0},\nu_{0}}(\uple{X}_{1},\uple{u}_{\mu_{0},\nu_{0}}(\uple{a},\uple{b}),\uple{Y}_{1},\uple{s}_{\mu_{0},\nu_{0}}(\uple{a},\uple{b})).\]
\end{proposition}

\begin{proof}
This follows, as in \cite[Corollary\,3.8]{LRUnif}, from the strong Noetherian property \cite[Theorem\,4.2.15 and Remark\,4.2.16]{CluLipAnn}.
\end{proof}

As $\ValR$ and $\MidR$ are not sorts in $\LAQ$, there is no way in this language to have a variable that ranges uniquely over one or the other. Hence we introduce the notion of well-formed formula that essentially simulates that sorted behaviour.

A $\K$-quantifier free $\LA$-formula $\phi(\uple{X},\uple{Y},\uple{Z},\uple{R})$ will be said to be well-formed if $\uple{X}$, $\uple{Y}$, $\uple{Z}$ are $\K$-variables and $\uple{R}$ are $\lt$-variables, symbols of functions from $\Ann$ are never applied to anything but variables and $\phi(\uple{X},\uple{Y},\uple{Z},\uple{R})$ implies that $\bigwedge_{i}\valR(X_{i})\geq 0$, $\bigwedge_{i}\valR(Z_{i})\geq 0$ and $\bigwedge_{i}\valR(Y_{i})> 0$. The $(\uple{X},\uple{Y})$-rank of $\phi$ is the tuple $(\card{\uple{X}},\card{\uple{Y}})$. We order ranks lexicographically.

\begin{lemma}[EQ TA ind]
Let $\phi(\uple{X},\uple{Y},\uple{Z},\uple{R})$ be a well-formed $\K$-quantifier free $\LA$-formula. Then there  exists a finite set of well-formed $\K$-quantifier free $\LA$-formulas $\phi_{i}(\uple{X}_{i},\uple{Y}_{i},\uple{Z}_i,\uple{R})$ of $(\uple{X}_{i},\uple{Y}_{i})$-rank strictly smaller than the $(\uple{X},\uple{Y})$-rank of $\phi$ and $\Sortrestr{\LAQ}{\K}$-terms $\uple{u}_i(\uple{Z})$ such that \[\TA\models\forall\uple{Z}\forall\uple{R}\,(\exists\uple{X}\exists\uple{Y}\phi\iff\bigvee_{i}\exists\uple{X}_{i}\exists\uple{Y}_{i}\phi_{i}(\uple{X}_{i},\uple{Y}_{i},\uple{u}_i(\uple{Z}),\uple{R})).\]
\end{lemma}

\begin{proof} Let $\Def{m}{\card{X}}$ and $\Def{n}{\card{Y}}$. As polynomials with variables in $\ValR$ are in fact elements of $\Ann$ and $\Ann$ is closed under composition (for the $\ValR$-variables), we may assumes that any $\Sortrestr{\LA}{\K}$-term appearing in $\phi$ is an element of $\Ann$. Let $f_i(\uple{X},\uple{Y},\uple{Z})$ be the  $\Sortrestr{\LA}{\K}$-terms appearing in $\phi$. Splitting $\phi$ into different cases, we may assume that whenever a variable $S$ appears as an $\MidR$-variable of an $f_i$ then $\phi$ implies that $\valR(S) > 0$ (in the part of the disjunction where $\valR(S) \leq 0$ we replace this $f_i$ by zero).

If an $X_{i}$ appears as an $\MidR$ variable in an $f_{i}$, then $\phi$ implies that $\valR(X_{i}) > 0$ and hence we can safely rename this $X_{i}$ as $Y_{n}$ and we obtain an equivalent formula of lower rank. If $Y_{i}$ appears as an $\ValR$-variable in an $f_{i}$, we can change this $f_{i}$ so that $Y_{i}$ appears as an $\MidR$-variable. Thus we may assume that the $X_{i}$ only appear as $\ValR$-variables and the $Y_{i}$ as $\MidR$-variables. Similarly adding new $Z_{j}$ variables, we may assume that each $Z_{j}$ appears only once (and in the end we can put the old variables back in) and that $\phi$ implies that $\valR(Z_{j}) > 0$ if it is an $\MidR$-variable.

Applying Proposition\,\ref{prop:bound deg prereg} to each of the $f_i(\uple{X},\uple{Y},\uple{Z}) = \sum_{\mu,\nu} f_{\mu,\nu}(\uple{Z})\uple{X}^{\nu}\uple{Y}^{\mu}$, we find $d$, $g_{i,\mu,\nu}$ and $u_{i,\mu,\nu}(\uple{Z})$ such that $g_{i,\mu,\nu}$ is preregular in $(\uple{X},\uple{Y})$ of degree $(\mu,\nu,d)$ and for every $M\models\TA$ and $\uple{a}\in M$, if $f_{i}(\uple{X},\uple{Y},\uple{a})$ is not the zero function, then there exists $(\mu,\nu)$ such that $\card{\mu}+\card{\nu} < d$ and $f_{i}(\uple{X},\uple{Y},\uple{a}) = f_{i,\mu,\nu}(\uple{a})g_{i,\mu,\nu}(\uple{X},\uple{Y},\uple{u}_{i,\mu,\nu}(\uple{a}))$. Splitting the formula into the different cases, we may assume that for each $i$, there are $\mu_{i}$ and $\nu_{i}$ such that $f_{i}(\uple{X},\uple{Y},\uple{a}) = f_{i,\mu_{i},\nu_{i}}(\uple{a})g_{i,\mu_{i},\nu_{i}}(\uple{X},\uple{Y},\uple{u}_{i}(\uple{a}))$ (in the case where no such $\mu_{i}$ and $\nu_{i}$ exist, we can replace $f_{i}$ by $0$). Let us now introduce a new variable $T_{j}$ to replace any argument of a $g_{i,\mu,\nu}$ that is not in $\uple{X}$ or $\uple{Y}$; and for each of these new $T_{j}$, we add to the formula $\valR(T_{j})\geq 0$ if $T_{j}$ is an $\ValR$-argument of $g_{i,\nu,\mu}$ or $\valR(T_{j})> 0$ if it is an $\MidR$-argument. Let us write $g_{i,\mu_{i},\nu_{i}} = \sum_{\nu} g_{i,\nu}\uple{Y}^{\nu}$. Note that $g_{i,\nu_{i}}$ is preregular in $\uple{X}$ of degree $(\mu_{i},0,d)$. We can split the formula some more (and still call it $\phi$) so that for each $i$, one of the the two conditions $\valR(g_{i,\nu_{i}}) > 0$ or $\valR(g_{i,\nu_{i}}) = 0$ holds.

If a condition $\valR(g_{i,\nu_{i}}) > 0$ occurs, let us add $\valR(Y_{n}) > 0 \wedge g_{i,\nu_{i}} - Y_{n} = 0$ to the formula. By Proposition\,\ref{prop:prereg 0 imp reg}, after a Weierstrass change of variable on the $\uple{X}$, we may assume that $g_{i,\nu_{i}} - Y_{n}$ is regular in $X_{m-1}$. By Weierstrass division, we can replace every $f_{j}$ by a term polynomial in $X_{m-1}$ and by Weierstrass preparation we can replace the equality $g_{i,\nu_{i}} - Y_{n} = 0$ by the equality of a term polynomial in $X_{m-1}$ to $0$. In the resulting formula, no $f\in\Ann$ is ever applied to a term containing $X_{m-1}$ and we can apply Remark\,\ref{rem:EQ alg form} to the formula where every $f\in\Ann$ is replaced by a new variable $S_{f}$ to obtain a $\K$-quantifier free formula $\psi(\upleneq{X}{m-1},\uple{Y},\uple{Z},\uple{T},\uple{S},\uple{R})$ such that \[\TA\models\exists X_{m-1}\phi\iff\psi(\upleneq{X}{m-1},\uple{Y},\uple{Z},\uple{u}(\uple{Z}),\uple{f}(\upleneq{X}{m-1},\uple{Y},\uple{Z}),\uple{R})\] and $\psi(\upleneq{X}{m-1},\uple{Y},\uple{Z},\uple{T},\uple{f}(\upleneq{X}{m-1},\uple{Y},\uple{Z}),\uple{R})$ is well-formed of $(\uple{X},\uple{Y})$-rank $(m-1,n+1)$.

If for all $i$ we have $\valR(g_{i,\nu_{i}}) = 0$, we add $\valR(X_{m}) \geq 0 \wedge X_{m}\prod_{i}g_{i,\nu_{i}} - 1 = 0$ to the formula. As every $g_{i,\nu_{i}}$ is preregular in $\uple{X}$ of degree $(\mu_{i},0,d)$, $g = X_{m}\prod_{i}g_{i,\nu_{i}} - 1$ is preregular in $\uple{X}$ of degree $(\mu,0,d')$ for some $\mu$ and $d'$. After a Weierstrass change of variables in $\uple{X}$, we may assume that $g$ and each $g_{i,\nu_{i}}$ are in fact regular in $X_{m}$. Hence by Weierstrass preparation we may replace $g$ in $g = 0$ by a term polynomial in $X_{m}$. Furthermore, by Remark\,\ref{rem:decomp form Kterm} the $f_{i}$ appear as $\ltf[n_{i}](f_{i})$ for some $n_{i}$ in the formula. Replacing $f_{i}$ by $f_{\mu_{i},\nu_{i}}g_{i,\mu_{i},\nu_{i}}$, we only have to show that $\ltf[n_{i}](g_{i,\mu_{i},\nu_{i}})$ can be replaced by  a term polynomial in $Y_{n-1}$ (and $X_{m}$). Let $h_{i} = X_{M}(\prod_{j\neq i}g_{j,\nu_{j}})g_{i,\nu_{i},\mu_{i}} = \sum_{\nu}h_{i,\nu}Y^{\nu}$. Then $h_{i,\nu_{i}} = X_{M}\prod_{i}g_{i,\nu_{i}} = 1$ and if $\nu < \nu_{i}$, $h_{i,\nu} = X_{M}(\prod_{j\neq i}g_{j,\nu_{j}})g_{i,\nu}\equiv 0 \mod J + (Z_{j}\mid Z_{j}$ is an $\MidR$-argument$)$. Hence $h_{i}$ is preregular in $(\uple{X},\uple{Y})$ of degree $(0,\nu_{i},d)$. After a Weierstrass change of variables of the $\uple{Y}$, we may assume that $h_{i}$ is in fact regular in $Y_{n-1}$.

Note that $\ltf[n_{i}](g_{i,\nu_{i},\mu_{i}}) = \ltf[n_{i}](X_{m})^{-1}\prod_{j\neq i}\ltf[n_{i}](g_{i,\nu_{i}})^{-1}\ltf[n_{i}](h_{i})$. By Weierstrass preparation we can replace $h_{i}$ by the product of a unit and $p_{i}$ a polynomial in $Y_{n-1}$. As we have included the trace of units on the $\lt[n]$ in our language, the unit is taken care of and by Weierstrass division by $g$, we can replace each coefficients in the $p_{i}$ and each of the $g_{i,\nu_{i}}$ by a term polynomial in $X_{m}$. Note that because we allow quantification on $\lt$, although the language does not contain the inverse on $\lt$, the inverses can be taken care of by quantifying over $\lt$. Hence we obtain a formula where $X_{m}$ and $Y_{n-1}$ only occur polynomially and we can proceed as in the previous case to eliminate them.
\end{proof}

\begin{corollary}[EQ TA wf]
Let $\phi(\uple{X},\uple{Y},\uple{Z},\uple{R})$ be a well-formed $\K$-quantifier free $\LA$-formula. Then there exists an $\LAQ$-formula $\psi(\uple{Z},\uple{R})$ such that $\TA\models\exists\uple{X}\exists\uple{Y}\phi\iff\psi$.
\end{corollary}

\begin{proof}
This follows from Lemma\,\ref{lem:EQ TA ind} and an immediate induction.
\end{proof}

\begin{proof}[Theorem\,\ref{thm:EQ TA}] Resplendence comes for free (see Proposition\,\ref{prop:EQ implies respl}). Hence, it suffices to show that if $\phi(X,\uple{Z})$ is a quantifier free $\LAQ$-formula, then there exists a quantifier free $\LAQ$-formula $\psi(\uple{Z})$ such that $\TA \models \exists X\phi\iff\psi$. First, splitting the formula $\phi$, we can assume that for any of its variables $S$, $\phi$ implies either $\valR(S)\geq 0$ or $\valR(S) < 0$, in the second case replacing $S$ by $S^{-1}$ we also have $\valR(S) > 0$. We also add one variable $X_{i}$ (resp. $Y_{i}$) per $\ValR$-argument (resp. $\MidR$-argument) of any $f\in\Ann$ applied to some non variable term $u$ and we add the corresponding equality $X_{i} = u$ (resp. $Y_{i} = u$) and the corresponding inequalities $\valR(X_{i})\geq 0$ (resp. $\valR(Y_{i})>0$) and quantify existentially over this variable. Splitting the formula further --- whether denominators in occurrences of $\Q$ are zero or not --- we can transform $\phi$ such that it contains no $\Q$. Now $\exists X\phi$ is equivalent to a disjunction of formulas $\exists\uple{X}\exists\uple{Y}\psi$ where $\psi$ is well-formed and we conclude by applying Corollary\,\ref{cor:EQ TA wf}.

This concludes the proof of Theorem\,\ref{thm:EQ TA}.
\end{proof}\bigskip

Recall that we denote by $\SWC<\ValR>(C)$ the set of all quantifier free $\Ldiv(C)$-definable sets.

\begin{definition}(Strong unit)
Let $M\models\TA$, $C = \K(\genan{C})$ and $S\in\SWC<\ValR>(C)$. We say that an $\Sortrestr{\LAQ}{\K}(C)$-term $E : \K\to\K$ is a strong unit on $S$ if for any open $\Val$-ball $\Def{b}{\oball[\Val]{\val(d)}{c}} \subseteq S$, there exists $\uple{a}$, $e\in\genan[C]{cd}$ and $F(t,\uple{z})\in\Ann$ such that $e\neq 0$ and for all $x\in b$, \[\val(F((x-c)/d,\uple{a})) = 0\] and \[E(x) = eF((x-c)/d,\uple{a}).\]
\end{definition}

Note that if $M$ is taken saturated saturated --- i.e. at least $(\card{\Ann}+\card{C})^{+}$-saturated --- if $E$ is a strong unit on $S$ then, by compactness, there exist a tuple $\uple{a}(y,z)$ of $\Sortrestr{\LAQ}{\K}(C)$-terms, a finite number of $\Sortrestr{\LAQ}{\K}(C)$-terms $e_{i}(y,z)$ and $F_{i}[t,\uple{u}]\in\Ann$ such that for all balls $b = \oball{\val(d)}{c}\subseteq S$, there is an $i$ such that for all $x\in b$, \[E(x) = e_{i}(c,d)F_{i}((x-c)/d,\uple{a}(c,d))\] and \[F_{i}((x-c)/d,\uple{a}(c,d))\in\inv{\Val}.\] Hence if $E$ is a strong unit on $S$ there is an $\LAQ(C)$-formula that witnesses it. If $E$ and $S$ are defined using some parameters $\uple{y}$ and for all $\uple{y}$ in some definable set $Y$, $E = E_{\uple{y}}$ is a strong unit on $S = S_{\uple{y}}$ then we can choose this formula uniformly in $\uple{y}$.

We will say that $E$ is an $\ValR$-strong unit on $S$ if it verifies all the requirements of a strong unit, where all references to $\Val$ are replaced by references to $\ValR$ (and references to $\ValR$ remain the same).

\begin{proposition}[ValR str unit]
If $E$ is an $\ValR$-strong unit on $S$ then it is also a strong unit on $S$.
\end{proposition}

\begin{proof}
If $b\subseteq S$ is an $\Val$-ball, then by Proposition\,\ref{prop:ball sub coarsen} there exists $d$ and $c$ such that $b = \oball[\Val]{\val(d)}{c} \subseteq \oball[\ValR]{\valR(d)}{c} \subseteq S$. But $E$ being a strong unit on $S$ for $\ValR$, it has the expected form on $\oball[\ValR]{\valR(d)}{c}$ and hence also on $\oball[\Val]{\val(d)}{c}$.
\end{proof}

\begin{definition}[Weierstrass prep](Weierstrass preparation for terms)
Let $M$ be an $\LAQ$-structure, $C = \K(\genan{C})\subseteq M$, $t : \K\to\K$ an $\Sortrestr{\LAQ}{\K}(C)$-term and $S\in\SWC<\ValR>(C)$. We say that $t$ has a Weierstrass preparation on $S$ if there exists an $\Sortrestr{\LAQ}{\K}(C)$-term $E$ that is a strong unit on $S$ and a rational function $R\in C(X)$ with no poles in $S(\alg{\K(M)})$ such that for all $x\in S$, $t(x) = E(x)R(x)$.

The structure $M$ has a Weierstrass preparation if for any $C = \K(\genan{C})$ and $\Sortrestr{\LAQ}{\K}(C)$-terms $t$ and $u : \ValR\to K$ we have:
\begin{thm@enum}
\item There exists a finite number of $S_{i}\in\SWC<\ValR>(C)$ that cover $\ValR$ such that $t$ has a Weierstrass preparation on each of the $S_{i}$.
\item If $t$ and $u$ have a Weierstrass preparation on some open ball $b$, and for all $x\in b$, $\val(t(x)) \geq \val(u(x))$, then $t+u$ also has a Weierstrass preparation on $b$.
\end{thm@enum}
\end{definition}

\begin{remark}
\begin{thm@enum}
\item An immediate consequence of Weierstrass preparation is that all $\Sortrestr{\LAQ}{\K}(M)$-terms in one variable have only finitely many isolated zeros. Indeed a zero of $t$ is the zero of one of the $R_{i}$ appearing in its Weierstrass preparation. That zero is isolated if $R_{i}$ is non-zero or the corresponding $S_{i}$ is discrete, i.e. is a finite set. In particular, let $\uple{m}$ be the parameters of $t$, then any isolated zero of $t$ is in the algebraic closure (in $\ACVF$) of $\K(\genan{\uple{m}})$. As the algebraic closure in $\ACVF$ coincides with the field theoretic algebraic closure, any isolated zero of $t$ is in fact also the zero of a polynomial (with coefficients in $\K(\genan{\uple{m}})$).
\item\label{rem:Weierstrass div 1st order} As for strong units, for each choice of term $t_{\uple{y}}$ (with parameters $\uple{y}$), there is an $\LAQ(\uple{y})$-formula that states that (i) holds for $t_{\uple{y}}$ in $M$ and we can choose this formula to be uniform in $\uple{y}$. For each choice of terms $t$, $u$ and formula defining $S$, there also is a (uniform) formula saying that (ii) holds for $t$, $u$ and $b$ in $M$.
\end{thm@enum}
\end{remark}

\begin{proposition}
Any $M\models\TA$ has Weierstrass preparation.
\end{proposition}

\begin{proof}
If $\ValR = \Val$, then the proposition is shown in \cite[Theorem\,5.5.3]{CluLipAnn} and (ii) --- called from now on invariance under addition --- is clear from the proof given there. The one difference in the Weierstrass preparation is that in \cite{CluLipAnn}, there is a finite set of points algebraic over the parameters where the behavior of the term is unknown. But this finite set can be replaced by discrete $S_i$ and as these exceptional points are common zeros of terms $u$ and $v$ such that $Q(u,v)$ is a subterm of $t$, it suffices to replace $Q(u,v)$ by $0$ and apply the theorem to the new term to obtain the Weierstrass preparation also on the discrete $S_{i}$. The fact that the strong units in \cite{CluLipAnn} have the proper form on open balls follows, for example, from the proof of  \cite[Lemma\,6.3.12]{CluLipAnn}.

If $\ValR \neq \Val$, the proposition follows from the $\Val=\ValR$ case and Proposition\,\ref{prop:ValR str unit}.
\end{proof}

\begin{remark}
\begin{thm@enum}
\item Let $t_{\uple{y}}$ be an $\Sortrestr{\LAQ}{\K}$-term with parameters $\uple{y}$. As shown in Remark\,\ref{rem:Weierstrass div 1st order}, there is an $\LAQ$-formula $\theta$ that states that Weierstrass preparation holds for $t_{\uple{y}}$ in models of $T$. More explicitly, there are finitely many choices of $S_{i}^{k}$, $E_{i}^{k}$ and $R_{i}^{k}$ (with parameters $\uple{u}(\uple{y})$ where $\uple{u}$ are $\Sortrestr{\LAQ}{\K}$-terms) such that for each $\uple{y}$ there is a $k$ such that the $S_{i}^{k}$, $E_{i}^{k}$ and $R_{i}^{k}$ work for $t_{\uple{y}}$. As $\TA$ eliminates $\K$-quantifiers, for each $k$ there is a $\K$-quantifier free $\LAQ$ formula $\theta_{k}(\uple{y})$ that is true when the $k$-th choice works for $t$ (and not the ones before). Hence taking $S_{i,k}$ to be $S^{k}_{i}\wedge\theta_{k}$, we could suppose that Weierstrass preparation for terms is uniform, but we will not be using that fact.
\item Conversely, the proof of Proposition\,\ref{prop:red alg} can be adapted to show that uniform Weierstrass preparation for terms implies $\K$-quantifier elimination. This is exactly the proof of quantifier elimination given in \cite{CluLipAnn}, although its authors did not see at the time that they were relying on a more uniform version of Weierstrass preparation for terms than what they had actually shown. Hence it would be interesting to know if one could prove uniform Weierstrass preparation for terms without using $\K$-quantifier elimination to recover their proof (see \cite{CluLipStrConv} for more on this subject).
\end{thm@enum}
\end{remark}

\begin{proposition}[TA alg ext]
Let $M\models\TA$, then the $\LAQ$-structure of $M$ can be extended (uniquely) to any algebraic extension of $\K(M)$, so that it remains a model of $\TA$. Moreover, if $C\substr M$ and $a\in\K(M)$ is algebraic over $\K(C)$, then $\K(\genan[C]{a}) = \K(C)[a]$.
 \end{proposition}

\begin{proof}
The case $\ValR = \Val$ is proved in \cite[Theorem\,2.18]{CLR}. The same proof applies when $\ValR\neq\Val$.
\end{proof}

To conclude this section, let us show that under certain circumstances analytic terms have a linear behavior.

\begin{proposition}[lin approx]
Let $M\models\TA$ and suppose that $\K(M)$ is algebraically closed. Let $t:\K\to\K$ be an $\LAQ(M)$-term and $b$ be an open ball in $M$ with radius $\xi\neq \infty$. Suppose that $t$ has a Weierstrass preparation on $b$ --- hence $t$ is differentiable at any $a\in b$ --- and $\ltf(\diff{t}{x})$ is constant on $b$. Also assume that $\val(t(x))$ is constant on $b$ or $t(x)$ is polynomial. Then for all $a$, $e\in b$, $\ltf(t(a) - t(e)) = \ltf(\diff{t}{a})\cdot\ltf(a-e)$.

Moreover, if $v(t(x))$ is constant on $b$ then $\val(t(a))\leq \val(\diff{t}{a}) + \xi$.
\end{proposition}

\begin{proof}
If $\val(t(x))$ is constant on $b$, then $\val(t(x)) - \val(t(a))\geq 0$ and by invariance under addition, $t(x) - t(a)$ has a Weierstrass preparation on $b$. If $t(x)$ is polynomial this is also clear. Hence there is $F_{a}\in\Ann$ (with other parameters in $\K(M)$), $P_{a}$, $Q_{a}\in\K(M)[X]$ such that for all $x\in b$, \[t(x) - t(a) = F_{a}\left(\frac{x-a}{g}\right)\frac{P_{a}(x)}{Q_{a}(x)}\] where $\val(F_{a}(y)) = 0$ for all $y\in \Mid$ and $\val(g) = \xi$. If $t$ is constant on $b$, i.e. $P_{a} = 0$, then the proposition follows easily. If not, $P_{a}$ has only finitely many zeros. Let $a_{i}$ be the zeros of $P_{a}$ in $\K(M)$ --- recall that $M$ is assumed algebraically closed --- and $m_{i}$ be the multiplicity of $a_{i}$. Let $c_{j}$ be the zeros of $Q_{a}$ and $n_{j}$ be their multiplicities. Note that every zero of $Q_{a}(x)$ is outside $b$, hence for all $j$, $\val(c_{j} - a) \leq \xi$. For all $e\in b$, note that $t(x) - t(a)$ is also differentiable at $e$  with differential $\diff{t}{e}$ and hence, if $e$ is distinct from all $a_{i}$, then:

\begin{eqnarray*}
\ltf\left(\frac{\diff{t}{a}}{t(e) - t(a)}\right) &=&
\ltf\left(\frac{\diff{t}{e}}{t(e) - t(a)}\right)\\
&=&\ltf\left(\frac{\diff[x]{\left(F_{a}\left(\frac{x-a}{g}\right)\right)}{e}}{F_{a}\left(\frac{e-a}{g}\right)} + \frac{\diff{(P_{a})}{e}}{P_{a}(e)} + \frac{\diff{(Q_{a})}{e}}{Q_{a}(e)}\right)\\
&=& \ltf\left(\frac{\diff{(F_{a})}{\frac{e-a}{g}}}{gF_{a}\left(\frac{e-a}{g}\right)} + \sum_{i}\frac{m_{i}}{e-a_{i}} + \sum_{j}\frac{n_{j}}{e-c_{j}}\right).
\end{eqnarray*}

For any $y\in\Mid$, $\val(\diff{(F_{a})}{y})\geq 0 = \val(F_{a}(y))$, hence $\val(\diff{(F_{a})}{y}/(gF_{a}(y)))\geq -\val(g) > -\val(e-a)$. We also have that for all $j$, $\val(1/(e-c_{j})) = -\val(e-c_{j}) > -\val(e-a)$. Finally, suppose that there is a unique $a_{i_{0}}$ such that $\val(e-a_{i_{0}})$ is maximal, then, for all $i\neq i_{0}$, $\val(1/(e-a_{i})) > \val(1/(e-a_{i_{0}}))$ and hence $\ltf(m_{i_{0}})\ltf(e-a_{i_{0}})^{-1} = \ltf(\diff{t}{a})\ltf(t(e)-t(a))^{-1}$, i.e. $\ltf(t(e)-t(a)) = \ltf(\diff{t}{a}m_{i_{0}}^{-1}(e-a_{i_{1}}))$.

As $t(e)\neq t(a)$, this immediately implies that $\diff{t}{a}\neq 0$. Let us now show that if $a_{i}\in b$ it cannot be a multiple zero. If it were \[\diff{t}{a_{i}} = \diff{(F_{a}((x-a)/c)/Q_{a}(x))}{a_{i}} P_{a}(a_{i}) + P'_{a}(a_{i}) F_{a}((a_{i}-a)/c)/Q_{a}(a_{i}) = 0\] which is absurd. Hence for all $a_{i}\in b$, $m_{i} = 1$ and if we could show that there is a unique $a_{i}\in b$ --- namely $a$ itself --- we would be done.

Suppose there are more that one $a_{i}$ in $b$ and let $\Def{\gamma}{\min\{\val(a_{i}-a_{j})\mid a_{i},a_{j}\in b \wedge i\neq j\}}$. We may assume $\val(a_{0}-a_{1}) = \gamma$. Let us also assume the $a_{i}$ have been numbered so that there exists $i_{0}$ such that for all $i\leq i_{0}$, $\val(a_{i}-a_{0}) = \gamma$ and for all $i > i_{0}$, $\val(a_{i} - a_{0}) < \gamma$. In particular, for all $i \neq j \leq i_{0}$, $\val(a_{i}-a_{j}) = \gamma$. For each $i\leq i_{0}$, let $e_{i}$ be such that $\val(e_{i}-a_{i})> \gamma$. Then we can apply the previous computation to $e_{i}$ and we get that $\ltf(t(e_{i}) - t(a)) = \ltf(\diff{t}{a})\ltf(e_{i}-a_{i})$. But
\[\ltf(t(e_{i}) - t(a)) = \ltf(F_{a}\left(\frac{e_{i}-x_{\alpha_{0}+1}}{g}\right))\ltf(p)\prod_{k}(\ltf(e_{i}-a_{k}))^{m_{k}}\ltf(q)^{-1}\prod_{j}(\ltf(e_{i}-c_{j}))^{-n_{j}}\]
where $p$ and $q$ are the dominant coefficients of respectively $P_{a}$ and $Q_{a}$ and hence
\[\ltf(\diff{t}{a}) = \ltf(F_{a}\left(\frac{e_{i}-x_{\alpha_{0}+1}}{g}\right))\ltf(p)\prod_{k\neq i}(\ltf(e_{i}-a_{k}))^{m_{k}}\ltf(q)^{-1}\prod_{j}(\ltf(e_{i}-c_{j}))^{-n_{j}}.\]
As $\ltf(F_{a}((e_{i}-x_{\alpha_{0}+1})/g))$, $\ltf(e_{i}-a_{k})$ for all $k > i_{0}$ and $\ltf(e_{i}-c_{j})$ do not depend on $i$, and for all $k\leq i_{0}$, $k\neq i$, $\ltf(e_{i}-a_{k}) = \ltf(a_{i}-a_{k})$, we obtain that for all $i,j\leq i_{0}$: \[\prod_{i\neq k\leq i_{0}}\ltf(a_{i}-a_{k}) = \prod_{j\neq k\leq i_{0}}\ltf(a_{j}-a_{k}).\]

Replacing $a_{i}$ by $(a_{i}-a_{0})/g$ where $\val(g) = \gamma$, we obtain the same equalities but we may assume that for all $i\leq i_{0}$, $a_{i}\in\Val$ and for all $i\neq j$, $a_{i} - a_{j} \in \inv{\Val}$. The equations can now be rewritten as $\prod_{i\neq k}\resf(a_{i}-a_{k}) = \prod_{i\neq k}(\resf(a_{i})-\resf(a_{k})) = c$ for some $c\in\res(M)$. Let $P = \prod_{k}(X-\resf(a_{k}))$ then our equations state that $P'(\resf(a_{i})) - c = 0$ for all $i\leq i_{0}$. But $P' - c$ is a degree $i_{0}$ polynomial, it cannot have $i_{0}+1$ roots.

Finally, if $\val(t(x))$ is constant on $b$, then, for all $a$ and $e\in b$, $\val(t(a)) \leq \val(t(a)-t(e)) = \val(\diff{t}{a}) + \val(a-e)$. As this holds for any $e$, we must have $\val(t(a)) \leq \val(\diff{t}{a}) + \zeta$.
\end{proof}

\begin{remark}
The conclusion of Proposition\,\ref{prop:lin approx} seems very close to the Jacobian property (e.g. \cite[Definition\,6.3.5]{CluLipAnn}). In fact, this lemma is very similar (both in its hypothesis and its conclusion) to \cite[Lemma\,6.3.9]{CluLipAnn}.
\end{remark}

\section{\texorpdfstring{$\sigma$}{Sigma}-Henselian fields}\label{sec:isom}

\begin{definition}(Analytic field with an automorphism)
Let us suppose that each $\Ann_{m,n}$ is given with an automorphism of the inductive system $t \mapsto t^{\sigma} : \Ann_{m,n}\to \Ann_{m,n}$. An analytic field $M$ with an automorphism is a model of $\TA$ with a distinguished $\Llt\cup\{\DivR[1]\}$-automorphism $\sigma$ such that for symbols $t\in\Ann_{m,n}$ and $\uple{x}\in\K(M)^{m+n}$, $\sigma(t(\uple{x})) = t^{\sigma}(\sigma(\uple{x}))$.
\end{definition}

Let $\Def{\LAQs}{\LAQ\cup\{\sigma\}\cup\{\sigma_{n}\mid n\in\Nn\}}$. An analytic field $M$ with an automorphism $\tau$ can be made into an $\LAQs$-structure by interpreting $\sigma$ as $\Sortrestr{\tau}{\K}$ and $\sigma_{n}$ as $\Sortrestr{\tau}{\lt[n]}$. Note that $\sigma$ also induces a ring automorphism on every $\res[n]$ and an ordered group automorphism $\sigma_{\Valgp}$ on $\Valgp$. We will write $\TAs$ for the $\LAQs$-theory of analytic fields with an automorphism. We will most often write $\sigma$ instead of $\sigma_{n}$ as there should not be any confusion.

If $K$ is a field with an automorphism $\sigma$ (also referred to as a difference field in this text), we will write $\Def{\fix!(K)}{\{x\in K\mid \sigma(x) = x\}}$ for its fixed field. For all $x\in K$, we will write $\sprol(x)$ for the tuple $x,\sigma(x),\ldots,\sigma^{n}(x)$ where the $n$ should be explicit from the context.

\begin{remark}[sigma terms]
In fact $\sigma$ induces an action on all $\Sortrestr{\LAQ}{\K}$-terms and we have $\TAs\models\sigma(t(\uple{x})) = t^{\sigma}(\sigma(\uple{x}))$. It follows immediately that for any $\Sortrestr{\LAQs}{\K}$-term $t$ there is an $\Sortrestr{\LAQ}{\K}$-term $u$ such that $\TAs\models t(\uple{x}) = u(\sprol(\uple{x}))$.
\end{remark}

\begin{definition}(Linearly closed difference field)
A difference field $(K,\sigma)$ is linearly closed if every equation of the form $\sum_{i=0}^{n} a_i\sigma^i(x) = b$, where $a_{n}\neq 0$, has a solution. 
\end{definition}

\begin{definition}[lin approx def](Linear approximation)
Let $K$ be a valued field with an automorphism $\sigma$, $f:K^{n}\to K^{n}$ a (partial) function and $\uple{d}\in K^{n}$.
\begin{thm@enum}
\item Let $\uple{b}$ be a tuple of open balls in $M$. We say that $\uple{d}$ linearly approximates $f$ on $\uple{b}$ if for all $\uple{a}$ and $\uple{c}\in \uple{b}$ we have:
\[\val(f(\uple{c}) - f(\uple{a}) - \uple{d}\cdot(\uple{c}-\uple{a})) >  \min_{i}\{\val(d_{i}) + \val(c_{i}-a_{i})\}.\]
\item Let $b$ be an open ball of $M$. We say that $\uple{d}$ linearly approximates $f$ at prolongations on $b$ if for all $a$, $c\in b$ we have:
\[\val(f(\sprol(c)) - f(\sprol(a)) - \uple{d}\cdot\sprol(c-a)) >  \min_{i}\{\val(d_{i}) + \val(\sigma^{i}(c-a))\}.\]
\end{thm@enum}
\end{definition}

\begin{remark}
\begin{thm@enum}
\item \label{rem:rad poly} Let $M\models\TAs$. Suppose that $\sigma$ is an isometry, i.e. $\sigma_{\Valgp} = \id$. Let $t$ be an $\Sortrestr{\LA}{\K}(\Val(M))$-term and $a\in\Val(M)$, we can show that $\diff{t}{\uple{a}}$ linearly approximates $t$ on $\oball{\gamma}{a}$ where $\gamma = \min_{i}\{\val(\diff[i]{f}{\uple{a}})\}$.
\item We allow a slight abuse of notation by saying that terms constant on a ball are linearly approximated (at prolongations) by the zero tuple, even though the required inequality does not hold as $\infty \not> \infty$.
\end{thm@enum}
\end{remark}

Let us first prove that it suffices to show linear approximation variable by variable to obtain linear approximation for the whole function. We will write $(\upleto{a}{x}{i})$ for the tuple $\uple{a}$ where the $i$-th component is replaced by $x_{i}$ (with a slight abuse of notation as the $x_{i}$ does not appear in the right place) and $\upleleq{a}{i}$ for the tuple $\uple{a}$ where the $j$-th components for $j > i$ are replaced by zeros.

\begin{proposition}[one var to all]
Let $(K,\val)$ be a valued field, $f:K^{n}\to K$, $\uple{d}\in K^{n}$ and $\uple{b}$ a tuple of balls. If for all $\uple{a}\in\uple{b}$ and $j<n$, $d_{j}$ linearly approximates $f(\upleto{a}{x}{j})$ on $b_{j}$, then $\uple{d}$ linearly approximates $f$ on $\uple{b}$. 
\end{proposition}

\begin{proof}
Let $\uple{a}$ and $\uple{e}\in\uple{b}$ and $\uple{\epsilon} = \uple{e} - \uple{a}$. Then, we have
\begin{eqnarray*}
\val(f(\uple{a} + \uple{\epsilon}) - f(\uple{a}) - \uple{d}\cdot\uple{\epsilon}) &=& \val(\sum_{j} f(\uple{a} + \upleleq{\epsilon}{j}) - f(\uple{a} + \upleleq{\epsilon}{j-1}) - d_{j}\epsilon_{j})\\
&\geq& \min_{j}\{f(\uple{a} + \upleleq{\epsilon}{j}) - f(\uple{a} + \upleleq{\epsilon}{j-1}) - d_{j}\epsilon_{j}\}\\
&>& \min_{j}\{\val(d_{j}) + \val(\epsilon_{j})\}.
\end{eqnarray*}
And that concludes the proof.
\end{proof}

Although linear approximation (at prolongations) looks like differentiability, one must be aware that linear approximations are not uniquely determined, because, among other reasons, we are only looking at tuples that are prolongations but also because the error term is only linear. But when $\sigma$ is an isometry, we can recover some uniqueness, and give an alternative definition (perhaps of a more geometric flavor) of linear approximation at prolongations.

\begin{definition}($\resval[\gamma]$)
Let $(K,\val)$ be a valued field and $\gamma\in\val(\inv*{K})$. We define $\Def{\resval[\gamma]}{\cball{\gamma}{0}/\oball{\gamma}{0}}$ and let $\resvalf[\gamma]$ denote the canonical projection $\cball{\gamma}{0} \to \resval[\gamma]$. Note that $\resval[\gamma]$ can be identified (canonically) with $\vallt[1]^{-1}(\gamma)\cup\{0\}\subseteq \lt[1]$.
\end{definition}

\begin{proposition}[def mdiff]
Let $(K,\val)$ be a valued field with an isometry $\sigma$ and a linearly closed residue field. Let $f : K^{n}\to K$, $\uple{d}$ be a linear approximation of $f$ at prolongations on some open ball $b$ with radius $\xi$, $\uple{e}\in K^{n}$, $\Def{\delta}{\val(\uple{d})}$ and $\Def{\eta}{\val(\uple{e})}$. The following are equivalent:
\begin{thm@enum}
\item $\uple{e}$ is a linear approximation of $f$ at prolongations on $b$;
\item $\val(\uple{d} - \uple{e}) > \min\{\delta,\eta\}$;
\item $\eta = \delta$ and $\resvalf[\delta](\uple{d}) = \resvalf[\delta](\uple{e})$.
\end{thm@enum}
\end{proposition}

\begin{proof}
\begin{list}{}{}
\item[(i)$\imp$(ii)] Suppose $\uple{d}\neq\uple{e}$. Let $\epsilon$ be such that $\val(\epsilon) > \xi$ and let $g\in K$ be such that $\val(g) = \val(\uple{d}-\uple{e})$. Then $P(\sprol(x)) := \sum_{i}(d_{i}-e_{i})\sigma^{i}(\epsilon)g^{-1}\epsilon^{-1}\sigma^{i}(x)$ is a linear difference polynomial with a non zero residue. As $K$ is residually linearly closed, the residue of $P$ cannot always be zero and hence there exists $c\in\inv{\Val}$ such that $\val(P(\sprol(c))) = 0$, i.e. $\val((\uple{d}-\uple{e})\cdot\sprol(\epsilon c)) = \val(g) + \val(\epsilon)$. But then, for all $a\in b$:
\begin{eqnarray*}
\val(g) + \val(\epsilon) &=& \val((\uple{d}-\uple{e})\cdot\sprol(\epsilon c))\\
&=& \val(f(\sprol(a+\epsilon c)) - f(\sprol(a)) - \uple{e}\cdot\sprol(\epsilon b))\\
&& - f(\sprol(a+\epsilon c)) + f(\sprol(a)) + \uple{d}\cdot\sprol(\epsilon b)\\
&>& \val(\epsilon) +  \min\{\delta,\eta\}
\end{eqnarray*}

i.e. $\val(\uple{d}-\uple{e}) > \min\{\delta,\eta\}$.
\item[(ii)$\imp$(iii)] First, suppose that $\delta < \eta$, then if $\val(d_{i})$ is minimal, $\val(d_{i}) = \delta < \eta\leq \val(e_{i})$ and hence $\val(d_{i}-e_{i}) = \val(d_{i}) = \delta = \min\{\delta,\eta\}$ contradicting our previous inequality. Hence we must have, by symmetry, $\delta = \eta$. Now inequality (ii) can be rewritten $\val(\uple{d} - \uple{e}) > \delta$ which exactly means that $\resvalf[\delta](\uple{d}) = \resvalf[\delta](\uple{e})$.
\item[(iii)$\imp$(i)] For all $\epsilon$ such that $\val(\epsilon) > \xi$, as $\val(\uple{d}-\uple{e})>\delta$, we have:
\begin{eqnarray*}
\lefteqn{\val(f(\sprol(a+\epsilon)) - f(\sprol(a)) - \uple{e}\cdot\sprol(\epsilon))}\\
&=& \val(f(\sprol(a+\epsilon)) - f(\sprol(a)) - \uple{d}\cdot\sprol(\epsilon) + (\uple{d}-\uple{e})\cdot\sprol(\epsilon))\\
&>& \delta + \val(\epsilon)\\
&=& \eta + \val(\epsilon).
\end{eqnarray*}
This concludes the proof.\qed
\end{list}
\end{proof}

\begin{remark}
\begin{thm@enum}
\item\label{item:lin app} In the isometry case, linear approximations describe the trace of a given function on $\lt[1]$. More precisely, a function $f$ is linearly approximated at prolongations on some open ball $b$ with radius $\xi$ if and only if there exists $\delta\in\val(K)$ and $\uple{d}\in\resval[\delta](K)$ such that for all $\gamma > \xi$ and $a\in b$, the function $\resvalf[\gamma](\epsilon)\mapsto \resvalf[\gamma+\delta](f(\sprol(a+\epsilon)) - f(\sprol(a))) : \resval[\gamma]\to \resval[\gamma+\delta]$ is well defined and coincides with the function $x\mapsto \uple{d}\cdot\sprol(x)$ (where the sum is given by $+_{1,1}$). 
\item If we are working in a valued field with a linearly closed residue field, it follows from Proposition\,\ref{prop:def mdiff}, that $\delta$ and $\uple{d}$ from \ref{item:lin app} are actually uniquely defined.
\end{thm@enum}
\end{remark}

\begin{definition}[sigma-Hensel]($\sigma$-Henselianity)
Let $M\models\TAs$, $t$ be an $\Sortrestr{\LAQ}{\K}(M)$-term, $\uple{d}\in\K(M)$, $a\in\K(M)$ and $\xi\in\Valgp(M)$. We say that $(t,a,\uple{d},\xi)$ is in $\sigma$-Hensel configuration if $\uple{d}$ linearly approximates $t$ at prolongations on $\oball{\xi}{a}$ and:
\[\val(t(\sprol(a))) > \min_{i}\{\val(d_{i}) + \sigma^{i}(\xi)\}.\]

We say that $M$ is $\sigma$-Henselian if for all $(t,a,\uple{d},\xi)$ in $\sigma$-Hensel configuration, there exists $c\in\K(M)$ such that $t(\sprol(c)) = 0$ and $\val(c-a) \geq \max_{i}\{\val(\sigma^{-i}(t(\sprol(a)) d^{-1}_{i}))\}$.
\end{definition}

\begin{remark}
By Remark\,\ref{rem:rad poly}, when $\sigma$ is an isometry, this form of the $\sigma$-Hensel lemma is equivalent to classical forms for difference polynomials --- i.e without any analytic structure --- as stated in \cite{ScaDvalF,ScaRelFrob,AzgvdD} for example. In particular, it implies Hensel's lemma (for polynomials).
\end{remark}

\begin{definition}[pseudo conv](Pseudo-convergence)
Let $M\models\TAs$.
\begin{thm@enum}
\item A sequence $(x_{\alpha})_{\alpha\in\beta}$ of (distinct) points in $\K(M)$ indexed by an ordinal is said to be pseudo-convergent if for all $\alpha$, $\gamma$, $\delta\in\beta$ such that $\alpha < \gamma < \delta$ we have $\val(x_{\alpha} - x_{\delta}) < \val(x_{\gamma} - x_{\delta})$;
\item We say that $a\in \K(M)$ is a pseudo-limit of the pseudo-convergent sequence $(x_\alpha)$ --- and we write  $x_{\alpha}\pc a$ --- if for all $\alpha < \gamma < \beta$, $\val(x_{\alpha} - a) < \val(x_{\gamma} - a)$;
\item A pseudo-convergent sequence of elements of $C\subseteq\K(M)$ is said to be maximal (over $C$) if it has no pseudo-limit in $C$;
\item We say that a sequence $(\uple{x}_\alpha)$ of tuples pseudo-solves an $\Sortrestr{\LAQ}{\K}(M)$-term $t$ if $t=0$ or for $\alpha\gg 0$ --- i.e. for $\alpha$ in a final segment --- $t(\uple{x}_{\alpha})\pc 0$.
\item We say that a sequence $(x_\alpha)$ $\sigma$-pseudo-solves an $\Sortrestr{\LAQ}{\K}(M)$-term $t$ if $(\sprol(x_{\alpha}))$ pseudo-solves $t$.
\item We say that $M$ is maximally complete if any pseudo-convergent sequence in $M$ (indexed by a limit ordinal) has a pseudo-limit in $M$;
\item We say $M$ is $\sigma$-algebraically maximally complete if any pseudo-sequence $(x_\alpha)$ from $M$ (indexed by a limit ordinal) $\sigma$-pseudo-solving an $\Sortrestr{\LAQ}{\K}(M)$-term $t\neq 0$ has a pseudo-limit in $M$.
\end{thm@enum}
\end{definition}

\begin{remark}
\begin{thm@enum}
\item Let $(x_{\alpha})$ be a pseudo-convergent sequence, then for all $\alpha_{0} < \alpha_{1}$, $\val(x_{\alpha_{0}}-x_{\alpha_{1}}) = \val(x_{\alpha_{0}}-x_{\alpha_{0}+1}) =: \gamma_{\alpha_{0}}$. The $\gamma_{\alpha}$ form a strictly increasing sequence. If $x_{\alpha}\pc a$ then $\val(a-x_{\alpha_{0}}) = \gamma_{\alpha_{0}}$ and if $b$ is such that for all $\alpha$, $\val(b-a) > \gamma_{\alpha}$ then we also have $x_{\alpha}\pc b$.
\item As, in any valued field, balls with a non infinite radius always have more than one point, if $(x_{\alpha})$ is a maximal pseudo-convergent sequence over $K$ then either $\gamma_{\alpha}$ is cofinal in $\val(\inv*{K})$ and $(x_{\alpha})$ is indexed by the successor of a limit ordinal or $(x_{\alpha})$ is indexed by a limit ordinal.
\end{thm@enum}
\end{remark}

\begin{proposition}[max compl sigmaH]
If $M$ is $\sigma$-algebraically maximally complete and $\res[1](M)$ is linearly closed then $M$ is $\sigma$-Henselian.
\end{proposition}

\begin{proof}First an easy claim:
\begin{claim}[sH conf]
Let $(t,a,\uple{d},\xi)$ be in $\sigma$-Hensel configuration, then \[\max_{i}\{\val(\sigma^{-i}(t(\sprol(a)) d^{-1}_{i}))\}>\xi.\]
\end{claim}

\begin{proof}
As $(t,a,\uple{d},\xi)$ is in $\sigma$-Hensel configuration, there exists $i_{0}$ such that $\val(t(\sprol(a))) > \val(d_{i_{0}}) + \sigma^{i_{0}}(\xi)$ and hence $\val(\sigma^{-i_{0}}(t(\sprol(a))d_{i_{0}}^{-1})) = \sigma^{-i_{0}}(\val(t(\sprol(a))) - \val(d_{i_{0}})) > \xi$.
\end{proof}

And now, two lemmas about finding better approximations to zeros of terms.
\begin{lemma}[succ Hensel]
Let $(t,a,\uple{d},\xi)$ be in $\sigma$-Hensel configuration such that $t(\sprol(a))\neq 0$. Then there exists $c$ such that $\val(c-a) = \max_{i}\{\val(\sigma^{-i}(t(\sprol(a)) d^{-1}_{i}))\}$, $\val(t(\sprol(c))) > \val(t(\sprol(a)))$ and $(t,c,\uple{d},\xi)$ is also in $\sigma$-Hensel configuration.
\end{lemma}

\begin{proof}
Choose any $\epsilon\in \K(M)$ with $\val(\epsilon) = \max_{i}\{\val(\sigma^{-i}(t(\sprol(a)) d^{-1}_{i}))\}$. By Claim\,\ref{claim:sH conf}, $\val(\epsilon) > \xi$. For all $x\in\Val$, let $\Def{R(a,\epsilon,x)}{t(\sprol(a)+\sprol(\epsilon x)) - t(\sprol(a)) - \uple{d}\cdot \sprol(\epsilon x)}$ and
\[\Def{u(x)}{\frac{t(\sprol(a) +\sprol(\epsilon x))}{t(\sprol(a))}} = 1 + \sum_{i}\frac{d_{i}\sigma^{i}(\epsilon)}{t(\sprol(a))}\sigma^{i}(x) + \frac{R(a,\epsilon,x)}{t(\sprol(a))}.\]
For all $i$, \[\val\left(\frac{d_{i}\sigma^{i}(\epsilon)}{t(\sprol(a))}\right) \geq \val(d_{i}) + \val(t(a)) - \val(d_{i}) - \val(t(a)) = 0\] and for any $i_{0}$ such that $\val(\epsilon) = \val(\sigma^{-i_{0}}(t(\sprol(a)) d^{-1}_{i_{0}}))$ it is an equality. As $\uple{d}$ linearly approximates $t$ at prolongations on $\oball{\xi}{a}$, we also have \[\val(R(a,\epsilon,x)) > \min_{i}\{\val(\sigma^{i}(\epsilon)) + \val(d_{i})\}\geq\val(t(\sprol(a)))\] and $\resf[1](u(x)) = 0$ is a non trivial linear equation in the residue field. As $\res[1](M)$ is linearly closed, this equation has a solution $\resf[1](e)$. Note that we must have $\resf[1](e)\neq 0$. 

Let $c = a+\epsilon e$. It is clear that $\val(c-a) = \val(\epsilon) = \max_{i}\{\val(\sigma^{-i}(t(\sprol(a)) d^{-1}_{i}))\} > \xi$ and that $\val(t(\sprol(c))) = \val(t(\sprol(a))u(e)) > \val(t(\sprol(a))) > \min_{i}\{\val(d_{i}) + \sigma_{i}(\xi)\}$.
\end{proof}

\begin{lemma}[lim Hensel]
Let $(x_{\alpha})$ be a pseudo-convergent sequence (indexed by a limit ordinal). Assume that for all $\alpha$, $(t,x_{\alpha},\uple{d},\xi)$ is in $\sigma$-Hensel configuration, $\val(x_{\alpha+1} - x_{\alpha}) \geq \max_{i}\{\val(\sigma^{-i}(t(\sprol(x_{\alpha})) d^{-1}_{i}))\}$ and $t(\sprol(x_{\alpha}))\pc 0$. If $c$ is such that $x_{\alpha}\pc c$, then $(t,c,\uple{d},\xi)$ is in $\sigma$-Hensel configuration and for all $\alpha$, $\val(t(\sprol(c))) > \val(t(\sprol(x_{\alpha})))$.
\end{lemma}

\begin{proof}
First of all, as $(t,x_{0},\uple{d},\xi)$ is in $\sigma$-Hensel configuration, $\uple{d}$ continuously linearly approximates $t$ at prolongations on $\oball{\xi}{x_{0}}$. By Claim\,\ref{claim:sH conf}, $\val(c-x_{0}) = \val(x_{1}-x_{0}) > \xi$. Moreover, let $\Def{R(x,c)}{t(\sprol(c)) - t(\sprol(x)) - \uple{d}\cdot\sprol(c-x)}$. Then for all $\alpha$,
\begin{eqnarray*}
\val(t(\sprol(c))) &=& \val(t(\sprol(x_{\alpha})) + \uple{d}(\sprol(x_{\alpha}))\cdot\sprol(c-x_{\alpha}) + R(x_{\alpha},c))\\
&\geq& \min_{i}\{\val(t(\sprol(x_{\alpha}))),\val(d_{i}) + \val(\sigma^{i}(c-x_{\alpha}))\}\\
&\geq& \val(t(\sprol(x_{\alpha}))).
\end{eqnarray*}
Finally, as $\val(t(\sprol(c))) \geq \val(t(\sprol(x_{0}))) > \min_{i}\{\val(d_{i}) + \sigma^{i}(\xi)\}$, $(t,c,\uple{d},\xi)$ is also in $\sigma$-Hensel configuration.
\end{proof}

Let $(t,a,\uple{d},\xi)$ be in $\sigma$-Hensel configuration. If $t = 0$, we are done, if not let $(x_{\alpha})_{\alpha\in\beta}$ be a maximal sequence (with respect to the length) such that $x_{0} = a$ and for all $\alpha$, $(t,x_{\alpha},\uple{d},\xi)$ is in $\sigma$-Hensel configuration, $\val(x_{\alpha+1} - x_{\alpha}) \geq \max_{i}\{\val(\sigma^{-i}(t(\sprol(x_{\alpha})) d^{-1}_{i}))\}$ and $t(\sprol(x_{\alpha}))\pc 0$. If $\alpha$ is a limit ordinal,  as $M$ is $\sigma$-algebraically maximally complete, and $t\neq 0$, $(x_{\alpha})$ has a pseudo-limit $x_{\beta}$. By Lemma\,\ref{lem:lim Hensel}, the sequence $(x_{\alpha})_{\alpha\in\beta+1}$ still meets the same requirements, contradicting the maximality of $(x_{\alpha})_{\alpha\in\beta}$. It follows that $\beta = \gamma+1$. If $t(\sprol(x_{\gamma}))\neq 0$, then applying Lemma\,\ref{lem:succ Hensel}, to $(t,x_{\gamma})$, we obtain an element $x_{\beta}$ such that $(x_{\alpha})_{\alpha\in\beta+1}$ still meets the same requirements, contradiction the maximality of $(x_{\alpha})_{\alpha\in\beta}$ once again. Hence we must have that $t(\sprol(x_{\gamma})) = 0$ and $c = x_{\gamma}$ is a solution to the $\sigma$-Hensel configuration $(t,a,\uple{d},\xi)$.
\end{proof}

\begin{definition}($\TAsH!$)
Let $\TAsH$ be the $\LAQs$-theory of analytic fields with an automorphism that are $\sigma$-Henselian and have a non-trivial valuation group. To specify the characteristic we will write $\TAsH[0,0]$ or $\TAsH[0,p]$.
\end{definition}

\begin{proposition}
Let $\Def{\Ann}{\bigcup_{\uple{X},\uple{Y}}\genan[\Wittr{\alg{\Ff_{p}}}]{\uple{X}}[[\uple{Y}]]}$ and $\WittLAQs!$ be the $\LAQs$-structure with base set $\Wittf(\alg{\Ff_{p}})$, the obvious valued field structure and analytic structure and interpreting $\sigma$ as the lifting to $\Wittf(\alg{\Ff_{p}})$ of the Frobenius automorphism on $\alg{\Ff_{p}}$. Then $\WittLAQs\models\TAsH[0,p]$.
\end{proposition}

\begin{proof}
It is clear that $\WittLAQs\models\TA$. As $\Wittf(\alg{\Ff_{p}})$ is complete with a discrete valuation it is maximally complete. It follows from Proposition\,\ref{prop:max compl sigmaH} that it is $\sigma$-Henselian.
\end{proof}

In the definition of $\TAsH$, we have not required the residue field to be linearly closed, since it comes for free:

\begin{proposition}
Let $M\models\TAsH$, then $\K(M)$ is linearly closed.
\end{proposition}

\begin{proof}
Let  $c\in\K(M)$ and $P(x) = \sum_{i}a_{i}x_{i}$ be a non zero linear polynomial. Let $\epsilon \in \K(M)$ be such that $\val(\epsilon) < \val(c) - \val(a_{0})$. Then $P(\sprol(\epsilon x)) = \sum a_{i}\sigma^{i}(\epsilon)\sigma^{i}(x)$ and \[\min_{i}\{\val(a_{i}\sigma^{i}(\epsilon))\} \leq \val(a_{0}) + \val(\epsilon) < \val(c).\] Thus, we may assume that $\min_{i}\{\val(a_{i})\} < \val(c) = \val(P(\sprol(0)) - c)$. But, as $P$ is linear, $\uple{a}$ linearly approximates $P$ at prolongations on $\Mid$ and $(P-c,0,\uple{a},0)$ is in $\sigma$-Hensel configuration. As $M$ is $\sigma$-Henselian, there exists $e\in\K(M)$ such that $P(\sprol(e)) = c$.
\end{proof}

To conclude this section, let us show that $\TAsH$ behaves well with respect to coarsening. The following results will mainly be used in Section\,\ref{subsec:EQ} to transfer quantifier results form equicharacteristic zero to mixed characteristic.

Let $\LL$ be an $\lt$-enrichment of $\LAQs$ and $T$ be an $\LL$-theory containing $\TAsH[0,p]$ Morleyized on $\lt$ (cf. Definition\,\ref{def:Morl}). By Section\,\ref{sec:coar} we can find an $\ltinf$-enrichment $\Linf$ of $\Lltinf$ --- the $\infty$ in $\Lltinf$ is there to recall that the leading term structure is given by $\ltinf$ and not the $\lt[n]$, although, to add to the general confusion, the $\lt[n]$ are indeed present in the enrichment --- an $\Linf$-theory $\Tinf[1]\supseteq\Tvf[0,0]<\infty>$ and two functors $\Coar[1]:\Str(T)\to\Str(\Tinf[1])$ and $\UCoar[1] : \Str(\Tinf[1])\to\Str(T)$. For any $C$ in $\Str(T)$ we enrich $\Coar[1](C)$ by defining:
\begin{itemize}
\item $\cdot_{\infty}$ and $1_\infty$ to be the multiplicative group structure of $\ltinf$;
\item $0_{\infty}$ to be $(0_{n})_{n\in\Nn_{>0}}$;
\item $x\Div_{\infty} y$ to hold if for some $n$,  $\pi_{1}(x)\Div_{1}\ltf[1](p^{-n})\pi_{1}(y)$ holds;
\item $x +_{\infty,\infty} y$ to be $(\pi_{mn}(x)+_{mn,m}\pi_{mn}(y))_{m\in\Nn_{>0}}$ if there exists $n\in\Nn_{>0}$ such that $\pi_{n}(x)+_{n,1}\pi_{n}(y) \neq 0_{1}$ and $0_{\infty}$ otherwise;
\item $x\DivR[\infty] y$ to hold if $\pi_{1}(x)\DivR[1]\pi_{1}(y)$ holds;
\item $E_{\infty}(x)$ to be $(E_{k}(x))_{k\in\Nn_{>0}}$ for all $E$ in some $\inv*{Ann}[m,n]$;
\item $\sigma_{\infty}$ to be $(\sigma_{n}(x))_{n\in\Nn_{>0}}$;
\end{itemize}
and we obtain a new functor $\Coar[2]:\Str(T)\to\Str(\Tinf_{2})$ where $\Def{\Tinf[2]}{\Tinf[1]\cup\TAs[0,0]<\infty>}$. One can check that we still have an equivalence of categories induced by $\Coar[2]$ and $\UCoar[1]$ and that $\Coar[2]$ also respects cardinality up to $\aleph_{0}$ and $\aleph_{1}$-saturated models. Finally, by Corollary\,\ref{cor:enrich functor}, as $T$ is Morleyized on $\lt$, we obtain functors $\Coar[3]:\Str(T)\to\Str(\Morl*{T_{2}}[(\ltinf\cup\lt)])$ and $\UCoar[3]:\Str(\Morl*{T_{2}}[(\ltinf\cup\lt)])\to\Str(T)$. Note that in this case, because we only enrich by predicates, the full subcategory $\Full$ of $\Str(T)$ is not actually needed.

Let us now show that for all $M\models T$, $\Coar[3](M)\models\TAsH[0,0]<\infty>$.

\begin{proposition}[lin app coar]
Let $M\models T$ and $t : \K<n>\to \K$ be an $\Sortrestr{\LAQ}{\K}(M)$-term, $\uple{d}\in\K(M)$ and $b$ an open $\Valinf$-ball. Then if, in $\Coar[3](M)$, $\uple{d}$ linearly approximates $t$ at prolongations on $b$, then for any open $\Val$-ball $b' \subseteq b$, $\uple{d}$ also linearly approximates $t$ at prolongations on $b'$ in $M$.
\end{proposition}

\begin{proof}
For all $a$ and $e\in b' \subseteq b$, we have \[\valinf(t(\sprol(a)) - t(\sprol(e)) - \uple{d}\cdot \sprol(a-e)) > \min_{i}\{\valinf(d_{i}) + \valinf(\sigma^{i}(a-e))\}.\] Let $i_{0}$ be such that $\valinf(d_{i_{0}}) + \valinf(\sigma^{i_{0}}(a-e))$ is minimal, then we have $\val(t(\sprol(a)) - t(\sprol(e)) - \uple{d}\cdot \sprol(a-e)) > \val(d_{i_{0}}) + \val(\sigma^{i_{0}}(a-e)) \geq \min_{i}\{\val(d_{i}) + \val(\sigma^{i}(a-e))\}$.
\end{proof}

\begin{proposition}[Hensel coar]
Let $M\models T$, then $\Coar[3](M)$ is $\sigma$-Henselian (for the valuation $\valinf$).
\end{proposition}

\begin{proof}
Let $(t,a,\uple{d},\xi)$ be in $\sigma$-Hensel configuration in $\Coar[3](M)$. Let $i_{0}$ be such that $\valinf(d_{i_{0}}) + \sigma^{i_{0}}(\xi)$ is minimal. As $(t,a,\uple{d},\xi)$ is in $\sigma$-Hensel configuration, $\val(t(\sprol(a))) > \valinf(d_{i_{0}}) + \sigma^{i_{0}}(\xi)$. Let $r = \sigma^{-i_{0}}(t(\sprol(a))d_{i_{0}}^{-1}p^{-1})$. Then \[\val(t(\sprol(a))) > \val(d_{i_{0}}) + \val(\sigma^{i_{0}}(r)) \geq \min_{i}\{\val(d_{i}) + \val(\sigma^{i}(r))\}.\] Moreover, \[\valinf(\sigma^{i_{0}}(r)) = \valinf(t(a)) - \valinf(d_{i_{0}}) > \sigma^{i_{0}}(\xi),\] i.e. $\valinf(r) > \xi$. It follows that $\oball[\Val]{\val(r)}{a} \subseteq \oball[\Valinf]{\xi}{a}$ and hence, by Proposition\,\ref{prop:lin app coar}, $\uple{d}$ linearly approximates $t$ at prolongations on $\oball[\Val]{\val(r)}{a}$ and $(t,a,\uple{d},\val(r))$ is in $\sigma$-Hensel configuration.

By $\sigma$-Henselianity of $M$, we can find $c\in\K(M)$ such that $t(\sprol(c)) = 0$ and $\val(c-a)\geq \max_{i}\{\val(\sigma^{-i}(t(\sprol(a))d_{i}^{-1}))\}$, and hence $\valinf(c-a) \geq \valinf(\sigma^{-i}(t(\sprol(a))d_{i}^{-1}))$ for all $i$.
\end{proof}

It follows from those two propositions that we can further enrich $\Morl*{T_{2}}[(\ltinf\cup\lt)]$ so that it is an $\lt$-enrichment of $\TAsH[0,0]<\infty>$. Hence we have proved:

\begin{proposition}[Coarsen TAsH]
Let $\LL$ be an $\lt$-enrichment of $\LAQs$ and $T$ be an $\LL$-theory containing $\TAsH[0,p]$ Morleyized on $\lt$. There exists an $\ltinf$-enrichment $\Linf$ of $\LAQs<\infty>$ --- with new sorts $\lt = \bigcup_{n}\lt[n]$ --- and an $\Linf$-theory $\Tinf\supseteq\TAsH[0,0]<\infty>$ Morleyized on $\ltinf\cup\lt$, and functors $\Coar:\Str(T)\to\Str(\Tinf)$ and $\UCoar:\Str(\Tinf)\to\Str(T)$ that respect cardinality up to $\aleph_{0}$ and induce an equivalence of categories between $\Str(T)$ and $\Str[\Coar,(\card{\LL}^{\aleph_{1}})^{+}](\Tinf)$ and such that $\UCoar$ respects models and elementary submodels and sends $\ltinf\cup\lt$ to $\lt$ and $\Coar$ respects $(\card{\Ann}^{\aleph_{1}})^{+}$-saturated models.
\end{proposition}

Similarly, we can prove the existence of these functors in the analytic and in the algebraic setting, and these functors are actually induced by those in the analytic difference case.

\begin{proposition}[Coar THen]
Let $\LL_{\ann}$ be any $\lt$-extension of $\LAQ$ contained in $\LL$ and $\LL_{\algop}$ be any $\lt$-extension of $\Lltplus$ contained in $\LL_{\ann}$. Define $\Def{T_{\ann}}{\Langrestr{T}{\LL_{\ann}}}$, and $\Def{T_{\algop}}{\Langrestr{T}{\LL_{\algop}}}$. Assume that both $T_{\ann}$ and $T_{\algop}$ are Morleyized on $\lt$.
\begin{thm@enum}
\item There exists an $\ltinf$-enrichment $\Linf[\ann]$ of $\LAQ<\infty>$ and an $\Linf[\ann]$-theory $\Tinf[\ann]\supseteq\TA[0,0]<\infty>$ Morleyized on $\ltinf\cup\lt$, and functors $\Coar[\ann]:\Str(T_{\ann})\to\Str(\Tinf[\ann])$ and $\UCoar[\ann]:\Str(T_{\ann})\to\Str(T_{\ann})$ with the same properties  as in Proposition\,\ref{prop:Coarsen TAsH}.

Moreover $\Coar[\ann](\Langrestr{\cdot\ }{\LL_{\ann}}) = \Langrestr{\Coar(\cdot)}{\Linf[\ann]}$ and similarly for $\UCoar[\ann]$.

\item There exists an $\ltinf$-enrichment $\Linf[\algop]$ of $\LL^{\ltinf^{+}}$ and an $\Linf[\algop]$-theory $\Tinf[\algop]\supseteq\THen[0,0]<\infty>$ Morleyized on $\ltinf\cup\lt$, and functors $\Coar[\algop]:\Str(T_{\algop})\to\Str(\Tinf[\algop])$ and $\UCoar[\algop]:\Str(\Tinf[\algop])\to\Str(T_{\algop})$ with the same properties  as in Proposition\,\ref{prop:Coarsen TAsH}.

Moreover $\Coar[\algop](\Langrestr{\cdot\ }{\LL_{\algop}}) = \Langrestr{\Coar[\ann](\Langrestr{\cdot\ }{\LL_{\ann}})}{\Linf[\algop]} = \Langrestr{\Coar(\cdot)}{\Linf[\algop]}$ and similarly for $\UCoar[\algop]$.
\end{thm@enum}
\end{proposition}

\section{Reduction to the algebraic case}\label{sec:alg}

In the following section, let $\LL_{\ann}$ be an $\lt$-enrichment of $\LAQ$ and let $T_{\ann}$ be an $\LL_{\ann}$-theory containing $\TA$, Morleyized on $\lt$.  We define $\Def{\LL_{\algop}}{\LL_{\ann}\sminus(\Ann\cup\{\Q\})}$ --- it is an $\lt$-enrichment of $\Lltplus$ --- and $T_{\algop} = \Langrestr{T_{\ann}}{\LL_{\algop}}$. As previously, if there are new sorts $\Sigma_{\lt}$, we write $\lt$ for $\lt\cup\Sigma_{\lt}$.

\begin{remark}[ext frac field ann]
Let $M_{1}$ and $M_{2}\models T_{\ann}$, $C_{i}\subs M_{i}$ and $f: C_{1}\to C_{2}$ an $\LL_{\ann}$-isomorphism. Obviously, $f$ extends uniquely to $\genan{C_{1}}$. As $\LL_{\ann}$ contains $\Q$, $\K(\genan{C_{1}})$ is a field. Hence any partial $\LL_{\ann}$-isomorphism with domain $C$ has a unique extension to $\Frac{\K(C)}$.
\end{remark}

Although it is well-known, the algebraic case (i.e. in $\LL_{\algop}$) is a bit more complicated because we do not have $\Q$ in $\LL_{\algop}$.

\begin{proposition}[ext frac field alg]
Let $M_{1}$ and $M_{2}\models T_{\algop}$ be two $\LL_{\algop}$-structures, $C_{i}\subs M_{i}$ and $f: C_{1}\to C_{2}$ be an $\Lltplus$-isomorphism. If $\ltf(\Frac{\K(C_{1})})\subseteq \lt(C_{1})$ then $f$ has a unique extension to $\Frac{\K(C_{1})}$.
\end{proposition}

\begin{proof}
Let $\Sortrestr{f'}{\K}$ be the unique extension of $\Sortrestr{f}{\K}$ to $\Frac{\K(C_{1})}$. It is a ring morphism. By Lemma\,\ref{lem:field ext}, it suffices to show that $\Sortrestr{f'}{\K}\cup\Sortrestr{f}{\lt}$ respects the $\ltf[n]$. As $\ltf(\Frac{\K(C_{1})})\subseteq \lt(C_{1})$, $\Sortrestr{f}{\lt}$ commutes with the inverse on any $\ltf[n]$ and hence
\[\ltf[n](f'(a/b)) = \ltf[n](f(a)f(b)^{-1}) = f(\ltf[n](a))f(\ltf[n](b)^{-1}) = f(\ltf[n](a/b)).\]

This concludes the proof.
\end{proof}

In the following proposition we will be working in equicharacteristic zero, hence, to avoid needlessly cluttered notation, we will write $\res$, $\resf$, $\lt$ and $\ltf$ for $\res[1]$, $\resf[1]$, $\lt[1]$ and $\ltf[1]$.

\begin{proposition}[red alg](Reduction to the algebraic case)
Suppose $T_{\ann}\supseteq \TA[0,0]$. Let $M_{1}$ and $M_{2}\models T_{\ann}$, $f : M_{1}\to M_{2}$ be a partial $\LL_{\ann}$-isomorphism with domain $C_{1}\substr M_{1}$ and $a_{1}\in M_{1}$. If $f$ can be extended to an $\LL_{\algop}$-isomorphism $f'$ whose domain contains $a_{1}$, then $f$ can be extended to an $\LL_{\ann}$-isomorphism sending $a_{1}$ to $f'(a_{1})$.
\end{proposition}

\begin{proof}
First, because $\Sortrestr{T_{\algop}}{\lt} = \Sortrestr{T_{\ann}}{\lt}$, $T_{\algop}$ is also Morleyized on $\lt$. By Lemma\,\ref{lem:ext on lt}, we can extend $f'$ on $\lt$ and we may assume that $\lt(\genan[C_{1}]{a_{1}})\subseteq \lt(C_{1})$. Moreover, as $f'$ respects $\DivR[1]$, $f'$ respects $\ValR$ and by Remark\,\ref{rem:ext frac field ann} and Proposition\,\ref{prop:ext frac field alg}, replacing, if need be, $a_{1}$ by its inverse, we can assume that $a_{1}\in\ValR$.

Let $a_{2} = f'(a_{1})$ and let us define $f''$ on $\K(\genan{C_1}{a_1})$ by $f''(t(a_1)) = t^f(a_2)$, where $t^{f}$ is the term obtained by applying $f$ to the parameters of $t$. This morphism $f''$ clearly coinciding with $f'$ on $\K(C_1)[a_1]$ and it is well defined. Indeed, it suffices to check that if $t(a_1) = 0$ then $t^f(a_2) = 0$. But, by Weierstrass preparation, there exists $S\in\SWC<\ValR>(C_{1})$, an $\Sortrestr{\LAQ}{\K}(C_{1})$ term $E$ (a strong unit on $S$) and $P$, $Q\in \K(C_{1})[X]$ such that $Q$ does not have any zero in $S(\alg{\K(C_{1})})$, $a_1\in S$ and for all $x\in S$, $t(x) = E(x)P(x)/Q(x)$. As $t(a_1) = 0$ and $E(x) \neq 0$, we must have $P(a_1) = 0$. As $f'$ is a partial $\LL_{\algop}$-isomorphism, we have $a_{2}\in S^{f}$ and $P^f(a_{2}) = 0$. As $f$ is an $\LL_{\ann}$-isomorphism, by Theorem\,\ref{thm:EQ TA} it is in fact an elementary partial $\LL_{\ann}$-isomorphism and we also have that for all $x\in S^{f}$, $t^{f}(x) = E^{f}(x)P^{f}(x)/Q^{f}(x)$ and $E^{f}$ is a strong unit on $S^{f}$. Hence, $t^f(a_2) = E^f(a_{2})P^f(a_{2})/Q^f(a_2) = 0$.

Let us show that $f''\cup\Sortrestr{f}{\lt}$ is an $\LAQ$-isomorphism. By Lemma\,\ref{lem:field ext}, it suffices to show that for all $\Sortrestr{\LAQ}{\K}(C_{1})$-terms $t$, $\ltf(t^{f}(a_{2})) = f(\ltf(t(a_{1})))$. By Remark\,\ref{rem:decomp form Kterm}, $S$ is defined by a formula of the form $\theta(\ltf(\uple{R}(x)))$ where $\theta$ is an $\Sortrestr{\LL_{\algop}}{\lt}$-formula and the $R_{i}$ are polynomials in $\K(C_{1})[X]$. By \cite[proof of Theorem\,7.5]{CluLo_Bmin}, there exists an $\LL_{\algop}(C_{1})$-definable function $g : K \to \prod_{i} \lt[n_{i}]$ such that every fiber is an open $\Val$-ball and for any polynomial $T$ equal to $P$, $Q$ or one of the $R_{i}$, $\ltf(T(x))$ is constant on any fiber of $g$. It follows immediately that every fiber of $g$ is either in $S$ or in its complement. Let $\uple{\alpha} = g(a_{1})$ and $\beta = \ltf(t(a_{1}))$. As $E$ is a strong unit, on $g^{-1}(\uple{\alpha}) = \oball{\val(d)}{c}$ it is of the form $eF((x-c)/d)$ with $\val(F((x-c)/d)) = 0$. As $\resf((x-c)/d) = 0$ on all of $g^{-1}(\uple{\alpha})$, by Corollary\,\ref{cor:ltf unit}, $\ltf(E(x))$ is constant on $g^{-1}(\uple{\alpha})$, and hence $\ltf(t(x))$ is constant on $g^{-1}(\uple{\alpha})$. As $f$ is a partial elementary $\LL_{\ann}$-isomorphism and $\uple{\alpha}$ and $\beta \in \lt(C_{1})$, the $\LAQ(C_{1})$-formula $\forall x, g(x) = \uple{\alpha} \imp \ltf(t(x)) = \beta$ is preserved by $f$. And as $f'$ is a partial elementary $\LL_{\algop}$-isomorphism (by Theorem\,\ref{thm:EQ THen}) and $g$ is $\LL_{\algop}(C_{1})$-definable, $g^{f}(a_{2}) = f(\uple{\alpha})$ and we have that $\ltf(t^{f}(a_{2})) = f(\beta) = f(\ltf(t(a_{1})))$.
\end{proof}

\begin{corollary}[red alg cor]
The previous proposition holds without any assumption on residue characteristic.
\end{corollary}

\begin{proof}
Recall Proposition\,\ref{prop:Coar THen} and assume $M_{1}$ and $M_{2}$ have mixed characteristic and $f$ and $f'$ are as in Proposition\,\ref{prop:red alg}.

Then $\Coar[\algop](f')$ is an extension of $\Coar[\ann](f)$ whose domain contains $a_{1}$. By Proposition\,\ref{prop:red alg}, we obtain $f''$ an $\Linf[\ann]$-isomorphism extending $\Coar[\ann](f)$ whose domain contains $a_{1}$ and we conclude by applying $\UCoar[\ann]$.
\end{proof}

\begin{corollary}[algebraization]
Let $\phi(x,\uple{y},\uple{r})$ be any $\LL_{\ann}$-formula where $x$ and $\uple{y}$ are $\K$-variables and $\uple{r}$ are $\lt\cup\Sigma_{\lt}$-variables, then there exists a $\K$-quantifier free $\LL_{\algop}$-formula $\psi(x,\uple{z},\uple{r})$ and $\Sortrestr{\LL_{\ann}}{\K}$-terms $\uple{u}(\uple{y})$ such that $T_{\ann}\models \phi(x,\uple{y},\uple{r})\iff\psi(x,\uple{u}(\uple{y}),\uple{r})$.
\end{corollary}

\begin{proof}
This follows from the previous corollary by a (classic) compactness argument. For the sake of completeness (and also because the uniformization part of that argument may be less usual), let us state it. Consider the set of formulas
\[\begin{array}{c}T_{\ann}\cup\{\phi(x_{1},\uple{y},\uple{r}),\neg\phi(x_{2},\uple{y},\uple{r})\}\cup\\
\{\psi(x_{1},\uple{u}(\uple{y}),\uple{r})\iff\psi(x_{2},\uple{u}(\uple{y}),\uple{r})\mid \psi\text{ is an }\LL_{\algop}\text{-formula and }\uple{u}\text{ are }\Sortrestr{\LAQ}{\K}\text{-terms}\}.
\end{array}\]
By Corollary\,\ref{cor:red alg cor}, this set of formulas cannot be consistent. Hence there is a finite set of $\LL_{\algop}$-formulas $(\psi_{i})_{0\leq i < n}$ --- that we can take $\K$-quantifier free by Theorem\,\ref{thm:EQ THen} --- and $\Sortrestr{\LAQ}{\K}$-terms $\uple{u}_{i}$ such that: \[T_{\ann}\models \forall\uple{y}x_{1}x_{2}(\bigwedge_{i}\psi_{i}(x_{1},\uple{u}_{i}(\uple{y}),\uple{r})\iff\psi_{i}(x_{2},\uple{u}_{i}(\uple{y}),\uple{r}))\imp (\phi(x_{1},\uple{y},\uple{r})\iff \phi(x_{2},\uple{y},\uple{r})).\]
For all $\epsilon \in 2^{n}$, let $\Def{\theta_{\epsilon}}{\bigwedge \psi_{i}(x,\uple{u}_{i}(\uple{y}),\uple{r})^{\epsilon(i)}}$ where $\psi^{1} = \psi$ and $\psi^{0} = \neg\psi$. For fixed $\uple{y}$ and $\uple{r}$, the $\theta_{\epsilon}(x,\uple{y},\uple{r})$ form a partition of $\K$ compatible with $\phi(x,\uple{y},\uple{r})$. For all $\eta\in2^{2^{n}}$, let $\chi_{\eta}(\uple{y},\uple{r})$ be a $\K$-quantifier free $\LL_{\ann}$-formula equivalent to $\bigwedge_{\epsilon}(\exists x\,\theta_{\epsilon}(x,\uple{y},\uple{r})\wedge\phi(x,\uple{y},\uple{r}))^{\eta(\epsilon)}$. Note that for any choice of $\uple{y}$ and $\uple{r}$ there is exactly one $\eta$ such that $\chi_{\eta}(\uple{y},\uple{r})$ holds. It is now quite easy to show that $\phi(x,\uple{y},\uple{r})\iff \bigvee_{\eta}(\chi_{\eta}(\uple{y},\uple{r})\wedge\bigvee_{\epsilon\in\eta}\theta_{\epsilon}(x,\uple{y},\uple{r}))$.
\end{proof}

\begin{remark}
\begin{thm@enum}
\item This corollary is a stronger version of \cite[Theorem\,B]{vdDHM}. Not only is it resplendent but it also has better control of the parameters (essentially due to a better control of the parameters in Weierstrass preparation in \cite{CluLipAnn}). In particular, it is uniform.
\item \label{rem:red alg ac} Let $\LAQac$ be $\Lac$ enriched with symbols for all the functions from $\Ann$, a symbol $\Q:\K<2>\to K$, for all units $E\in\Ann$ a symbol $E_{k}:\res[k]\to\res[k]$, a symbol $\DivR\subseteq(\Valgpinf)^{2}$. Then, any $\LAQac$-formula (or even formulas in an $\res\cup\Valgp$-enrichment of $\LAQac$) can be translated into an $\lt$-enrichment of $\LAQ$ (see Proposition\,\ref{prop:equiv Lsec Lac}), and hence Corollary\,\ref{cor:algebraization} also holds (resplendently) for the $\LAQac$-theory $\TAac$ of Henselian valued fields with separated $\Ann$-structure and angular components. Note that some of the symbols we should have added have disappeared, like the trace of $E_{k}$ on $\Valgpinf$ which is constant equal to $0$. Similarly the $E_{k}$ and $\DivR[1]$ are missing one of their arguments --- the $\Valgpinf$-argument in the case of $E_{k}$ and the $\res[n]$-argument for $\DivR$ --- but they depend trivially on it.
\end{thm@enum}
\end{remark}

\section{\texorpdfstring{$\K$}{K}-quantifier elimination in \texorpdfstring{$\TAsH$}{TAsHen}}\label{sec:EQ}

Up to Section\,\ref{subsec:EQ}, we will be working solely in equicharacteristic zero, hence, we will once again write $\lt$ and $\ltf$ for $\lt[1]$ and $\ltf[1]$. We will also be considering that variables are indexed by $\Nn$ and we will sometimes identify a variable and its index. But hopefully no confusion should arise.

Let $M\models\TA$ and $C\substr M$.

\begin{definition}(Order-degree)
We say that an $\Sortrestr{\LAQ}{\K}(C)$-term $t = \sum_{i=0}^{d}t_{i}(\upleneq{x}{m})x_{m}^{i}$ is polynomial of order (at most) $d$ in $x_{m}$. If $t$ is not of this form, we take the convention that $t$ has infinite degree in $x_{m}$. Let $\Term{C}$ be the set of tuples $(t,I,m,d)$ where $I$ is a finite set of variables, $m\in I$, $d\in\Nn\cup\{\infty\}$ and $t\neq 0$ is an $\Sortrestr{\LAQ}{\K}(C)$-term whose variables are contained in $I$ and which is either polynomial in $x_{m}$ of degree at most $d$ or not polynomial in $x_{m}$ (and $d = \infty$). Let $\Term[0]{C} = \Term{C}\cup\{0\}$.

We (partially) order $\Term{C}$ by saying that $(u,J,n,e)$ has lower order-degree than $(t,I,m,d)$ if one of the following holds:
\begin{thm@enum}
\item $\max(J) < \max(I)$;
\item $\max(J) = \max(I)$ and $J \subset I$, i.e. $J$ is \emph{strictly} included in $I$;
\item\label{o-d} $J = I$ and $n > m$;
\item $J= I$ and $n = m$ and $e<d$.
\end{thm@enum}

We extend this order to $\Term[0]{C}$ by making the zero term greater than any element of $\Term{C}$.
\end{definition}

\begin{remark}
\begin{thm@enum}
\item This is a well-founded (partial) order.
\item In condition (iii), the order is inverse of what one would expect but that is because we want minimal terms to be polynomial in the last variable.
\item\label{rem:order I} We will also write $J< I$ to mean that conditions (i) or (ii) hold.
\end{thm@enum}
\end{remark}

When $\uple{a}$ is indexed by some set $I\subseteq\Nn$ and $n\in I$, we will denote by $\upleneq{a}{n}$ the tuple $\uple{a}$ missing its $n$-th component and $(\upleto{a}{x}{n})$ for the tuple $\uple{a}$ where the $n$-th component is replaced by $x_{n}$. We define $\sprolneq{a}{n}$ and $(\sprolto{a}{x}{n})$ similarly and let $\Def{\sprolleq{a}{n}}{(a,\sigma(a),\ldots,\sigma^{n}(a))}$. Finally, we will write $\Def{\genans{C}}{\gen{\LAQs}{C}}$ and $\Def{\genans[C]{\uple{c}}}{\gen{\LAQs}[C]{\uple{c}}}$ (cf. Definition\,\ref{def:gen}).

\subsection{Residual and ramified extensions}\label{subsec:res}

\begin{definition}(Regularity)
Let $t(\uple{x}) = \sum_{i}t_{i}(\upleneq{x}{m})x_{m}^{i}$ be an $\LAQ(M)$ term and $\uple{a}\in\K(M)$. We say that $t$ is regular at $\uple{a}$ in $x_{m}$ if  \[\val(t(\uple{a})) = \min_{i}\{\val(t_{i}(\upleneq{a}{m})) + i\val(a_{m})\}.\]
By convention the zero term is never regular.
\end{definition}

First, we state a proposition which has nothing to do with automorphisms:

\begin{proposition}[descr term lt]
Let $\uple{\alpha}\in\ltf(\ValR(M))$, $\uple{a}\in\ltf<-1>(\uple{\alpha})$ and $(t,I,m,d)\in\Term[0]{C}$ be of minimal order-degree such that $t(\uple{x})$ is polynomial in $x_{m}$ and $t$ is not regular at $\uple{a}$.
Then for all $(u,J,n,e) < (t,I,m,d)$, $\ltf(u(\uple{x}))$ is constant on $\ltf<-1>(\uple{\alpha})$. Moreover for all $\upleneq{a}{n}\in\ltf<-1>(\upleneq{\alpha}{n})$, $u(\upleto{a}{x}{n})$ has Weierstrass division on $\ltf<-1>(\alpha_{n})$.
\end{proposition}

\begin{proof} First, we may assume that $\alg{\K(M)} = \K(M)$ (see Proposition\,\ref{prop:TA alg ext}). We work by induction on $J$ for the order defined in Remark\,\ref{rem:order I}. The proposition is trivial for constant terms. Now, assume the proposition is true for any $(v,K,p,f)$ with $K < J$.

Let us first assume that $u$ is polynomial in $x_{n}$. Then, $u = \sum_{i}u_{i}(\upleneq{x}{n})x_{n}^{i}$ must be regular at $\uple{a}$ and hence $\val(u(\uple{a})) = \min_{i}\{\val(u_{i}(\upleneq{a}{n})) + i\val(a_{n})\}$ and $\ltf(u(\uple{a})) = \sum_{i}\ltf(u_{i}(\upleneq{a}{n}))\alpha_{n}^{i}\neq 0$. For any $\uple{e}\in\ltf<-1>(\uple{\alpha})$ and $i$, $\ltf(u_{i}(\upleneq{e}{n}))\ltf(e_{n})^{i} = \ltf(u_{i}(\upleneq{a}{n}))\alpha_{n}^{i}$. Moreover, if $\sum_{i}\ltf(c_{i}) \neq 0$ then $\ltf(\sum_{i}c_{i}) = \sum_{i}\ltf(c_{i})$ hence we must also have $\ltf(u(\uple{e})) = \sum_{i}\ltf(u_{i}(\upleneq{a}{n}))\alpha_{n}^{i}\neq 0$. As $u$ is polynomial in $x_{n}$, it has a Weierstrass preparation. Hence for polynomial $u$, the proposition is proved.

Suppose now that $u$ is of infinite degree in $x_{n}$ and hence that all terms $(v,J,n,e)$ with $e\neq\infty$ have been taken care of in the previous paragraph. By Weierstrass preparation, there exists $S\in\SWC<\ValR>(\genan[C]{\upleneq{a}{n}})$ such that $u(\upleto{a}{x}{n})$ has a Weierstrass preparation on $S$ and $a_{n}\in S$. But then either $\ltf<-1>(\alpha_{n})\subseteq S$ or $\ltf<-1>(\alpha_{n})$ contains a $\alg{\K(\genan[C]{\upleneq{a}{n}})}$-ball and hence a point $c\in\alg{\K(\genan[C]{\upleneq{a}{n}})}$. Let $P = \sum p_{i}(\upleneq{a}{n})X^{i}\in\K(\genan[C]{\upleneq{a}{n}})[X]$ be its minimal polynomial, then for all $e\in\ltf<-1>(\alpha_{n})$, $\ltf(P(e)) = 0$ --- i.e. $P(e) = 0$ --- but that is absurd. Hence $t$ has a Weierstrass preparation on $\ltf<-1>(\alpha_{n})$ and there exists $F(x,\uple{z})\in\Ann$, $\uple{c}\in\K(\genan[C]{\uple{a}})$, $P$ and $Q\in\K(\genan[C]{\upleneq{a}{n}})[X]$ such that for all $x_{n}\in\ltf<-1>(\alpha_{n})$:
\[u(\upleto{a}{x}{n}) = F\left(\frac{x_{n}-a_{n}}{a_{n}},\uple{c}\right)\frac{P(x_{n})}{Q(x_{n})}\]
and $\val(F((x_{n}-a_{n})/a_{n},\uple{c})) = 0$. But $\ltf(P(x_{n}))$ and $\ltf(Q(x_{n}))$ do not depend on $x_{n}$ and $\ltf(F((x_{n}-a_{n})/a_{n},\uple{c}))$ only depends on $\resf((x_{n}-a_{n})/a_{n}) = 0$ (see Corollary\,\ref{cor:ltf unit}). Hence $\Def{\beta}{\ltf(u(\upleto{a}{x}{n}))}$ does not depend on $x_{n}\in\ltf<-1>(\alpha_{n})$. The $\LAQ$-formula \[\forall x_{n}\,\ltf(x_{n}) = \alpha_{n}\imp\ltf(u(\upleto{a}{x}{n})) = \beta\] is in the $\LAQ$-type of $\upleneq{a}{n}$ over $C\alpha_{n}\beta$. By induction (and Corollary\,\ref{cor:uniformization}), all tuples $\upleneq{a}{n}\in\ltf<-1>(\upleneq{\alpha}{n})$ have the same $\LAQ(C\alpha\beta)$-type and $\ltf(u(\uple{a})) = \beta$ for all $\uple{a}\in\ltf<-1>(\uple{\alpha})$.
\end{proof}

Let us now prove the first embedding theorem we will need to eliminate quantifiers. Let $M_{1}$ and $M_{2}$ be models of $\TAsH$, $C_{i}\substr M_{i}$ and $f:C_{1}\to C_{2}$ an $\Morl{LAQs}[\lt]$-isomorphism.

\begin{proposition}[lt ext]
Let $\alpha\in \ltf(\ValR(M_{1}))\cap\lt(C_{1})$, $a\in\ltf<-1>(\alpha)$ and $(t,I,m,d)\in\Term[0]{C}$ be polynomial in $x_{m}$ for some $m\in\Nn$. Assume that $(t,I,m,d)$ is of minimal order-degree such that $t$ is not regular at $\sprol(a)$. Then:
\begin{thm@enum}
\item There exists $a_1\in \ValR(M_{1})$ and $a_2\in\ValR(M_{2})$ such that $t(\sprol(a_1))=0 = t^f(\sprol(a_2))$, $\ltf(a_1) = \alpha$ and $\ltf(a_{2}) = f(\alpha)$.
\item For any such $a_{i}$, $f$ can be extended to an $\Morl{LAQs}[\lt]$-isomorphism sending $a_{1}$ to $a_{2}$.
\end{thm@enum}
\end{proposition}

\begin{proof}
Let $t = \sum_{i=0}^{d}t_{i}(\upleneq{x}{m})x_{m}^{i}$. By minimality of $t$, we cannot have $t_{d}(\sprolneq{a}{m}) = 0$. Dividing by $t_{d}$, we may assume that $t_{d} = 1$.

\begin{claim}
There exists $\uple{c}\in\K(M)$ that linearly approximates $t$ on $\ltf<-1>(\alpha)$ at prolongations and such that \[\min_{j}\{\val(c_{j}) + \val(\sigma^{j}(a))\} = \min_{i}\{\val(t_{i}(\sprol(a))) + i\val(\sigma^{m}(a))\}.\]
\end{claim}

\begin{proof} Let $N_{i} = \alg{M_{i}}$ (see Proposition\,\ref{prop:TA alg ext}). Let $\Def{s_{e}}{\sum_{i<e}t_{i}x_{m}^{i}}$ and $\Def{s}{s_{d}}$. For all $i$ and $j\neq m$, by Proposition\,\ref{prop:descr term lt} applied in $N_{1}$, $t_{i}(\sprolto{a}{x}{j})$ has Weiestrass preparation on the ball $\Def{b_{j}}{\ltf<-1>(\alpha_{j})}$ and constant valuation. By Proposition\,\ref{prop:descr term lt}, for all $e\leq d$, $s_{e}$ also has constant valuation on $\ltf<-1>(\sprol(\alpha))$. By invariance under addition --- and an induction on $e$ --- we can show that $s(\sprolto{a}{x}{j})$ also has Weiestrass preparation on $b_{j}$. Moreover $\diff[j]{s}{\uple{x}}$ is also given by an $\LAQ(C_{1})$-term of degree $d-1$ in $x_{m}$ hence $\ltf(\diff[j]{s}{\sprolto{a}{x}{j}})$ is constant on $b_{j}$ (equal to some $\ltf(c_{j})$, where $c_{j}\in\K(M_{1})$). By Proposition\,\ref{prop:lin approx}, for all $y_{j}$ and $z_{j}\in b_{j}$:
\[\ltf(t(\sprolto{a}{y}{j})- t(\sprolto{a}{z}{j})) = \ltf(s(\sprolto{a}{y}{j})- s(\sprolto{a}{z}{j})) = \ltf(c_{j})\ltf(y-z).\]
This last statement is in the $\LAQ$-type of $\sprolneq{a}{j}$ over $C_{1}\ltf(c_{j})$. By Proposition\,\ref{prop:descr term lt} and Corollary\,\ref{cor:uniformization}, any $\upleneq{e}{j}\in\ltf<-1>(\sprolneq{\alpha}{j})$ has the same $\LAQ(C_{1}\ltf(c_{j}))$-type and hence the same $c_{j}$ works for any $\uple{e}\in\ltf<-1>(\sprol(\alpha))$.

By minimality of $t$, $s$ is regular at $\sprol(a)$ and \[\val(s(\sprol(a))) = \min_{i<d}\{\val(t_{i}(\sprolneq{a}{j})) + i\val(\sigma^{m}(a))\}.\] Because $\val(s(\sprolto{a}{x}{j}))$ is constant on $b_{j}$, by the last statement of Proposition\,\ref{prop:lin approx} and because $\rad(b_{j}) = \ltf(\sigma^{j}(a))$, we obtain that: \[\val(c_{j}) + \val(\sigma^{j}(a)) \geq \val(s(\sprol(a))) \geq  \min_{i}\{\val(t_{i}(\sprol(a))) + i\val(\sigma^{m}(a))\}.\] 

When $j=m$, as $t$ is polynomial in $x_{m}$ and $\diff[m]{t}{\uple{x}}$ is of degree $d-1$ in $x_{m}$, by Proposition\,\ref{prop:lin approx}, we also find $c_{m}\in\K(M_{1})$ that linearly approximates $t(\upleto{e}{x}{m})$ on $\ltf<-1>(\sigma^{m}(\alpha))$ for any $\uple{e}\in\ltf<-1>(\sprol(\alpha))$. And \[\val(c_{m}) + \val(\sigma^{m}(a)) = \val(\diff[m]{t}{\uple{x}}) + \val(\sigma^{m}(a)) = \min_{i}\{\val(t_{i}(\sprol(a))) + i\val(\sigma^{m}(a))\}.\]

It now follows from Proposition\,\ref{prop:one var to all} that $\uple{c}$ linearly approximates $t$ on $\ltf<-1>(\alpha)$ at prolongations.
\end{proof}

If $t\neq 0$, as $t$ is not regular at $\sprol(a)$, $\val(t(\sprol(a))) > \min_{i}\{\val(t_{i}(\sprol(a))) + i\val(\sigma^{m}(a))\} = \val(c_{m}) + \val(\sigma^{m}(a)) = \min_{j}\{\val(c_{j}) + \val(\sigma^{j}(a))\}$, and $(t,a,\uple{c},\vallt(\alpha))$ is in $\sigma$-Hensel configuration. Hence there exists $a_{1}\in M_{1}$ such that $\val(a_{1}-a) \geq \max_{i}\{\sigma^{-i}(t(\sprol(a)c_{i}^{-1}))\}$ and $t(a_{1}) = 0$. In particular, 
\begin{eqnarray*}
\lefteqn{\val(\sigma^{m}(a_{1}-a))\geq\val(t(\sprol(a))) - \val(c_{m})}\\
&>&
\min_{i}\{\val(t_{i}(\sprol(a))) + i\val(\sigma^{m}(a))\} - \val(c_{m})\\
&=& \val(\sigma^{m}(a)),
\end{eqnarray*}

and thus $\ltf(a_{1}) = \ltf(a)$.

If $x_{m}$ is not the highest variable appearing in $t$ --- that we call $x_{n}$ --- then, applying Proposition\,\ref{prop:descr term lt} to $(t,I,n,\infty) < (t,I,m,d)$, we get that $\ltf(t(\uple{x}))$ is constant equal to $0$ on all of $\ltf<-1>(\sprol(\alpha))$. As $t(\sprolto{a}{x}{m})$ is polynomial and has infinitely many zeros, we must have $t_{i}(\sprol(a)) = 0$ for all $i$, but that contradicts the non-regularity of $t$ in $x_{m}$ at $\sprol(a)$. Thus $x_{m}$ must be the highest variable appearing in $t$. For the same reasons, we cannot have $\ltf(c_{m}) = 0$.

Note that we have also proved that for all $\uple{e}\in\ltf<-1>(\sprol(\alpha))$, $t$ is minimal such that it is not regular in $\uple{e}$, hence the $\LAQ$-type of $\sprol(\alpha)$ says so and hence, as $\TA$ eliminates field quantifiers, $t^{f}$ has the same minimality property (relatively to $f(\alpha)$) and we find $a_{2}$ in the exact same way. If $t=0$ then any $a_{1}$ and $a_{2}\in\ltf<-1>(\alpha)$ will work.

Let us now show that $f$ can be extended to send $a_{1}$ to $a_2$. For all $n < m$, any term $u(\upleleq{x}{n})\in\Term[0]{C}$ has order-degree strictly smaller than $(t,I,m,d)$ and hence $\Def{\beta}{\ltf(u(\uple{e}))}$ does not depend on the choice of $e\in\ltf^{-1}(\sprol(\alpha))$. The formula ``$\forall \uple{x}\,\ltf(\uple{x}) = \sprol(\alpha)\imp \ltf(u(\uple{x})) = \beta$'' is an $\LAQ(C_{1})$-formula respected by $f$ and thus $f(\beta) = f(\ltf(u(\sprol(a_{1})))) = u(\sprol(a_{2}))$. It follows immediately that the function $f_{n}$ sending $u(\sprol(a_{1}))$ to $u(\sprol(a_{2}))$ is well defined. It is obviously an $\Langrestr{\LAQ}{\K}$-morphism, and, as it respects $\ltf$, by Lemma\,\ref{lem:field ext}, it is an $\LAQ$-morphism.

Let us now assume that $t\neq 0$.

\begin{claim}
Let $\Def{P}{t(\sprolto{a_{1}}{x}{m})}\in\K(C_{1,m-1})[X]$. For all $n \geq m$, $\sigma^{n}(a_{1})$ is the only zero of $P^{\sigma^{n-m}}$ whose leading term is $\sigma^{n}(\alpha)$. 
\end{claim}

\begin{proof}
Because $\sigma$ is an automorphism of valued fields, it suffices to prove the case $n=m$. Let $e\in\ltf^{-1}(\sigma^{m}(\alpha))\sminus\{\sigma^{m}(a_{1})\}$, then $\ltf(P(e)) = \ltf(P(e) - P(a)) = \ltf(c_{m})\ltf(e-\sigma^{m}(a_{1}))\neq 0$.
\end{proof}

The same claim is true of $\sigma^{m}(a_{2})$ with respect to $f(P^{\sigma^{n-m}})$ and $f(\sigma^{n}(\alpha))$. Therefore, it suffices to extend $f_{m-1}$ to the $\LAQ$-definable closure of $C_{1,m-1}$, which we can certainly do as $f_{m-1}$ is an $\LAQ$-elementary isomorphism (by resplendent field quantifier elimination in $\TA$).

Thus we have obtained an $\LAQ$-isomorphism between $\genan[C_{1}]{\sprol(a_{1})}$ and $\genan[C_{2}]{\sprol(a_{2})}$. By Remark\,\ref{rem:sigma terms}, it is an $\LAQs$-isomorphism. This morphism is also an $\Morl{LAQs}[\lt]$-isomorphism by Lemma\,\ref{lem:field ext}.
\end{proof}

\begin{corollary}[lt ext cor]
Let $\alpha\in\lt(C_{1})$, then there exists $a_{1}\in M_{1}$ such that $\ltf(a_{1}) = \alpha$ and $f$ extends to an isomorphism on $\genans[C_{1}]{a_{1}}$.
\end{corollary}

\begin{proof}
If $\alpha\in\resf(\ValR(M_{1}))$, then Proposition\,\ref{prop:lt ext} applies. If not apply Proposition\,\ref{prop:lt ext} to $\alpha^{-1}$ and conclude by extending the isomorphism to the analytic field generated by its domain by Remark\,\ref{rem:ext frac field ann}.
\end{proof}

\subsection{Immediate extensions}\label{subsec:imm}

Let $M\models\TA$ be sufficiently saturated and $C\substr M$.

\begin{definition}(Pseudo-convergent $\star$-sequences)
Let $(\uple{x}_{\alpha})$ be a sequence of tuples of the same length. We say that it is a pseudo-convergent sequence if for all $i$, $(\uple{x}_{i,\alpha})$ is pseudo-convergent. Moreover, we will say that $\uple{a}$ is a pseudo-limit of $(\uple{x}_{\alpha})$ if for all $i$, $\uple{x}_{i,\alpha}\pc a_{i}$.
\end{definition}

\begin{definition}(Equivalent pseudo-convergent sequences)
We will say that two pseudo-convergent sequences are equivalent if they have the same pseudo-limits.
\end{definition}

\begin{lemma}[equiv pc seq]
Let $\uple{x}_{\alpha}$ be a pseudo-convergent sequence, $\uple{a}$ a pseudo-limit of this sequence and $\uple{y}_{\alpha}$ a sequence such that for all $i$, $\val(a_{i}-y_{i,\alpha}) = \val(a_{i}-x_{i,\alpha})$, then $(y_{\alpha})$ is also a pseudo-convergent sequence, equivalent to $(x_{\alpha})$.
\end{lemma}

\begin{proof}
We may assume that $\card{x_{\alpha}} = 1$. Note that for all $\beta > \alpha$, $\val(y_{\beta}-y_{\alpha}) = \val(y_{\beta} - a + a - y_{\alpha}) = \val(a-x_{\alpha}) = \val(x_{\beta}-x_{\alpha})$, as $\val(a-x_{\beta}) > \val(a-x_{\alpha})$. Hence $(y_{\alpha})$ is also pseudo-convergent. Moreover, if $b$ is any pseudo-limit of $(x_{\alpha})$, then $\val(b-y_{\alpha}) = \val(b-x_{\alpha+1}+x_{\alpha+1}-a+a-y_{\alpha}) = \val(a-y_{\alpha}) = \val(a-x_{\alpha}) = \val(b-x_{\alpha})$ and $y_{\alpha}\pc b$. The symmetric argument shows that if $y_{\alpha}\pc b$ then $x_{\alpha}\pc b$.
\end{proof}

\begin{definition}(Rich enough families)
We say that a family $\cF$ of equivalent pseudo-convergent sequences of $C$ is rich enough if for any linear polynomial $P(\uple{X}) = \sum_{i}\pi_{i}X_{i}\in\ltf(\K(C))[\uple{X}]$, there exists $(\uple{x}_{\alpha})\in\cF$ such that for all pseudo-limit $\uple{a}$ and all $\alpha$, $P(\ltf(\uple{a}-\uple{x}_{\alpha})) \neq 0$, i.e. if $\ltf(p_{i}) = \pi_{i}$ then $\val(\sum_{i}p_{i}(a_{i} - x_{i,\alpha})) = \min_{i}\{\val(p_{i}) + \val(a_{i}-x_{i,\alpha})\}$.
\end{definition}

We will say that a term $u = \sum_{i=0}^{d}u_{i}(\upleneq{x}{m})\sigma^{m}(x)^{i}$ is monic if $u_{d} = 1$. As in section \ref{subsec:res}, let us begin by a proposition that does not seem to have anything to do with automorphisms.

\begin{proposition}[descr term imm]
Let $\cF$ be a rich enough family of equivalent pseudo-convergent sequences of $C$ that are eventually in $\ValR$ and $(t,I,m,d)\in\Term[0]{C}$. Suppose that $(t,I,m,d)$ has minimal order-degree such that $t$ is a monic polynomial in $x_{m}$ and there exists a pseudo-convergent sequence $(\uple{x}_{\alpha})\in\cF$ that pseudo-solves $t$. Then for all $(u,J,n,e) < (t,I,m,d)$, there exists $\alpha_{0}$ such $\ltf(u(\uple{x}))$ is constant on $\Def{\uple{b}_{0}}{\oball{\uple{\gamma}_{0}}{\uple{x}_{\alpha_{0}+1}}}$, where $\Def{\uple{\gamma}_{0}}{\val(\uple{x}_{\alpha_{0}+1} - \uple{x}_{\alpha_{0}})}$ --- it follows immediately that $\ltf(u(\uple{x}))\in\ltf(\K(C))$ --- and for any $\uple{a}\in\uple{b}_{0}$, $u(\upleto{a}{x}{n})$ has a Weierstrass preparation on $b_{n,0}$.
\end{proposition}

\begin{proof} We may assume that $\alg{\K(M)} = \K(M)$. The proof proceeds by induction on $J$. Suppose that Proposition\,\ref{prop:descr term imm} holds for any term $(v,K,p,f)$ such that $K < J$. Let us prove a few claims to take care of certain induction steps.

\begin{claim}[ind step unit]
Fix $e$ and $n\in\Nn$. Suppose the lemma holds for all $(u,J,n,e)$, then it holds for any $(u,J,n,e+1)$ where $u$ is a monic polynomial in $x_{n}$.
\end{claim}

Note that the case $e=0$ does not require any hypothesis (other than the induction hypothesis on $J$).

\begin{proof}
Let $u = x_{n}^{e+1} + \sum_{i\leq e} u_{i}(\upleneq{x}{n})x_{n}^{i}$ and for all $f\leq e$, $\Def{s_{f}}{\sum_{i\leq f} u_{i}(\upleneq{x}{n})x_{n}^{i}}$. Let $\uple{a}$ be a pseudo-limit of $(\uple{x}_{\alpha})$. Then we can find $\alpha_{0}$ such that, for all $j\neq n$, $\val(s_{f}(\upleto{a}{x}{j}))$ and $\val(u_{f+1}(\upleto{a}{x}{j}))$ are constant on $b_{j,0}$ and $u_{f+1}(\upleto{a}{x}{j})$ has a Weierstrass preparation. By induction on $f$ and invariance under addition, $s_{e}(\upleto{a}{x}{j})$ has a Weierstrass preparation on $b_{j,0}$. Let $\Def{s}{s_{e}}$. Making $\alpha_{0}$ bigger we can also assume that $\val(\diff[j]{s}{\upleto{a}{x}{j}})$ is constant on $b_{j,0}$. By Proposition\,\ref{prop:lin approx}, we find $c_{j}\in\K(C)$ such that for all $y_{j}$ and $z_{j}\in b_{j,0}$, $\ltf(u(\upleto{a}{y}{j}) - u(\upleto{a}{z}{j})) = \ltf(s(\upleto{a}{y}{j}) - s(\upleto{a}{z}{j})) = \ltf(c_{j})\ltf(y_{j}-z_{j})$. By field quantifier elimination, this statement only depends on the value of $\ltf(v(\upleneq{a}{j}))$ for a finite number of $\Sortrestr{\LAQ}{\K}(C)$-terms $v$ and hence, by induction, making $\alpha_{0}$ bigger, we may assume that this statement is true of all $\uple{a}\in \uple{b}_{0}$. When $j = n$, the same arguments yields some $c_{n}\in\K(C)$ as $u$ is already polynomial in $x_{n}$ and $\diff[n]{u}{\uple{x}}$ is polynomial in $x_{n}$ of degree at most $e$. It now follows from Proposition\,\ref{prop:one var to all} that $\uple{c}$ linearly approximates $u$ on $\uple{b}_{0}$.

Let $(\uple{y}_{\alpha})\in\cF$ be such that for any pseudo-limit $\uple{a}$, $\val(\uple{c}\cdot(\uple{a}-\uple{y}_{\alpha})) = \min_{j}\{\val(c_{j}) + \val(a_{j} - y_{j,\alpha})\}$. Then $\val(u(\uple{a}) - u(\uple{y}_{\alpha})) = \min_{j}\{\val(c_{j}) + \val(a_{j} - y_{j,\alpha})\}$. If for all $\alpha$, $\val(u(\uple{a})) > \min_{j}\{\val(c_{j}) + \val(a_{j} - y_{j,\alpha})\}$, then $\val(u(\uple{y}_{\alpha})) = \min_{j}\{\val(c_{j}) + \val(a_{j} - y_{j,\alpha})\}$ and $u(\uple{y}_{\alpha})\pc 0$ contradicting the minimality of $t$. Hence $\val(u(\uple{a})) < \min_{j}\{\val(c_{j}) + \val(a_{j} - y_{j,\alpha})\}$ for $\alpha\gg 0$ and $\ltf(u(\uple{a})) = \ltf(u(\uple{y_{\alpha}}))\in\ltf(\K(C))$. By compactness, making $\alpha_{0}$ bigger, this is true for any $\uple{a}\in\uple{b}_{0}$.
\end{proof}

\begin{claim}[ind step non unit]
Fix $e$ and $n\in\Nn$. Suppose the lemma holds for $(u,J,n,e)$ monic polynomial in $x_{n}$, then it holds for any $(u,J,n,e)$.
\end{claim}

\begin{proof}
Dividing by the dominant coefficient $u_{e}$ (which has constant $\ltf$ on $\uple{b}_{0}$ by induction), we obtain a term $v$ monic polynomial of degree at most $e$ in $x_{n}$ and which must also have constant $\ltf$ on $\uple{b}_{0}$ if we take $\alpha_{0}$ big enough.
\end{proof}

\begin{claim}[ind step analytic]
Fix $n\in\Nn$. Suppose that for all $e\in\Nn$, the lemma holds for all $(u,J,n,e)$. Then it also holds for all $(u,J,n,\infty)$.
\end{claim}

\begin{proof}
Let $\uple{a}$ be a pseudo-limit of $(x_{\alpha})$. Any $S\in\SWC<\ValR>(\genan[C]{\upleneq{a}{n}})$ that contains $a_{n}$ must contain $b_{n,0}$ for $\alpha_{0}$ big enough. If not, there exists $c\in\alg{\K(\genan[C]{\upleneq{a}{n}})}$ such that $x_{n,\alpha}\pc c$. Let $P(\upleto{a}{x}{n}) = \sum_{i}p_{i}(\upleneq{a}{n})x_{n}^{i}$ be its minimal polynomial. Then, by hypothesis, for all $\uple{e}\in\uple{b}_{0}$ (for $\alpha_{0}$ big enough), $\ltf(P(\uple{e})) = 0$ and we must have $p_{i}(\upleneq{a}{n}) = 0$ for all $i$, but that is absurd.

It follows that we can find $\alpha_{0}$ such that, $u(\upleto{a}{x}{n})$ has a Weierstrass preparation on $b_{n,0}$, i.e. there exists $F\in\Ann$, $\uple{c}\in\K(\genan[C]{\upleneq{a}{n}})$ and $\Sortrestr{\LAQ}{\K(C)}$-terms $P$ and $Q$ polynomial in $x_{n}$ such that for all $x_{n}\in b_{j,0}$, \[u(\upleto{a}{x}{n}) = F\left(\frac{x_{n}-x_{n,\alpha_{0}+1}}{x_{n,\alpha_{0}+1}-x_{n,\alpha_{0}}},\uple{c}\right)\frac{P(\upleto{a}{x}{n})}{Q(\upleto{a}{x}{n})}\] and $\val(F((x_{n}-x_{n,\alpha_{0}+1})(x_{n,\alpha_{0}+1}-x_{n,\alpha_{0}}),\uple{c})) = 0$. In turn, this implies that $\ltf(u(\upleto{a}{x}{n}))$ does not depend on $x_{n}\in b_{n,0}$ and by the usual uniformization argument (making $\alpha_{0}$ bigger), we can ensure that $\ltf(u(\uple{e}))$ does not depend on $\uple{e}\in\uple{b}_{0}$.
\end{proof}

Proposition\,\ref{prop:descr term imm} follows by induction.
\end{proof}

\begin{remark}[lin app imm]
Note that the proof of Claim\,\ref{claim:ind step unit} also shows that there exists $\uple{d}\in\K(C)$ that linearly approximates $t$ on $\uple{b}_{0}$ for $\alpha_{0}$ big enough.
\end{remark}

Let $M\models\TAs$ be sufficiently saturated and $C\substr M$ such that $\resf(\K(C))$ is linearly closed.

\begin{proposition}[lin cl rich]
Let $x_{\alpha}$ be a pseudo-convergent sequence of $C$. The family $\{\sprol(y_{\alpha})\mid y_{\alpha}$ is a pseudo-convergent sequence of $C$ equivalent to $x_{\alpha}\}$ is a rich enough family of equivalent pseudo-convergent sequences of $C$.
\end{proposition}

\begin{proof}
Let $P(\uple{X}) = \sum_{i}p_{i}X_{i}\in\K(C)[\uple{X}]$. If for all $i$ $p_{i} = 0$, we are done. Otherwise, let $\epsilon_{\alpha} = x_{\alpha+1}-x_{\alpha}$, let $i_{0}$ such that $\val(p_{i_{0}}) + \val(\sigma^{i_{0}}(\epsilon_{\alpha}))$ is minimal, and let \[Q_{\alpha}(\sprol(X)) = p_{i_{0}}^{-1}\sigma^{i_{0}}(\epsilon_{\alpha})^{-1}P(\sprol(\epsilon_{\alpha}X)) = \sum_{i}p_{i}p_{i_{0}}^{-1}\sigma^{i}(\epsilon_{\alpha})\sigma^{i_{0}}(\epsilon_{\alpha}^{-1}) \sigma^{i}(X).\] As $\resf(Q_{\alpha})$ is linear with coefficients in $\resf(\K(C))$, which is linearly closed, we can find $d_{\alpha}\in\K(C)$ such that $\resf(Q_{\alpha}(\sprol(d_{\alpha})))\neq\resf(Q_{\alpha}(\sprol(1)))$. In particular, $\resf(d_{\alpha})\neq\resf(1)$ and $\val(d_{\alpha}-1) = 0$. Let $y_{\alpha} = x_{\alpha} + \epsilon_{\alpha}d_{\alpha}$.

Let $\uple{a}$ be such that $\sprol(x_{\alpha})\pc \uple{a}$, then
\begin{eqnarray*}
\ltf(a_{i}-\sigma^{i}(y_{\alpha})) &=& \ltf(a_{i}-\sigma^{i}(x_{\alpha+1}) + \sigma^{i}(x_{\alpha+1}) - \sigma^{i}(x_{\alpha}) + \sigma^{i}(x_{\alpha}) - \sigma^{i}(y_{\alpha}))\\
&=& \ltf(\sigma^{i}(\epsilon_{\alpha}))\ltf(1 - \sigma^{i}(d_{\alpha})).
\end{eqnarray*}
It follows that $\val(a_{0} - y_{\alpha}) = \val(\epsilon_{\alpha}) = \val(a_{0} - x_{\alpha})$. By Lemma\,\ref{lem:equiv pc seq}, $(y_{\alpha})$ is equivalent to $(x_{\alpha})$. Let $c_{i} = (a_{i} - \sigma^{i}(y_{\alpha}))/\sigma^{i}(\epsilon_{\alpha})$. Then

\begin{eqnarray*}
\resf(P(\uple{a}-\sprol(y_{\alpha}))p_{i_{0}}^{-1}\epsilon_{\alpha}^{-1})&=& \resf(Q)(\resf(\uple{c}))\\
&=& \resf(Q)(\resf(\sprol(1) - \sprol(d_{\alpha})))\\
&=& \resf(Q(\sprol(1))) - \resf(Q(\sprol(d_{\alpha})))\\
&\neq& 0.
\end{eqnarray*}
Hence, we have $\val(P(\uple{a}-\sprol(y_{\alpha}))) = \val(p_{i_{0}}) + \val(\sigma^{i_{0}}(\epsilon_{\alpha})) = \min_{i}\{\val(p_{i}) + \val(a_{i}-\sigma^{i}(y_{\alpha}))\}$.
\end{proof}

And now let us prove another embedding theorem for immediate extensions. Let $M_{1}$ and $M_{2}\models\TAsH$ be sufficiently saturated, $N_{i}\substr M_{i}$ have no immediate extension in $M_{i}$ and be $\sigma$-Henselian --- as we will see in Remark\,\ref{rem:max imm imp sH} this second hypothesis follows from the first one ---, $C_{i}\substr N_{i}$ be such that $\resf(\K(C_{1}))$ is linearly closed and $f:C_{1}\to C_{2}$ an $\Morl{LAQs}[\lt]$-isomorphism.

\begin{definition}(Minimal term of a pseudo-convergent sequence)
Let $(x_{\alpha})$ be a pseudo-convergent sequence of $C_{1}$. We say that $(t,I,m,d)\in\Term[0]{C_{1}}$ is its minimal term if it is minimal such that it is monic polynomial in $x_{n}$ and it is $\sigma$-pseudo-solved by a pseudo-convergent sequence equivalent to $(x_{\alpha})$.
\end{definition}

Note that any pseudo-convergent sequence has a minimal term, as any pseudo-convergent sequence $\sigma$-pseudo-solves $0$.

\begin{proposition}[imm ext]
Let $(x_{\alpha})$ be a pseudo-convergent sequence of $\K(C_{1})$ (indexed by a limit ordinal) which is eventually in $\ValR$. Let $(t,I,m,d)$ be its minimal term. Then:
\begin{thm@enum}
\item There exists $a_{1}\in N_{1}$ and $a_{2}\in N_{2}$ such that $x_{\alpha}\pc a_{1}$, $f(x_{\alpha})\pc a_{2}$ and $t(\sprol(a_{1}))=0 = t^{f}(\sprol(a_{2}))$.
\item For any such $a_{1}$, $\genans[C_{1}]{a_{1}}$ is an immediate extension of $C_{1}$;
\item For any such $a_{i}$, $f$ can be extended to an $\Morl{LAQs}[\lt]$-isomorphism sending $a_{1}$ to $a_{2}$.
\end{thm@enum}
\end{proposition}

\begin{proof}
If $t$ is zero, it suffices to choose any $a_{1}$ and $a_{2}$ such that $x_{\alpha}\pc a_{1}$ and $f(x_{\alpha})\pc a_{2}$. These exist in $M_{i}$ and we will see in the end why they exist in $N_{i}$. Let us now assume that $t$ is not zero. By Remark\,\ref{rem:lin app imm} --- and Propositions\,\ref{prop:descr term imm} and \ref{prop:lin cl rich} --- we find $\alpha_{0}$ and $\uple{d}\in\K(C_{1})$ that linearly approximate $t$ at prolongations on $\Def{b_{0}}{\oball{\val(x_{\alpha_{0}+1} - x_{\alpha_{0}})}{x_{\alpha_{0}+1}}}$. By Proposition\,\ref{prop:lin cl rich}, we can find a pseudo-convergent sequence $(z_{\alpha})$ of $C_{1}$ equivalent to $(x_{\alpha})$ such that for all pseudo-limit $a$ of $(x_{\alpha})$, $\val(t(\sprol(a)) - t(\sprol(z_{\alpha}))) = \min_{i}\{\val(d_{i}) + \val(\sigma^{i}(a - z_{\alpha}))\}$. If for all such $a$, $\val(t(\sprol(a))) <  \min_{i}\{\val(d_{i}) + \val(\sigma^{i}(a - z_{\alpha}))\}$ for $\alpha$ big enough, then $\val(t(\sprol(a))) = \val(t(\sprol(y_{\alpha})))$. By compactness, $\val(t(\sprol(x)))$ is constant on some $b_{0}$. But this contradicts the fact that we can find $y_{\alpha}$ equivalent to $x_{\alpha}$ that $\sigma$-pseudo-solves $t$.

Hence there exists a pseudo-limit $a$ such that $\val(t(\sprol(a))) >  \min_{i}\{\val(d_{i}) + \val(\sigma^{i}(a - z_{\alpha}))\} \geq \min_{i}\{\val(d_{i}) + \val(\sigma^{i}(\xi_{0}))\}$ where $\xi_{0}$ is the radius of $b_{0}$ and $(t,a,\uple{d},\xi_{0})$ is in $\sigma$-Hensel configuration. As $N_{1}$ is $\sigma$-Henselian, we can find $a_{1}\in\K(N_{1})$ such that $t(a_{1}) = 0$ and $\val(a_{1} - a) \geq \max_{i}\{\val(\sigma^{-i}(t(\sprol(a))d_{i}^{-1}))\} > \val(x_{\alpha+1}-x_{\alpha})$, i.e. $x_{\alpha}\pc a_{1}$. As $f$ is an $\LAQs$-isomorphism, $(t^{f},I,m,d)$ is the minimal term of $(f(x_{\alpha}))$ and the same argument shows that there is $a_{2}\in \K(N_{2})$ such that $t^{f}(a_{2}) = 0$ and $f(x_{\alpha})\pc a_{2}$.

If $t\neq 0$, let us now show that $x_{m}$ must be the last variable appearing in $t$. If it is not, let $x_{n}$ be that last variable. By Proposition\,\ref{prop:descr term imm}, we can find $\Sortrestr{\LAQ}{\K}(A)$-terms $E$, $P$ and $Q$ such that $E$ is a strong unit in $x_{n}$, $P$ and $Q$ are polynomial in $x_{n}$ and for all $\uple{a}\in\uple{b}_{0}$, $t(\uple{a}) = E(\uple{a})P(\uple{a})/Q(\uple{a})$. As $t(\sprol(a_{1})) = 0$ we also have $P(\sprol(a_{1})) = 0$ and because $(P,I,n,\infty) < (t,I,m,d)$, we have $\ltf(P(\uple{a})) = 0$ --- and hence $\ltf(t(\uple{a})) = 0$ --- for all $\uple{a}\in\uple{b}_{0}$, contradicting the fact that we can find $y_{\alpha}$ equivalent to $x_{\alpha}$ that $\sigma$-pseudo-solves $t$.

We can now conclude as in Proposition\,\ref{prop:lt ext} by extending $f$ to $\Def{C_{1,n}}{\genan[C_{1}]{\sprolleq{a_{1}}{n}}}$ progressively, by sending $\sigma^{n}(a_{1})$ to $\sigma^{n}(a_{2})$. For $n < m$, it is exactly the same and for $n \geq m$, use the fact that $\sigma^{n}(a_{1})$ is the only zero of $P^{\sigma^{n-m}}(X)$ in $(b_{n,0})$ for $\alpha_{0}\gg 0$, where $P(X_{m}) = t(\sprolto{a_{1}}{X}{m})$.

If $n < m$, we have proved in Proposition\,\ref{prop:descr term imm} that the extension is immediate. If $n \geq m$, we have just seen that $\sigma^{n}(a_{1})$ is $\ACVF$-definable over $\K(\genan[C_{1}]{\sprolleq{a_{1}}{n-1}})$. It follows that $\sigma^{n}(a_{1})\in\K(\genan[C_{1}]{\sprolleq{a_{1}}{n-1}})^{h}$ which is an immediate extension of $\K(\genan[C_{1}]{\sprolleq{a_{1}}{n-1}})$ and we conclude by Proposition\,\ref{prop:TA alg ext}.

In the case where $t$ is zero, we have yet to show that we can take $a_{i}\in N_{i}$. Let $(u,J,n,e)$ be minimal over $N_{1}$ that is $\sigma$-pseudo-solved by a pseudo-converging sequence equivalent to $x_{\alpha}$. We can find $a_{1}$ in $M_{1}$ such that $u(\sprol(a_{1})) = 0$ and $x_{\alpha}\pc a_{1}$. But then $\K(\genans[N_{1}]{a_{1}})$ is an immediate extension of $N_{1}$ and we must have $a_{1}\in N_{1}$.
\end{proof}

\begin{remark}[max imm imp sH]
Note that we have just shown that if we only assume that $N_{1}$ has no immediate extension in $M_{1}$ (and not that $N_{1}$ is $\sigma$-Henselian), then $N_{1}$ is maximally complete and hence, by Proposition\,\ref{prop:max compl sigmaH} it is $\sigma$-Henselian.
\end{remark}

\begin{definition}(Minimal term of a point)
Let $a \in M_{1}$. We say that $(t,I,m,d)\in\Term[0]{C_{1}}$ is the minimal term of $a$ over $C_{1}$ if it is minimal such that it is monic polynomial in $x_{m}$ and $t(\sprol(a)) = 0$.
\end{definition}

Note that because of Weierstrass preparation, minimal terms will always be polynomial in their last variable.

\begin{definition}($(t,I,m,d)$-fullness)
Let $(t,I,m,d)\in\Term[0]{C_{1}}$. We will say that $C_{1}$ is $(t,I,m,d)$-full if for all pseudo-convergent sequences $(x_{\alpha})$ (indexed by a limit ordinal) of elements in $C_{1}$ that are eventually in $\ValR$ with minimal term $(u,J,n,e) < (t,I,m,d)$, $(x_{\alpha})$ has a pseudo-limit in $C_{1}$.
\end{definition}

\begin{corollary}[imm ext full]
Let $(x_{\alpha})$ be a maximal pseudo-convergent sequence in $C_{1}$ (indexed by a limit ordinal) pseudo-converging to some $a_{1}\in \ValR(M_{1})$. If $(t,I,m,d)\in\Term[0]{C_{1}}$ is its minimal term over $C_{1}$ and  $C_{1}$ is $(t,I,m,d)$-full, then $\K(\genans[C_{1}]{a_{1}})$ is an immediate extension of $\K(C_{1})$ and $f$ extends to a morphism from $\genans[C_{1}]{a_{1}}$ into $N_{2}$.
\end{corollary}

\begin{proof}
Since $C_{1}$ is $(t,I,m,d)$-full, $(x_{\alpha})$ (or any equivalent pseudo-convergent sequence) cannot pseudo-solve a term of order-degree strictly less than $(t,I,m,d)$ (this would contradict either $(t,I,m,d)$-fullness of $C_{1}$ or maximality of $(x_{\alpha})$). By Propositions\,\ref{prop:descr term imm} and \ref{prop:lin cl rich}, there is a tuple $\uple{d}$ and a sequence $(y_{\alpha})$ equivalent to $(x_{\alpha})$ such that \[\val(t(\sprol(y_{\alpha}))) = \val(t(\sprol(a)) - t(\sprol(y_{\alpha}))) = \min_{i}\{\val(d_{i}) + \val(\sigma^{i}(a-y_{\alpha}))\},\] i.e. $t(\sprol(y_{\alpha}))\pc 0$. We have just showed that $t$ is the minimal term of the pseudo-convergent sequence $(x_{\alpha})$ and thus we can now apply Proposition\,\ref{prop:imm ext}.
\end{proof}

From now on, suppose that $\K(N_{i})$ is an immediate extension of $\K(C_{i})$, hence it is a maximal immediate extension of $\K(C_{i})$ in $M_{i}$.

\begin{corollary}[min imp full]
Suppose that all $a\in \ValR(N_{1})$ with a minimal term of order-degree strictly smaller than $(t,I,m,d)$ are already in $C_{1}$, then $C_{1}$ is $(t,I,m,d)$-full.
\end{corollary}

\begin{proof}
Let $(x_{\alpha})$ be a pseudo-convergent sequence of $C_{1}$ (indexed by a limit ordinal) that is eventually in $\ValR$ and $(u,J,n,e) < (t,I,m,d)$ that is $\sigma$-pseudo-solved by $(x_{\alpha})$. We may assume that  $(u,J,n,e)$ is its minimal term. By Proposition\,\ref{prop:imm ext}, there is $a_{1}\in N_{1}$ such that $x_{\alpha}\pc a_{1}$ and $u(a_{1}) = 0$. Hence $a_{1}$ has a minimal polynomial of order-degree strictly lower than $(t,I,m,d)$, so $a_{1}\in C_{1}$ and $C_{1}$ is indeed $(t,I,m,d)$-full.
\end{proof}

\begin{corollary}[ext imm max]
The isomorphism $f$ extends to an isomorphism between $N_{1}$ and $N_{2}$, i.e. maximum immediate extensions (in some saturated model) --- and hence maximally complete extensions --- are unique up to isomorphism.
\end{corollary}

We could prove this corollary without using the notion of fullness and without doing the extensions in the right order --- just pick any maximal pseudo-convergent sequence indexed by a limit ordinal, find its minimal term and apply Proposition\,\ref{prop:imm ext} to extend $f$ some more and iterate. But the following proof provides a better description of the information needed to describe the type of a given point in an immediate extension.

\begin{proof}
Let us consider the extensions $C_{1} \substr L_{\alpha} \substr N_{1}$ defined by taking $L_{\alpha+1} = \genans[L_{\alpha}]{c_{\alpha}}$ where $c_{\alpha}\in \ValR(N_{1})\sminus L_{\alpha}$ has a minimal term of minimal order-degree over $L_{\alpha+1}$ and $L_{\lambda} = \bigcup_{\alpha<\lambda} L_{\alpha}$ for $\lambda$ limit. Then we can show by induction that we can extend $f$ to $L_{\alpha}$ in a coherent way.

Let us suppose that we have extended $f$ to $f_{\alpha}$ on $L_{\alpha}$. Let $a = c_{\alpha}$. Let $x_{\beta}\pc a$ be a maximal pseudo-converging sequence of $L_{\alpha}$. Then if $(t,I,m,d)$ is a minimal term of $a$, then by Corollary\,\ref{cor:min imp full}, $L_{\alpha}$ is $(t,I,m,d)$-full. Applying Corollary\,\ref{cor:imm ext full}, we obtain that $f_{\alpha}$ can be extended to $\genans[L_{\alpha}]{a} = L_{\alpha+1}$. The limit case is trivial.

As $N_{1}$ is the field generated by $\bigcup_{\alpha} L_{\alpha}$, by Remark\,\ref{rem:ext frac field ann} we can extend $f$ to a morphism from $N_{1}$ into $N_{2}$. Now if $f$ is not onto, pick $a\in \K(N_{2})\sminus\K(f(N_{1}))$, $(x_{\alpha})$ maximal pseudo-converging to $a$ and $(t,I,m,d)$ its minimal term. Then applying Proposition\,\ref{prop:imm ext} the other way round, we would find an immediate extension of $N_{1}$ in $M_{1}$, but that is absurd.
\end{proof}

\subsection{Relative quantifier elimination}\label{subsec:EQ}

\begin{Theorem}[EQ TAsH]
The theory $\TAsH$ eliminates quantifiers resplendently relatively to $\lt$.
\end{Theorem}

\begin{proof}
By Proposition\,\ref{prop:EQ implies respl}, it suffices to show that $\TAsH$ eliminates quantifiers relatively to $\lt$. Note that if two models of $\TAsH$ contain isomorphic substructures they have the same characteristic and residual characteristic, hence it also suffices to prove the result for $\TAsH[0,0]$ and $\TAsH[0,p]$. Let us first consider the equicharacteristic zero case.

It suffices to show that if $M_{1}$ and $M_{2}$ are sufficiently saturated models of $\Morl{TAsH}|0,0|[\lt]$, $f$ a partial $\Morl{LAQ}[\lt]$-isomorphism with (small) domain $C_{1}$, and $a_{1}\in K(M_{1})$, $f$ can be extended to $\genans[C_{1}]{a_{1}}$. Let $N_{1}\substr M_{1}$ with no immediate extension in $M_{1}$ and containing both $C_{1}$ and $a_{1}$. By Morleyization on $\lt$ and Lemma\,\ref{lem:ext on lt} we can extend $f$ to $D_{1}\substr N_{1}$ such that $\lt(D_{1}) = \lt(N_{1})$. Then applying Corollary\,\ref{cor:lt ext cor} repetitively we can extend $f$ to $E_{1}\substr N_{1}$ such that $\ltf(\K(E_{1})) = \lt(E_{1})$. Now $\K(N_{1})$ is a maximal immediate extension of $\K(E_{1})$ and we can extend $f$ to $N_{1}$ by Proposition\,\ref{prop:imm ext}.

Now that we know the equicharacteristic zero case, the mixed characteristic case follows from Propositions\,\ref{prop:EQ transfer} and \ref{prop:Coarsen TAsH}.
\end{proof}

We also obtain the corresponding results when there are angular components. Let $\LAsac$ be $\LAQac$ enriched with a symbol $\sigma:\K\to\K$, symbols $\sigma^{n}:\res<n>\to\res<n>$ and a symbol $\sigma_{\Valgp}:\Valgpinf\to\Valgpinf$. Let $\TAsHac$ be the $\LAsac$-theory of $\sigma$-Henselian analytic difference valued fields with a linearly closed residue field and angular components that are compatible with $\sigma$, i.e. $\ac_{n}\comp\sigma = \sigma_{n}\comp\ac_{n}$. Let $\LAsacfr$ be the enrichment of $\Lacfr$ with the same symbols and $\TAsHacfr{e}{p}$ be the theory of finitely ramified characteristic $(0,p)$ valued fields as above with ramification index at most $e$, i.e. $e\cdot 1 \geq \val(p)$.

\begin{corollary}[EQ TAsHac]
$\TAsHac$ and $\TAsHacfr{e}{p}$ for all $p$ and $e$, eliminate $\K$-quantifiers resplendently.
\end{corollary}

\begin{proof}
By Proposition\,\ref{prop:EQ implies respl}, resplendence comes for free once we have $\K$-quantifier elimination. Moreover, by Propositions\,\ref{prop:equiv Lsec Lac} and \ref{prop:EQ transfer}, we can transfer quantifier elimination in an $\lt$-enrichment of $\TAsH$ (cf. Theorem\,\ref{thm:EQ TAsH}) to quantifier elimination in a definable $\res\cup\Valgp$-enrichment of $\TAsHac$ and hence $\K$-quantifier elimination in $\TAsHac$.

The proof for $\TAsHacfr{e}{p}$ now follows by Remark\,\ref{rem:mixed char ac}.
\end{proof}

\begin{remark}
\begin{thm@enum}
\item In a valued field with an isometry and $\val(\fix(\K)) = \val(\K)$, angular components that are compatible with $\sigma$ are determined by their restriction to the fixed field. Indeed if $\val(x) = \val(\epsilon)$ where $\epsilon\in\fix(\K)$, then $\ac_{n}(x) = \res[n](x\epsilon^{-1})\ac_{n}(\epsilon)$. In fact, any angular components on the fixed field can be extended using this formula to angular components on the whole field that are compatible with $\sigma$ and hence any valued field with an isometry and $\val(\fix(\K)) = \val(\K)$ can be elementarily embedded into a valued field with an isometry and compatible angular components.
\item In fact, the existence of angular components in a $\sigma$-Henselian valued field with an isometry implies that $\val(\fix(\K)) = \val(\K)$.
\end{thm@enum}
\end{remark}

Until the end of this section, we will add constants to $\LAsac$ and $\LAsacfr$ for $\ac_{n}(t)$ and $\val(t)$ for every $\Sortrestr{\LAQs}{\K}$-term $t$ without any free variables. The reason for which we need to add theses constants is that although these are $\LAsac$-terms, we may have no trace of them in $\Sortrestr{\LAsac}{\res}$ and $\Sortrestr{\LAsac}{\Valgp}$. Ax-Kochen-Eršov type results now follow by the usual arguments.

\begin{corollary}(Ax-Kochen-Eršov principle for analytic difference valued fields)
\begin{thm@enum}
\item Let $\LL$ be an $\res$-extension of a $\Valgp$-extension of $\LAsac$, $T$ an $\LL$-theory containing $\TAsHac[0,0]$ and $M$ and $N\models T$ then:
\begin{thm@enum}
\item $M\equiv N$ if and only if $\res[1](M)\equiv\res[1](N)$ as $\Sortrestr{\LL}{\res[1]}$-structures and $\Valgpinf(M) \equiv \Valgpinf(N)$ as $\Sortrestr{\LL}{\Valgpinf}$-structures;
\item Suppose $M \substr N$ then $M\subsel N$ if and only if $\res[1](M)\subsel\res[1](M)$ as $\Sortrestr{\LL}{\res[1]}$-structures and $\Valgpinf(M) \subsel \Valgpinf(N)$ as $\Sortrestr{\LL}{\Valgpinf}$-structures.
\end{thm@enum}
\item Let $\LL$ be an $\res$-extension of a $\Valgp$-extension of $\LAsac$, $T$ an $\LL$-theory containing $\TAsHacfr{e}{p}$ and $M$ and $N\models T$ then:
\begin{thm@enum}
\item $M\equiv N$ if and only if $\res(M)\equiv\res(N)$ as $\Sortrestr{\LL}{\res}$-structures and $\Valgpinf(M) \equiv \Valgpinf(N)$ as $\Sortrestr{\LL}{\Valgpinf}$-structures;
\item Suppose $M\substr N$ then $M\subsel N$ if and only if $\res(M)\subsel\res(N)$ as $\Sortrestr{\LL}{\res}$-structures and $\Valgpinf(M) \subsel\Valgpinf(N)$ as $\Sortrestr{\LL}{\Valgpinf}$-structures.
\end{thm@enum}
\end{thm@enum}
\end{corollary}

\begin{remark}
\begin{thm@enum}
\item In mixed characteristic with finite ramification, if $\ValR = \Val$, we have better results. Indeed, the trace of any unit $E$ on any $\lt[k]$ is given by the trace of a polynomial (which depends only on $E$ and not on its interpretation) and the $E_{k}$ are in fact useless. Hence the $\res[n]$ are pure rings with an automorphism. If there is no ramification (i.e. $e=1$), the $\res[n]$ are ring schemes over $\res[1]$ (the Witt vectors of length $n$) --- the ring scheme structure does not depend on the actual model we are looking at, contrary to the general finite ramification case --- and the automorphism on $\res[n]$ can be defined using the automorphism on $\res[1]$, hence $\res$ is definable in $\res[1]$. Finally if $\sigma$ is a lifting of the Frobenius, $\sigma_{0}$ is definable in the ring structure of $\res[1]$. It follows that we obtain Ax-Kochen-Eršov results looking only at $\res[1]$ as a ring and $\Valgpinf$ as an ordered abelian group (after adding some constants).
\item The fact that the $E_{k}$ are useless is also true in equicharacteristic zero when $\ValR = \Val$.
\item It also follows that in equicharacteristic zero or mixed characteristic with finite ramification (with angular component), $\res$ and $\Valgpinf$ are stably embedded and have pure $\Sortrestr{\LL}{\res}$-structure (resp. $\Sortrestr{\LL}{\Valgpinf}$-structure) where $\LL$ is either $\LAsac$ or $\LAsacfr$. In particular it will make sense to speak of the theory induced on $\res$ or $\Valgpinf$.
\end{thm@enum}
\end{remark}

\begin{proposition}[Th Witt]
Let $\LL$ be the language $\LAQs$ enriched with predicates $P_{n}$ on $\lt[1]$ interpreted as $n|\vallt[1](x)$. The $\LL$-theory of $\WittLAQs$ is axiomatized by $\TAsH$ and $\sigma_{1}$ is the Frobenius, the induced theory on $\res[1]$ is $\ACF_{p}$, $p$ has minimal positive valuation, $\Valgp$ is a $\Zz$-group and $\sigma_{\Valgp}$ is the identity. Moreover $\res[1]$ is a pure algebraically closed valued field and $\Valgp$ is a pure $\Zz$-group and they are stably embedded.
\end{proposition}

\begin{proof}
Any model of that theory has definable angular components compatible with $\sigma$. And these angular components extend the usual ones on the field of constants $\Wittf(\alg{\Ff_{p}})$. Hence the only constants we add are for elements of $\alg{\Ff_{p}}\subseteq\res[1]$ and $\Zz\subseteq\Valgp$. The proposition now follows from the discussion above (and the fact that $\ACF$ and $\Zz$-groups are model complete).
\end{proof}

\section{The \texorpdfstring{$\NIP$}{NIP} property in analytic difference valued fields}\label{sec:NIP}

Let us first recall what is shown by Delon and Bélair in the algebraic case \cite{DelonNIP,BelairNIP}. Let $\THenac$ be the $\Lac$-theory of Henselian valued fields with angular component maps.

\begin{theorem}[NIP alg]
Let $\LL$ be an $\res$-enrichment of a $\Valgpinf$-enrichment of $\Lac$ and $T\supseteq\THenac$ be an $\LL$-theory implying either equicharacteristic zero or finite ramification in mixed characteristic. Then $T$ is $\NIP$ if and only if $\res$ (with its $\Sortrestr{\LL}{\res}$-structure) and $\Valgpinf$ (with its $\Sortrestr{\LL}{\Valgpinf}$ -structure) are $\NIP$.
\end{theorem}

\begin{proof}
See \cite[Théorème\,7.4]{BelairNIP}. The resplendence of the theorem is not stated there but the proof is exactly the same after enriching on $\res$ and $\Valgpinf$.
\end{proof}

This result can be extended first to analytic fields then to analytic fields with an automorphism.

\begin{corollary}[NIP ann]
Let $\LL$ be an $\res$-enrichment of a $\Valgpinf$-enrichment of $\LAQac$ and $T\supseteq\TAac$ be an $\LL$ theory implying either equicharacteristic zero or finite ramification in mixed characteristic. Then $T$ is $\NIP$ if and only if $\res$ (with its $\Sortrestr{\LL}{\res}$-structure) and $\Valgpinf$ (with its $\Sortrestr{\LL}{\Valgpinf}$ -structure) are $\NIP$.
\end{corollary}

\begin{proof}
Suppose $T$ is not $\NIP$. Then there is a formula $\phi(x,\uple{y})$ which has the independence property and where $\card{x} = 1$. Note that, since for any sort there is an $\emptyset$-definable function from $\K$ onto that sort, we may assume that $x$ and $\uple{y}$ are $\K$-variables. By Remark\,\ref{rem:red alg ac}, there is an $\LL\sminus(\Ann\cup\{\Q\})$-formula $\psi(x,\uple{z})$ and $\Sortrestr{\LAQ}{\K}$terms $\uple{u}(\uple{y})$ such that $\phi(x,\uple{y})$ is equivalent to a $\psi(x,\uple{u}(\uple{y}))$. But then $\psi$ would have the independence property too, contradicting Theorem\,\ref{thm:NIP alg}.
\end{proof}

\begin{corollary}[NIP AsH]
Let $\LL$ be an $\res$-enrichment of a $\Valgpinf$-enrichment of $\LAsac$ and $T\supseteq\TAsHac$ be an $\LL$ theory implying either equicharacteristic zero or finite ramification in mixed characteristic. Then $T$ is $\NIP$ if and only if $\res$ (with its $\Sortrestr{\LL}{\res}$-structure) and $\Valgpinf$ (with its $\Sortrestr{\LL}{\Valgpinf}$ -structure) are $\NIP$.
\end{corollary}

\begin{proof}
Suppose $T$ is not $\NIP$, then there is a formula $\phi(x,\uple{y})$ which has the independence property (where $x$ and the $\uple{y}$ are $\K$-variables). By Corollary\,\ref{cor:EQ TAsHac}, we may assume that $\phi$ is without $\K$-quantifiers, i.e. there is a $\K$-quantifier free $\LAsac\sminus\{\sigma\}$-formula $\psi(\uple{x},\uple{z})$ such that $\phi(x,\uple{y})$ is equivalent to $\psi(\sprol(x),\sprol(\uple{y}))$. But then $\psi$ would have the independence property too, contradicting Corollary\,\ref{cor:NIP ann}.
\end{proof}

\begin{remark}[NIP no ac]
In the isometry case with $\val(\fix(\K)) = \val(\K)$, this last result also holds without angular components because any such valued field can be elementarily embedded into a valued field with angular components compatible with $\sigma$.
\end{remark}

\begin{corollary}[NIP Witt]
The $\LAQs$-theory of $\WittLAQs$ is $\NIP$.
\end{corollary}

\begin{proof}
This is an immediate corollary of Remark\,\ref{rem:NIP no ac}, Corollary\,\ref{cor:NIP AsH} and the fact that $\res$ is definable in $\res[1]$ which is a pure algebraically closed field (where the Frobenius automorphism is definable) and that $\Valgp$ is a pure $\Zz$-group (see Proposition\,\ref{prop:Th Witt}).
\end{proof}

\appendix
\part*{Appendices}

\section{Resplendent relative quantifier elimination}\label{sec:resplEQ}

The following section, although it may appear fastidious and nitpicking, is actually an attempt at clarifying some notions and properties that are often assumed to be clear when studying model theory of valued fields, but may actually need precise and careful presentation. In all this section, $\LL$ will denote a language and $\Sigma$, $\Pi$ a partition of its sorts.

\begin{definition}(Restriction)
If $\LL'\subseteq \LL$ be another language and $T$ an $\LL$-theory we will denote by $\Langrestr!{T}{\LL'}$ the $\LL'$-theory $\{\phi$ an $\LL'$-formula $\mid T\models\phi\}$ and if $C$ is an $\LL$-structure, $\Langrestr{C}{\LL'}$ will have underlying set $\bigcup_{S\in\LL'}S(C)$ with the obvious $\LL'$-structure. In particular, when $\Sigma$ is a set of $\LL$ sorts, let $\Sortrestr!{\LL}{\Sigma}$ be the restriction of $\LL$ to the predicate and function symbols that only concern the sorts in $\Sigma$. Then we will write $\Def{\Sortrestr{T}{\Sigma}}{\Langrestr{T}{\Sortrestr{\LL}{\Sigma}}}$ and $\Def{\Sortrestr{C}{\Sigma}}{\Langrestr{C}{\Sortrestr{\LL}{\Sigma}}}$.
\end{definition}

Note that the restriction is a functor from $\Str(T)$ to $\Str(\Langrestr{T}{\LL'})$ respecting models, cardinality and elementary submodels (see Section\,\ref{sec:cat} for the definitions).

\begin{definition}(Enrichment)
Let $\LL_{e}\supseteq\LL$ be a another language and $\Sigma_{e}$ the set of new $\LL_{e}$-sorts, i.e. the $\LL_{e}$-sorts that are not $\LL$-sorts. The language $\LL_{e}$ is said to be a $\Sigma$-enrichment of $\LL$ if $\LL_{e}\sminus\Sortrestr{\LL_{e}}{\Sigma\cup\Sigma_{e}} \subseteq \LL$, i.e. the enrichment is limited to the new sorts and the sorts in $\Sigma$. If, moreover, $\Sigma_{e} = \emptyset$ and $\LL_{e}\sminus\LL$ consists only of function symbols, we will say that $\LL_{e}$ is a $\Sigma$-term enrichment of $\LL$.

Let $T$ be an $\LL$-theory. An $\LL_{e}$-theory $T_{e}\supseteq T$ is said to be a definable enrichment of $T$ if there are no new sorts and for every predicate $P(\uple{x})$ (resp. function $f(\uple{x})$) symbol in $\LL_{e}\sminus\LL$, there is an $\LL$-formula $\phi_{P}(\uple{x})$ (resp. $\phi_{f}(\uple{x},y)$ such that $T\models\forall\uple{x}\exists^{=1}y,\,\phi_{f}(\uple{x},y)$) and that $T_{e} = T\cup\{P(\uple{x})\iffform\phi_{P}(\uple{x})\}\cup\{\phi_{f}(\uple{x},f(\uple{x}))\}$.
\end{definition}

\begin{definition}[Morl](Morleyization)
The Morleyization of $\LL$ on $\Sigma$ is the language $\Def{\Morl*!{\LL}[\Sigma]}{\LL\cup\{P_{\phi}(\uple{x})\mid\phi(\uple{x})$ an $\Sortrestr{\LL}{\Sigma}$-formula$\}}$. If $T$ is an $\LL$-theory, the Morleyization of $T$ on $\Sigma$ is the following $\Morl*{\LL}[\Sigma]$-theory $\Def{\Morl*{T}[\Sigma]}{T\cup\{P_{\phi}(\uple{x})\leftrightarrow\phi(\uple{x})\}}$ and if $M$ is an $\LL$-structure, $\Morl*{M}[\Sigma]$ is the $\Morl*{\LL}[\Sigma]$-structure with the same $\LL$-structure as $M$ and where $P_{\phi}$ is interpreted by $\phi(M)$.

On the other hand, we will say that an $\LL$-theory $T$ is Morleyized on $\Sigma$ if every $\Sortrestr{\LL}{\Sigma}$-formula is equivalent, modulo $T$, to a quantifier free $\Sortrestr{\LL}{\Sigma}$-formula.
\end{definition}

Note that $\Morl*{T}[\Sigma]$ is a definable $\Sigma$-enrichment of $T$ and if $M\models T$ then $\Morl*{M}[\Sigma]\models\Morl*{T}[\Sigma]$.

\begin{definition}(Elementary on $\Sigma$)
Let $M_{1}$ and $M_{2}$ be two $\LL$-structures. A partial isomorphism $M_{1}\to M_{2}$ is said to be $\Sigma$-elementary if it is a partial $\Morl*{\LL}[\Sigma]$-isomorphism.
\end{definition}

\begin{definition}[respl elim](Resplendent relative elimination of quantifiers)
Let $T$ be an $\LL$-theory. We say that $T$ eliminates quantifiers relatively to $\Sigma$ if $\Morl*{T}[\Sigma]$ eliminates quantifiers.

We say that $T$ eliminates quantifiers resplendently relatively to $\Sigma$ if for any $\Sigma$-enrichment $\LL_{e}$ of $\LL$ (with possibly new sorts $\Sigma_{e}$) and any $\LL_{e}$-theory $T_{e}\supseteq T$, $T_{e}$ eliminates quantifiers relatively to $\Sigma\cup\Sigma_{e}$.
\end{definition}

\begin{definition}(Resplendent elimination of quantifiers from a sort)
We will say that an $\LL$-theory $T$ eliminates $\Pi$-quantifiers if every $\LL$-formula is equivalent modulo $T$ to a formula where quantification only occurs on variables from the sorts in $\Sigma$.

We will say that $T$ eliminates $\Pi$-quantifiers resplendently if for any $\Sigma$-enrichment $\LL_{e}$ of $\LL$ and any $\LL_{e}$-theory $T_{e}\supseteq T$, $T_{e}$ eliminates $\Pi$-quantifiers.
\end{definition}

\begin{definition}[cl sorts](Closed sorts)
We will say that $\Sigma$ is closed if $\LL\sminus(\Sortrestr{\LL}{\Pi}\cup\Sortrestr{\LL}{\Sigma})$ only consists of function symbols $f : \prod_{i} P_{i}\to S$ where $P_{i}\in\Pi$ and $S\in\Sigma$. Equivalently, any predicate involving a sort in $\Sigma$ and any function with a domain involving a sort in $\Sigma$ only involves sorts in $\Sigma$.
\end{definition}

\begin{remark}
\begin{thm@enum}
\item Note that if, the sorts $\Sigma$ are closed then in any $\Sigma$-enrichment --- with possibly new sorts $\Sigma^{e}$ --- of a $\Pi$-enrichment of $\LL$ (or vice-versa), the sorts $\Sigma\cup\Sigma^{e}$ are still closed.
\item Elimination of quantifiers relative to $\Sigma$ implies elimination of $\Pi$-quantifiers. But the converse is in general not true. Indeed, if $\LL$ is a language with two sorts $S_{1}$ and $S_{2}$ and a predicate on $S_{1}\times S_{2}$, then the formula $\exists x R(x,y)$ is an $S_{2}$-quantifier free formula but there is no reason for it to be equivalent to any quantifier free $\Morl*{\LL}[S_{1}]$-formula.
\item However, if the sorts $\Sigma$ are closed, then it follows from Remark\,\ref{rem:decomp formula} that $T$ eliminates $\Pi$-quantifiers if and only if $T$ eliminates quantifiers relatively to $\Sigma$. If $\LL_{e}$ is a $\Sigma$-enrichment of $\LL$ with new sorts $\Sigma_{e}$, then $\Sigma\cup\Sigma_{e}$ is still closed, thus the equivalence is also true resplendently.
\end{thm@enum}
\end{remark}

We will now suppose that $\Sigma$ is \emph{closed} and we will denote by $\mathcal{F}$ the set of functions $f : \prod_{i} P_{i}\to S$ where $P_{i}\in\Pi$ and $S\in\Sigma$.

\begin{proposition}[EQ implies respl]
Let $T$ be an $\LL$-theory. If $T$ eliminates quantifiers relatively to $\Sigma$ then $T$ eliminates quantifiers resplendently relatively to $\Sigma$.
\end{proposition}

Let us begin with some remarks and lemmas that will have a more general interest.

\begin{remark}
\begin{thm@enum}
\item \label{rem:decomp formula} Any atomic $\LL$-formula $\phi(\uple{x},\uple{y})$ where $\uple{x}$ are $\Pi$-variables and $\uple{y}$ are $\Sigma$-variables, is either of the form $\psi(\uple{x})$ where $\psi$ is an atomic $\Sortrestr{\LL}{\Pi}$-formula or of the form $\psi(\uple{f}(\uple{u}(\uple{x})),\uple{y})$ where $\psi$ is an atomic $\Sortrestr{\LL}{\Sigma}$-formula, $\uple{u}$ are $\Sortrestr{\LL}{\Pi}$-terms and $\uple{f}$ are functions from $\mathcal{F}$.
\item If $T$ eliminates quantifiers relatively to $\Sigma$, it follows from Remark\,\ref{rem:decomp formula} above that for any $M\models T$, any $\LL(M)$-definable set in a product of sorts from $\Sigma$ is defined by a formula of the form $\phi(\uple{x},\uple{f}(\uple{a}),\uple{b})$ where $\phi$ is a $\Sortrestr{\LL}{\Sigma}$-formula. Hence $\Sigma$ is stably embedded in $T$, i.e. any $\LL(M)$-definable subset of $\Sigma$ is in fact $\LL(\Sigma(M))$-definable. Moreover, these sets are in fact $\Sortrestr{\LL}{\Sigma}(\Sigma(M))$-definable. In that case, we say that $\Sigma$ is a pure $\Sortrestr{\LL}{\Sigma}$-structure.
\end{thm@enum}
\end{remark}

\begin{lemma}[ext on lt]
Suppose $T$ is an $\LL$-theory Morleyized on $\Sigma$, then for any sufficiently saturated $M_{1}$, $M_{2}\models T$, any partial $\LL$-isomorphism $f : M_{1}\to M_{2}$ with small domain $C_{1}$ and any $c_{1}\in\Sigma(M_{1})$, $f$ can be extended to a partial $\LL$-isomorphism whose domain contains $c_{1}$.
\end{lemma}

\begin{proof}
First we may assume that $C_{1}\substr M_{1}$ and in particular for all $g\in\mathcal{F}$, $g(C_{1})\subseteq\Sigma(C_{1})$. Because $f$ is a partial $\LL$-isomorphism and $T$ is Morleyized on $\Sigma$, $\Sortrestr{f}{\Sigma}$ is a partial elementary $\Sortrestr{\LL}{\Sigma}$-isomorphism. By saturation of $M_{2}$ we can extend $\Sortrestr{f}{\Sigma}$ to $\Sortrestr{f'}{\Sigma} : \Sortrestr{M_{1}}{\Sigma} \to \Sortrestr{M_{2}}{\Sigma}$ a partial elementary $\Sortrestr{\LL}{\Sigma}$-isomorphism whose domain contains $c_{1}$. Let $f' = \Sortrestr{f}{\Pi}\cup \Sortrestr{f'}{\Sigma}$.

As $\Sortrestr{f}{\Pi}$ is a partial $\Sortrestr{\LL}{\Pi}$-isomorphism, $f'$ respects formulae $\phi(\uple{x})$ where $\phi$ is an atomic $\Sortrestr{\LL}{\Pi}$-formula ($\Sortrestr{f}{\Pi}$ also respects $\Sortrestr{\LL}{\Pi}$-terms). Moreover, as for all $g\in\mathcal{F}$, $\restr{f'}{g(C_{1})} = \restr{f}{g(C_{1})}$, $f'$ still respects $g$. As $\Sortrestr{f'}{\Sigma}$ is a partial $\Sortrestr{\LL}{\Sigma}$-isomorphism, it respects all atomic $\Sortrestr{\LL}{\Sigma}$-formulae. It follows that $f'$ also respects formulae of the form $\psi(\uple{g}(\uple{u}(\uple{x})),\uple{y})$ where $\psi$ is an atomic $\Sortrestr{\LL}{\Sigma}$-formula, $\uple{u}$ are $\Sortrestr{\LL}{\Pi}$-terms and $\uple{g}\in\mathcal{F}$. By Remark\,\ref{rem:decomp formula}, $f'$ respects all atomic $\LL$-formulae and hence is a partial $\LL$-isomorphism.
\end{proof}

\begin{definition}[gen](Generated structure)
Let $\LL$ be a language, $M$ an $\LL$-structure and $C\subseteq M$. The $\LL$-structure generated by $C$ will be denoted $\gen!{\LL}{C}$. If $C$ is an $\LL$-structure and $\uple{c}\in M$, the $\LL$-structure generated by $C$ and $\uple{c}$ will be denoted $\gen{\LL}[C]{\uple{c}}$.
\end{definition}

\begin{lemma}[field ext]
Let $M_{1}$, $M_{2}\models T$, $f : M_{1} \to M_{2}$ a partial $\LL$-isomorphism with domain $C_1\substr M_{1}$ and $c_{1}\in \Pi(M_{1})$ such that $\Sigma(\gen{\LL}[C_{1}]{c_{1}}) \subseteq \Sigma(C_{1})$. Suppose that $f'$ is a partial $\Sortrestr{\LL}{\Pi}\cup\mathcal{F}$-isomorphism extending $f$ whose domain is $\gen{\LL}[C_{1}]{c_{1}}$, then $f'$ is also a partial $\LL$-isomorphism.
\end{lemma}

\begin{proof}
First, by hypothesis, $f'$ respects atomic $\Sortrestr{\LL}{\Pi}$-formulae. Moreover as $\Sigma(\gen{\LL}[C_{1}]{c_{1}}) \subseteq \Sigma(C_{1})$, $\Sortrestr{f'}{\Sigma} = \Sortrestr{f}{\Sigma}$ and it is a partial $\Sortrestr{\LL}{\Sigma}$-isomorphism. As, by hypothesis, $f'$ respects $g\in\mathcal{F}$, it respects all formulae of the form $\psi(\uple{g}(\uple{u}(\uple{x})),\uple{y})$ where $\psi$ is an atomic $\Sortrestr{\LL}{\Sigma}$-formula, $\uple{u}$ are $\Sortrestr{\LL}{\Pi}$-terms and $g\in\mathcal{F}$. Hence by Remark\,\ref{rem:decomp formula}, $f'$ is a partial $\LL$-isomorphism.
\end{proof}

\begin{proof}[Proposition\,\ref{prop:EQ implies respl}]
We want to show that if $\LL_{e}$ is a $\Sigma$-enrichment of $\LL$ (with new sorts $\Sigma_{e}$) and $T_{e}\supseteq T$ an $\LL_{e}$-theory, then $\Morl*{T_{e}}[\Sigma\cup\Sigma_{e}]$ eliminates quantifiers. It suffices to show that for all $M_{1}$ and $M_{2}\models T_{e}$ that are $|\LL_{e}|^{+}$-saturated, for all partial $\Morl*{\LL_{e}}[\Sigma\cup\Sigma_{e}]$-isomorphism $f : M_{1}\to M_{2}$ of domain $C_{1}$ with $|C_{1}| \leq |\LL_{e}|$, and for all $c_{1}\in M_{1}$, $f$ can be extended to a partial $\Morl*{\LL_{e}}[\Sigma\cup\Sigma_{e}]$-isomorphism whose domain contains $c_{1}$.

Note first that $\Sigma\cup\Sigma_{e}$ is closed. If $c_{1}\in \Sigma\cup\Sigma_{e}(M_{1})$, then we can conclude by Lemma\,\ref{lem:ext on lt} (where $\LL$ is now $\Morl*{\LL_{e}}[\Sigma\cup\Sigma_{e}]$). If $c_{1}\in\Pi(M_{1})$, by repetitively applying Lemma\,\ref{lem:ext on lt}, we can extend $f$ to $f'$ whose domain contains all of $\Sigma\cup\Sigma_{e}(\gen{\LL_{e}}[C_{1}]{c_{1}})$. Then $f'$ is in particular an $\Morl*{\LL}[\Sigma]$-isomorphism and, as $T$ eliminates quantifiers relatively to $\Sigma$, $f'$ is in fact a partial elementary $\LL$-isomorphism that can be extended to a partial $\LL$-isomorphism $f''$ whose domain contain $c_{1}$. But, by Lemma\,\ref{lem:field ext}, $\restr{f''}{\gen{\LL_{e}}[C_{1}]{c_{1}}}$ is also a partial $\Morl*{\LL_{e}}[\Sigma\cup\Sigma_{e}]$-isomorphism.
\end{proof}

\section{Categories of structures}\label{sec:cat}

Recall that structures are always non empty.

\begin{definition}(${\Str!(T)}$)
Let $\LL$ be a language, $T$ an $\LL$-theory. We will denote by $\Str(T)$ the category whose objects are the $\LL$-structures that can be embedded in a model of $T$ --- i.e. models of $T_{\forall}$ --- and whose morphisms are the $\LL$-embeddings between those structures.

Moreover, let $T_{i}$ be an $\LL_{i}$-theory for $i=1,2$, $F : \Str(T_{1})\to\Str(T_{2})$ be a functor and $\kappa$ be a cardinal. We will denote by $\Str[F,\kappa](T_{2})$ the full subcategory of $\Str(T_{2})$ of structures that embed into some $F(M)$ for $M\models T_{1}$ $\kappa$-saturated. 
\end{definition}

A functor $F : \Str(T_{1})\to\Str(T_{2})$ is said to respect:
\begin{itemize}
\item models if for all $M\models T_{1}$, $F(M)\models T_{2}$;
\item $\kappa$-saturated models if for all $\kappa$-saturated $M\models T_{1}$, $F(M)\models T_{2}$;
\item cardinality if for all $C\models T_{1,\forall}$, $|F(C)| \leq |C|$;
\item cardinality up to $\kappa$ if for all $C\models T_{1,\forall}$, $|F(C)| \leq |C|^{\kappa}$;
\item elementary submodels if for all $M_{1}\subsel M_{2}\models T_{1}$, $F(M_{1})\subsel F(M_{2})$.
\end{itemize}

Let $\Sigma_{i}$ be a closed set of $\LL_{i}$-sorts for $i=1,2$. We say that $f : C_{1}\to C_{2}$ in $\Str(T_{1})$ is a $\Sigma_{1}$-extension if $C_{2}\sminus f(C_{1})\subseteq\Sigma_{1}(C_{2})$. We say that the functor $F$ sends $\Sigma_{1}$ to $\Sigma_{2}$ if for all $\Sigma_{1}$-extensions $C_{1}\to C_{2}$, $F(C_{1})\to F(C_{2})$ is a $\Sigma_{2}$-extension.

Let me recall some basic notions of category theory. A natural transformation $\alpha$ between functors $F$, $G : \cat_{1} \to \cat_{2}$ associates a morphism $\alpha_{c}\in\Hom[\cat_{2}]{F(c)}{G(c)}$ to every object $c\in \cat_{1}$ such that for all morphism $f\in\Hom[\cat_{1}]{c}{d}$, we have $G(f)\comp\alpha_{c} = \alpha_{d}\comp F(f)$. A natural transformation is said to be a natural isomorphism if for all $c\in \cat_{1}$, $\alpha_{c}$ is an isomorphism in $\cat_{2}$. It is easy to check that when $\alpha$ is a natural isomorphism, its inverse --- namely the transformation that associates $\alpha^{-1}_{c}$ to any $c\in \cat_{1}$ --- is also natural.

A pair of functors $F:\cat_{1}\to \cat_{2}$ and $G:\cat_{2}\to \cat_{1}$ are said to be an equivalence of categories between $\cat_{1}$ and $\cat_{2}$ if $GF$ and $FG$ are naturally isomorphic to the identity functor of resp. $\cat_{1}$ and $\cat_{2}$. We can always choose the natural isomorphisms $\alpha:FG\to\Id$ and $\beta:GF\to\Id$ such that $\alpha_{F} = F(\beta)$ and $\beta_{G} = G(\alpha)$ where $\alpha_{F} : c\mapsto \alpha_{F(c)}$ and $F(\alpha) : c\mapsto F(\alpha_{c})$.

Until the end of this section, let $\kappa$ be a cardinal, $T_{i}$ be an $\LL_{i}$-theory and $\Sigma_{i}$ be a set of closed $\LL_{i}$-sorts for $i=1,2$ and $\Full$ be a full subcategory of $\Str(T_{1})$ containing the $\kappa^{+}$-saturated models such that for any $C \to M_{1}\models T_{1}$ where $M_{1}$ is $\kappa^{+}$-saturated and $|C|\leq \kappa$, there is some $D$ in $\Full$ such that $C\to D \to M_{1}$ and $C\to D$ is a $\Sigma_{1}$-extension. Let $F : \Str(T_{1})\to\Str(T_{2})$ and $G : \Str(T_{2})\to \Str(T_{1})$ be functors that respect cardinality up to $\kappa$ and induce an equivalence of categories between $\Full$ and $\Str[F,\kappa^{+}](T_{2})$. We will also suppose that $G$ respects models and elementary submodels and sends $\Sigma_{2}$ to $\Sigma_{1}$ and $F$ respects $\kappa^{+}$-saturated models.

The goal of this section is to show that these (somewhat technical) requirements are a way to transfer elimination of quantifiers results from one theory to another and to give a meaning to --- and in fact extend --- the impression that if theories are quantifier free bi-definable (whatever that means) then elimination of quantifiers in one theory should imply elimination in the other. Proposition\,\ref{prop:EQ transfer} will be used, for example, to deduce valued field quantifiers elimination with angular components from valued field quantifiers elimination with sectioned leading terms. It will also be used to reduce the mixed characteristic case to the equicharacteristic zero case.

Proposition\,\ref{prop:iso are elem} is only used to prove Corollary\,\ref{cor:enrich functor} which in turn will be very useful to show that the functors between mixed characteristic and equicharacteristic zero can be modified to take in account Morleyization on $\lt$ while remaining in the right setting to transfer elimination of quantifiers.

\begin{proposition}[iso are elem]
Suppose $T_{1}$ is Morleyized on $\Sigma_{1}$ and let $M_{1}$ and $M_{2}\models T_{1}$ be $(|\LL_{2}|^\kappa)^{+}$-saturated. Then any partial $\LL_{2}$-isomorphism $f: F(M_{1})\to F(M_{2})$ is $\Sigma_{2}$-elementary.
\end{proposition}

\begin{proof}
To show that $f$ is $\Sigma_{2}$-elementary, it suffices to show that the restriction of $f$ to any finitely generated structure is $\Sigma_{2}$-elementary. To do so it suffices to show that the restriction of $f$ can be extended (on both its domain and its image) to any finitely generated $\Sigma_{2}$-extension. By symmetry, it suffices to prove the following property: if $D_{1}$, $D_{2} \substr F(M_{1})$ are such that $D_{1}\to D_{2}$ is a $\Sigma_{2}$-extension, $|D_{2}|\leq |\LL_{2}|$ and $f:D_{1}\to F(M_{2})$ is an $\LL_{2}$-embedding, then $f$ can be extended to some $g:D_{2}\to F(M_{2})$.

Applying $G$ to the initial data, we obtain the following diagram:
\[\xymatrix{GF(M_{1})&M_{2}&GF(M_{2})\ar[l]_(0.65){\beta_{M_{2}}}\\
G(D_{2})\ar[u]\ar@{.>}[ur]^{g}\\
G(D_{1})\ar[u]\ar[uurr]_{G(f)}}\]
where $g$ comes from the fact that, as $T$ is Morleyized on $\Sigma_{1}$, $\Sortrestr{\beta_{M_{2}}\comp G(f)}{\Sigma_{1}}$ is in fact elementary and, as $|G(D_{2})|\leq|\LL_{2}|^{\kappa}$, $M_{2}$ is $(|\LL_{2}|^{\kappa})^{+}$-saturated and $G(D_{1})\to G(D_{2})$ is a $\Sigma_{1}$-extension, by Lemma\,\ref{lem:ext on lt}, $\beta_{M_{2}}\comp G(f)$ can be extended to $g : G(D_{2}) \to M_{2}$. Applying $F$, we now obtain:
\[\xymatrix{D_{2}\ar[r]^{\alpha_{D_{2}}^{-1}}&FG(D_{2})\ar[r]^{F(g)}&F(M_{2})\\
D_{1}\ar[r]\ar[u]\ar@/_40pt/_{f}[rru]&FG(D_{1})\ar[u]\ar[ur]}\]
and $F(g)\comp\alpha^{-1}_{D_{2}}$ is the extension we were looking for.
\end{proof}

\begin{remark}
\begin{thm@enum}
\item One could hope the proposition to be true without the saturation hypothesis. But without some saturation, it is not even true that $M_{1}\subsel M_{2}$ implies $F(M_{1})\subsel F(M_{2})$. Take for example the coarsening functor $\Coar$ of Section\,\ref{sec:coar} and $\Qq_{p}\subsel M$ where $M$ is $\aleph_{0}$-saturated, then $\Coar(\Qq_{p})$ is trivially valued but $\Coar(M)$ is not.
\item One should beware that as $F(M_{1})$ and $F(M_{2})$ are not saturated, we have not proved that $T_{2}$ eliminates quantifiers.
\item We have proved nonetheless that, if $\Sigma_{i}$ is the set of all $\LL_{i}$ sorts (in that case we ask that $T_{2}$ eliminates all quantifiers) then for all $M_{1}$ and $M_{2}\models T_{1}$ sufficiently saturated, $M_{1}\equiv M_{2}$ implies $F(M_{1})\equiv F(M_{2})$.
\end{thm@enum}
\end{remark}

\begin{corollary}[enrich functor]
Let $T_{2}^e$ be a definable $\Sigma_{2}$-enrichment of $T_{2}$ (in the language $\LL_{2}^{e}$). Then $F$ induces a functor $F^{e}:\Str(T_{1})\to\Str(T_{2}^{e})$ and $G$ induces a functor $G^{e}:\Str(T_{2}^{e})\to \Str(T_{1})$. We can also find a full subcategory $\Full^{e}$ of $\Full$ such that $F^{e}$ and $G^{e}$ induce an equivalence of categories between $\Full^{e}$ and $\Str[F^{e},(|\LL_{2}|^{\kappa})^{+}](T_{2}^{e})$. The functor $G^{e}$ still respects cardinality up to $\kappa$, models and elementary submodels and sends $\Sigma_{2}$ to $\Sigma_{1}$ and $F^{e}$ respects cardinality up to $\kappa + |\LL_{2}|$ and $(|\LL_{2}|^{\kappa})^{+}$-saturated models. Finally, $\Full^{e}$ contains all $(|\LL_{2}|^{\kappa})^{+}$-saturated models and any $C$ in $\Str(T_{1})$ has a $\Sigma_{1}$-extension $D$ in $\Full^{e}$. Moreover, if $C\substr M_{1}\models T_{1}$ and $M_{1}$ is $(|\LL_{2}|^{\kappa})^{+}$-saturated, then we can find such a $D\substr M_{1}$.
\end{corollary}

\begin{proof}
Let $C\substr M \models T_{1}$. We can suppose that $M$ is $(|\LL_{2}|^\kappa)^{+}$-saturated. As $F(M)\models T_{2}$, we can enrich $F(M)$ to make it into an $\LL_{2}^{e}$-structure $F(M)^{e}\models T_{2}^{e}$ and we take $F^{e}(C) = \gen{\LL_{2}^{e}}{C}$. Note that if $M_{1}$ and $M_{2}$ are two $(|\LL_{2}|^\kappa)^{+}$-saturated models containing $C$, then Proposition\,\ref{prop:iso are elem} implies that $\id_{F(C)}$ is a partial isomorphism $F(M_{1})\to F(M_{2})$ $\Sigma_{2}$-elementary and hence the generated $\LL_{2}^{e}$-structures are $\LL_{2}^{e}$-isomorphic. As $F^{e}(C)$ does not depend (up to $\LL_{2}^{e}$-isomorphism) on the choice of $(|\LL_{2}|^\kappa)^{+}$-saturated model containing $C$, $F^{e}$ is well-defined on objects. If $f: C_{1}\to C_{2}$ is a morphism in $\Str(T_{1})$, by the same Proposition\,\ref{prop:iso are elem}, $F(f)$ is $\Sigma_{2}$-elementary and can be extended to a $\LL_{2}^{e}$-isomorphism on the $\LL_{2}^{e}$-structure generated by its domain. Note that if we denote by $i_{C}$ the embedding $F(C)\to F^{e}(C)$, we have also defined a natural transformation from $F$ to $F^{e}$ (a meticulous reader might want to add the forgetful functor $\Str(T_{2}^{e})\to\Str(T_{2})$ for it all to make sense).

We define $G^{e}$ to be $G$ (precomposed by the same forgetful functor). All the statements about $G^{e}$ follow immediately from those about $G$. As $\gen{\LL_{2}^{e}}{F(C)}$ has cardinality at most $|C|^{\kappa}|\LL_{2}|\leq |C|^{\kappa+|\LL_{2}|}$, $F$ respect cardinality up to $\kappa + \LL_{2}$ and if $M\models T_{1}$ is $(|\LL_{2}|^{\kappa})^{+}$-saturated then seeing it as a substructure of itself we obtain that $F^{e}(M)\models T_{2}^{e}$.

We define $\Full^{e}$ to be the full-subcategory of $\Full$ containing the $C$ such that $i_{C}$ is an isomorphism. In particular, it contains $(|\LL_{2}|^{\kappa})^{+}$-saturated models. Let $D$ be an $\LL_{2}^{e}$-substructure of $F^{e}(M)$ for some  $(|\LL_{2}|^{\kappa})^{+}$-saturated $M\models T_{1}$. Then $F^{e}G^{e}(D) = \gen{\LL_{2}^{e}}{FG(D)}$, where the generated structure is taken in $F(M)$. By Proposition\,\ref{prop:iso are elem}, the (natural) isomorphism $D \to FG(D)$ is $\Sigma_{2}$-elementary and can be extended (uniquely) into an $\LL_{2}^{e}$-isomorphism between $D = \gen{\LL_{2}^{e}}{D}$ and $F^{e}G^{e}(D)$. This new isomorphism is also natural. It follows that $FG(D) = F^{e}G^{e}(D)$ and that $i_{G(D)}$ is in fact an isomorphism, hence $G(D)\in\Full^{e}$.

If $C\in\Full^{e}$, $\beta_{C}\comp G(i_{C}^{-1}) : G^{e}F^{e}(C)\to C$ is a natural isomorphism. Finally, there remains to show that any $C\to M\models T_{1}$, where $M$ is $(|\LL_{2}|^{\kappa})^{+}$-saturated, can be embedded in some $E\in\Full^{e}$ such that $C\to E$ is a $\Sigma_{1}$-extension and $E\to M$. We already know that there exists $D\in\Full$ such that $C\to D\to M$ and $C\to D$ is a $\Sigma_{1}$-extension. Now $F(D)\to F^{e}(D)$ is a $\Sigma_{2}$-extension hence $D\isom GF(D)\to GF^{e}(D)$ is a $\Sigma_{1}$-extension. Moreover $GF^{e}(D)\to GF^{e}(M)\isom M$ and, as $F^{e}(D)$ is an $\LL_{2}^{e}$-structure of $F^{e}(M)$, $GF^{e}(D)\in \Full^{e}$. Thus we can take $E = GF^{e}(D)$.
\end{proof}

Let us now prove a second result in the spirit of Proposition\,\ref{prop:iso are elem}, but the other way round.

\begin{proposition}[EQ transfer]
If $T_{1}$ is Morleyized on $\Sigma_{1}$ and $T_{2}$ eliminates quantifiers, then $T_{1}$ eliminates quantifiers.
\end{proposition}

\begin{proof}
To show that $T_{1}$ eliminates quantifiers it suffices to show that for all $\kappa^{+}$-saturated $M_{i}\models T_{1}$, $i=1,2$, and $C_{1}\substr C_{2} \subs M_{1}$ and $f: C_{1}\to M_{2}$ an $\LL_{1}$-embedding, then $f$ can be extended to an embedding from $C_{2}$ into some elementary extension of $M_{2}$. Let $D_{1}\in\Full$ be such that $C_{1}\to D_{1}\to M_{1}$ and $C_{1}\to D_{1}$ is a $\Sigma_{1}$-extension. As $T_{1}$ is Morleyized on $\Sigma_{1}$, by Lemma\,\ref{lem:ext on lt}, we can extend $f$ to an embedding from $D_{1}$ into an elementary extension of $M_{2}$. Replacing $C_{1}$ by $D_{1}$, $C_{2}$ by $\gen{\LL_{1}}{D_{1}C_{2}}$ and $M_{2}$ by its elementary extension, we can consider that $C_{1}\in\Full$. Applying $F$, we obtain the following diagram:
\[\xymatrix{F(M_{1})&M_{2}^{\star}\\
F(C_{2})\ar[u]\ar@{.>}[ur]^{g}&F(M_{2})\ar[u]_{\subsel}\\
F(C_{1})\ar[u]\ar[ur]_{F(f)}}\]
where $M_{2}^{\star}$ is a $(|C_{1}|^{\aleph_{0}})^{+}$-saturated extension of $F(M_{2})$ and $g$ comes from quantifier elimination in $T_{2}$ and saturation of $M_{2}^{\star}$. Applying $G$ we obtain:
\[\xymatrix{C_{2}\ar[r]&GF(C_{2})\ar[r]_{G(g)}&G(M_{2}^{\star})&\\
C_{1}\ar[u]\ar[r]\ar@/_{20pt}/[rrr]_{f}&GF(C_{1})\ar[u]\ar[r]_{GF(f)}&GF(M_{2})\ar[u]_{\subsel}\ar@{<->}[r]&M_{2}\ar[ul]_{\subsel}}\]
and we have the required extension.
\end{proof}

\printsymbollist
\printbibli
\end{document}